\documentclass{amsart}

\usepackage{amsfonts, amssymb, amsmath, amsthm}                               
\usepackage{slashed}                                                          
\usepackage{graphicx}

\newtheorem{theorem}{Theorem} 	      	      	                              
\newtheorem{lemma}[theorem]{Lemma}     	       	      	      	      	      
\newtheorem{proposition}[theorem]{Proposition} 	      	      	      	      
\newtheorem*{question}{Question}                                              
\theoremstyle{definition}
\newtheorem{definition}[theorem]{Definition} 	      	      	                
\theoremstyle{remark}
\newtheorem*{remark}{Remark}                                                  
\newtheorem{rmk}{Remark}[theorem]

\numberwithin{equation}{section}                                              
\numberwithin{theorem}{section}                                               

\newcommand{\ol}[1]{\overline{#1}}                                            
\newcommand{\mc}[1]{\mathcal{#1}}                                             

\newcommand{\R}{\mathbb{R}}                                                   
\newcommand{\Sph}{\mathbb{S}}                                                 

\newcommand{\sGamma}{{\Gamma^\sharp}}
\newcommand{\snabla}{{\nabla^\sharp}}

\newcommand{\ud}{\mathrm{d}}
\newcommand{\rs}{r^\ast}

\newcommand{\g}{\mathring{\gamma}}
\newcommand{\gb}{\overline{g}}
\newcommand{\h}[2]{\bigl(\nabla^2 #1\bigr)_{#2}}
\newcommand{\Hb}[1]{\overline{\nabla}^2_{#1}}
\newcommand{\Nb}{\bar{N}}
\newcommand{\Tb}{\bar{T}}

\newcommand{\beq}{\begin{equation}}
\newcommand{\eeq}{\end{equation}}
\newcommand{\e}{\epsilon}

\newcommand{\sg}{{g^\sharp}}

\newcommand{\vv}{\tilde{v}}
\newcommand{\uu}{\tilde{u}}

\newcommand{\gi}{g}
\newcommand{\nablai}{\nabla}
\DeclareMathOperator{\gradi}{grad}                                            
\newcommand{\Boxi}{\Box}

\newcommand{\smooth}[1]{\mc{C}^\infty (#1)}                                   
\newcommand{\Ei}[1]{E_{#1}}                                                   

\allowdisplaybreaks

\begin{document}


\title[Unique continuation from infinity for linear waves]{Unique continuation from infinity\\ for linear waves}

\author{Spyros Alexakis}
\address{Department of Mathematics\\ University of Toronto\\ 40 St George Street Rm 6290\\ Toronto ON M5S 2E4\\ Canada}
\email{alexakis@math.utoronto.ca}

\author{Volker Schlue}
\address{Department of Mathematics\\ University of Toronto\\ 40 St George Street Rm 6290\\ Toronto ON M5S 2E4\\ Canada}
\curraddr{Mathematical Sciences Research Institute\\ 17 Gauss Way\\ Berkeley CA 94720\\ United States}
\email{vschlue@math.utoronto.ca}

\author{Arick Shao}
\address{Department of Mathematics\\ University of Toronto\\ 40 St George Street Rm 6290\\ Toronto ON M5S 2E4\\ Canada}
\email{ashao@math.utoronto.ca}


\begin{abstract}
We prove various uniqueness results from null infinity, for linear waves on asymptotically flat space-times.
Assuming vanishing of the solution to infinite order on suitable parts of future and past null infinities, we derive that the solution must vanish in an open set in the interior.
We find that the parts of infinity where we must impose a vanishing condition depend strongly on the background geometry. In particular, for backgrounds with positive mass (such as Schwarzschild or Kerr), the required assumptions are much weaker than the ones in the Minkowski space-time.
The results are nearly optimal in many respects.
They can be considered analogues of uniqueness from infinity results for second order elliptic operators.
This work is partly motivated by questions in general relativity. 
\end{abstract}

\maketitle

\tableofcontents

\section{Introduction} \label{sec:intro}

We prove unique continuation results from infinity for wave equations over asymptotically flat backgrounds.
In particular, consider solutions $\phi$ of a linear wave equation
\begin{equation} \label{wave} L_g \phi := \Box_g \phi + a^\alpha \partial_\alpha \phi + V \phi = 0 \end{equation}
over a Lorentzian manifold $(M, g)$, with $\Box_g$ being the Laplace-Beltrami operator for $g$.
We show that if the solution vanishes (in a suitable sense) at parts of future and past null infinities 
$\mc{I}^+, \mc{I}^- \subset \partial M$, then the solution $\phi$ must vanish on an open domain inside $(M, g)$.

\smallskip

One motivation for this paper comes from older and newer studies in general relativity 
regarding the possibility of periodic-in-time solutions of the Einstein equations.
This has been considered both for vacuum space-times, \cite{papapetrou:1957, papapetrou:1958:1, papapetrou:1958:2}, 
and for gravity coupled with matter fields, \cite{bic_scho_tod:time_per_sc:1, bic_scho_tod:time_per_sc:2}.
In most settings, the problem reduces to whether solutions of the Einstein equations (vacuum or with matter) 
which emit no radiation towards the null infinities $\mc{I}^+,\mc{I}^-$ must be stationary.
In view of the techniques developed in \cite{alex_io_kl:hawking_anal}, it seems that a positive 
answer to the uniqueness question for linear waves should be applicable towards the above problem.
We intend to return to this  in a subsequent paper; see also Section~\ref{sec:intro:relativity}.

The present paper, however, is primarily inspired by the challenge of deriving analogues of unique continuation from infinity and from a point, which have been studied for second order elliptic operators, \cite{aron:uc, cord:uc, hor:lpdo2, hor:lpdo3, kato:growth, reed-sim:scat, mesh:inf_decay, mesh:diff_ineq}, to the setting of wave equations.
We begin with a brief review of this subject; excellent broader reviews can be found in \cite{kenig:ams, tataru:ucbig, tataru:ucp} and references therein.

\subsection{Unique Continuation and Pseudoconvexity}

\emph{Unique continuation} is essentially a question of uniqueness of solutions to PDEs.
Consider a smooth function $h$ defined over a domain $\mc{D} \subset \R^m$, with $\ud h \neq 0$ everywhere in $\mc{D}$.
Let
\[ \Sigma := \{h = 0\} \text{,} \qquad \Sigma^+ := \{h \geq 0\} \text{,} \qquad \Sigma^- = \{h \leq 0\} \text{.} \]
The uniqueness problem is the following:
\footnote{For brevity, we present the discussion here for scalar equations.}

\begin{question}
Assume that:
\begin{itemize}
\item $\phi \in \mc{C}^2(\mc{D})$ solves the second-order PDE $p (x, D) \phi = 0$ in $\mc{D}$, and

\item $\phi = 0$ and $\ud \phi = 0$ on $\Sigma \cap B (P, \delta)$, where $P \in \Sigma$, and where $B (P, \delta)$ denotes the ball in $\R^m$ about $P$ of radius $\delta > 0$.
\end{itemize}
Is it true that $\phi = 0$ in $\Sigma^+ \cap B (P, \delta^\prime)$, for some (possibly smaller) $\delta^\prime > 0$?
\end{question}
 
\subsubsection{Second-order Hyperbolic Equations} 

When $p(x, D)$ is a wave operator, with principal part of the form $\Box_g$, this question is of particular interest
 when $\Sigma$ is non-characteristic and time-like (with respect to $g$, thought of as a Lorentzian metric). 
In that case, the Cauchy problem is ill-posed, \cite{PC:non-uniq}, yet uniqueness may sometimes hold. 
The key condition that ensures uniqueness is H\"ormander's strong pseudo-convexity condition, which in this setting takes the form \cite{ler_robb:unique}:
\begin{equation} \label{pseudo.cvx} D^2 h (X, X) < 0 \text{,} \qquad \text{if } g (X, X) = g (X, Dh) = 0 \text{.} \end{equation}
In particular, the pseudo-convexity depends only on the principal symbol of $p (x, D)$ and its relation to the hypersurface $\Sigma$.
In this setting, one has the following result:
 
\begin{theorem} \label{uniq.H}\cite{hor:lpdo4}
Assume $L_g := \Box_g + a^\alpha \partial_\alpha + V$ is a wave operator with smooth coefficients, and suppose $\psi \in \mc{C}^2 (\mc{D})$ solves $L_g \phi = 0$ in $\mc{D}$.
Assume furthermore that $\phi=0, \ud \phi \ne  0$ on $\Sigma$ around $P$.
Then, if $\Sigma$ is strongly pseudo-convex at $P$ in the sense of \eqref{pseudo.cvx}, then $\phi$ vanishes in a (relatively open) domain in $\Sigma^+$ near $P$.
\end{theorem}

The condition \eqref{pseudo.cvx} can be interpreted geometrically as convexity of $\Sigma^+$ with respect to null geodesics at $P \in \Sigma$.
Specifically, \eqref{pseudo.cvx} holds if and only if any null geodesic that is tangent to $\Sigma$ at $P$ locally lies in $\Sigma^-$, with first order of contact at $P$.
The necessity of the pseudo-convexity condition has been shown by Alinhac,
 \cite{alin:non_unique}: He produced examples of wave operators $L_g$ for which 
 unique continuation, as formulated in Theorem \ref{uniq.H}, fails across a {\it non-pseudoconvex} surface $\Sigma$. 

The proof of Theorem \ref{uniq.H} relies on Carleman estimates for $\Box_g$. A key element of such estimates is the construction of a weight function $f$ whose level sets are pseudo-convex, and for which the level sets $\{f = C\} \cap \Sigma^+$ are compact near $P$.

\subsubsection{Second-order Elliptic Equations}

Unique continuation holds always for second order elliptic operators across smooth hypersurfaces; see \cite{cald:unique}.
However, an interesting modification is when one replaces the assumption of zero Cauchy data on a hypersurface with data \emph{at a point}.
In that case, the necessary requirement, due to \cite{aron:uc, cord:uc}, is vanishing to infinite order at that point:\footnote{The infinite-order vanishing 
 is clearly necessary, as the example of homogenous harmonic polynomials of any order shows. }

\begin{definition}
We say that $\phi$ \emph{vanishes to infinite order} at $P \in \R^m$ if  there exists $\delta > 0$ such that for every $N \in \mathbb{N}$,
\[ \int_{ B (P, \delta) } |\phi|^2 r^{-N} < \infty \text{,} \]
where $r (x) = |x - P|$.
\end{definition}

\begin{theorem} \cite{aron:uc, cord:uc}
Assume that $\phi$ solves a second order elliptic equation
\beq
\label{elliptic}
H\phi:=-\Delta_g \phi + a^\alpha \partial_\alpha \phi + V \phi = 0,
\eeq 
with smooth coefficients in a domain $\mc{D} \subset \R^m$.
Then, if $\phi$ vanishes to infinite order at $P \in \mc{D}$, it vanishes in a neighborhood of $P$.
\end{theorem}

An analogue of this question is to assume infinite-order vanishing at infinity:
\begin{question}
Given self-adjoint operators $H$ of the form (\ref{elliptic}) and a solution $\phi$ to 
\begin{equation}
  \label{eq:ev}
  H\phi=\lambda \phi \qquad \text{with } \lambda\ge 0\,,
\end{equation}
in  a neighborhood $G$ of infinity in $\mathbb{R}^n$, then if $\phi$ satisfies
\beq \int_G \phi^2r^N<\infty \text{,} \qquad \text{for all } N \in \mathbb{N} \text{,} \eeq
and $a^\alpha, V$ satisfy suitable decay conditions, does $\phi$ vanish in $G$?
\end{question}  

This has been derived by many authors in various settings, for example \cite{agmon:lower, hor:lpdo2, hor:lpdo3, IJ:pos-eig, kato:growth, koch_tat:carl, mesh:inf_decay}, to name a few. 
One of the motivations for this question is that (in the case where $g$ is the Euclidean metric or a small perturbation  on $\mathbb{R}^n$) an affirmative answer implies the non-existence of $L^2$-solutions to \eqref{eq:ev} with $\lambda>0$.
For if $\phi\in\mathrm{L}^2$ is a solution to \eqref{eq:ev} with $\lambda>0$, then the equation implies that $\phi$ vanishes to infinite order at infinity, and the question reduces to unique continuation from infinity. (This is implicit in \cite[Ch. XIV]{hor:lpdo2}.)
In the case where $\lambda=0$, the infinite order vanishing cannot be derived and must be a priori assumed. 
Our results in this paper can be thought of as direct analogues of the above problem for second-order hyperbolic equations. 

\subsection{Time-periodicity and applications}
We now turn to the connection with periodic-in-time solutions.

\subsubsection{Time-periodic solutions}
Let us note here that the absence of positive eigenvalues $\lambda>0$ for $H$ is equivalent to the non-existence of periodic-in-time solutions of the form
\begin{equation}
  \label{eq:sepv}
  \phi(t,x):=e^{i\sqrt{\lambda}t}\psi(x)\,,
\end{equation}
to the corresponding wave equation 
\beq
\label{wave2}
-\partial^2_{tt}\phi(t,x)-H\phi(t,x)=0\,,
\eeq
over the $(n+1)$-dimensional Minkowski space-time $\mathbb{R}^{n+1} = \{ t \in \mathbb{R} \text{, } x \in \mathbb{R}^n \}$, with $\psi(x) \in L^2 (\mathbb{R}^n )$.
Results on the absence of positive eigenvalues 
of Schr\"odinger operators $H:=-\Delta +V$, (with $\Delta$ being the Euclidean (flat) Laplacian over $\mathbb{R}^n$ and $V$ a suitably decaying potential),
have been derived by Kato, \cite{kato:growth}, and Agmon, \cite{agmon:lower}, for potentials $V$ obeying suitable pointwise decay conditions; see also \cite{reed-sim:scat}. 
More recently,
results have been obtained for rough potentials $V$ in suitable $\mathrm{L}^p$ spaces, \cite{IJ:pos-eig, koch_tat:carl}. 
One can also prove the absence of the zero eigenvalue of operators $H$, which correspond to constant-in-time solutions of (\ref{wave2}), but the decay assumptions on the solution must be strengthened to vanishing to infinite order at infinity, and the potential $V$ must decay faster; see \cite{krs:carl}.
\footnote{In a separate direction, Finster, Kamran, Smoller, and Yau, \cite{fksy:nonex}, proved the non-existence of periodic-in-time solutions of the Dirac equation in the Kerr exterior.}

Now a time-periodic solution of the form \eqref{eq:sepv} to (\ref{wave2}) would in fact vanish on the entire $\mc{I}^+, \mc{I}^-$, in the sense that we will be considering here for solutions of (\ref{wave}), and would thus fall under the assumptions of our theorems below.
In this sense, our work can be considered a generalization of the above results; however, we treat time-dependent wave equations directly.
Furthermore we will be considering a more localized version of this problem: we show that vanishing of the 
solution of (\ref{wave}) on {\it parts} of $\mc{I}^+,\mc{I}^-$ suffices to derive the vanishing of the solution in a {\it part} of the
interior. Our results are robust in that they will hold for  perturbations of the Minkowski metric. 
In fact a large part of this paper is concerned with understanding  how the (optimal) result depends on the 
geometry of the background metric. 
The conditions we impose on the lower-order terms are similar to those one imposes for the zero-eigenvalue problem for operators $H=\Delta +a^\alpha\partial_\alpha+V$. 
Results on unique continuation from spatial  infinity for time-dependent Schr\"odinger equations have also recently been obtained by Escauriaza, Kenig, Ponce and Vega; see \cite{ekpv:ucSchr} and references therein.

\subsubsection{Applications to general relativity}
\label{sec:intro:relativity}

A separate question to which our study here pertains directly, is that of inheritance of symmetry: whether matter fields coupled to the space-time via the Einstein equations must inherit the symmetries of the underlying space-time. In particular, Bi\v{c}\'ak, Scholtz, and Tod, \cite{bic_scho_tod:time_per_sc:1, bic_scho_tod:time_per_sc:2}, consider Einstein-Maxwell and (various massless and massive) Einstein-scalar field space-times, for which the underlying space-time is stationary. Under the assumptions of analyticity at $\mc{I}^-$ and asymptotic simplicity, they derive the inheritance of symmetry for some of the fields in question as follows: From the asymptotic simplicity of the space-time metric, the authors show that the $\bf T$-derivative of the matter field $\tilde{\phi}$ must vanish to infinite order on $\mc{I}^-$, where $\mathbf{T}$ is the  stationary  Killing field. The real-analyticity of the matter fields at $\mc{I}^-$ then implies that the fields vanish off $\mc{I}^-$ as well. Our Theorems \ref{thm1} and \ref{thm1.pm} below, applied to $\mathbf{T} \tilde{\phi}$, allow us to remove the analyticity assumption for those matter models for which the equations of motion reduce to  a wave equation of the type considered there.\footnote{In view of \cite{IK:killing}, one can probably not hope that analyticity can in fact be {\it derived}  from the nature of the problem.}

On the other hand,  as  recalled in \cite{bic_scho_tod:time_per_sc:1, bic_scho_tod:time_per_sc:2}, there are examples obtained by Bizon-Wasserman, \cite{biz_wass:boson}, of massive Einstein-Klein-Gordon space-times in which the space-time is static but the field is in fact time-periodic, thus the underlying symmetry is {\it not} inherited. These examples do  {\it not} contradict the theorems in \cite{bic_scho_tod:time_per_sc:1, bic_scho_tod:time_per_sc:2}, since these solutions are manifestly non-analytic at $\mc{I}^-, \mc{I}^+$. Nonetheless, such examples raise the challenge of finding a suitable condition on the vanishing of the $\mathbf{T}$-derivatives  of the fields which would allow for the extension of symmetry (alternatively, the vanishing of $\mathbf{T} \tilde{\phi}$ in the space-time) to hold. Theorems \ref{O1} and \ref{O1.pm} below provide such a condition, in the vanishing of the solution at a specific exponential rate; for general wave equations, this is nearly optimal.

\subsection*{Acknowledgments.}

We are grateful to Sergiu Klainerman and Alex Ionescu for many helpful conversations. We thank them for generously sharing their results in \cite{ik:private}.
The first author was partially supported by NSERC grants 488916 and 489103, and a Sloan fellowship.
The second author was supported by NSF grant 0932078 000, while in residence at MSRI, Berkeley, CA, in the fall semester 2013.

\section{Results and Discussion} \label{sec:results}

Our main results deal with linear wave equations on various backgrounds.  The results readily apply to semi-linear wave equations,
under suitable pointwise decay assumptions on the solutions.
For simplicity of the presentation, we consider the case of one scalar equation, although the methods generalize readily to systems.

\subsection{Perturbations of Minkowski Space-time}

The first set of results deals with $(n+1)$-dimensional Minkowski space-time $\R^{n+1}$ and a large class of its perturbations.
Recall that the Minkowski metric itself is given by
\begin{equation} \label{eq.met_mink} g_{\rm M} := -4 \,\ud u \ud v + \sum_{A,B=1}^{n-1} r^2 \g_{AB} \ud y^A \ud y^B \text{,} \end{equation}
where $u$ and $v$ are the standard optical functions
\begin{equation} \label{optical} u := \frac{1}{2} (t - r) \text{,} \qquad v := \frac{1}{2} (t + r) \text{,} \end{equation}
and where $r (u, v) = v - u$.
In addition, $y^A, y^B$ are coordinates on the sphere $\Sph^{n-1}$, while $\g_{AB}$ is the round metric on $\Sph^{n-1}$ expressed in those coordinates. 
Recall also that the future and past null infinities $\mc{I}^\pm$ of $\R^{n+1}$ correspond to the boundaries
\[ \mc{I}^+ := \{ v = +\infty \text{, } u \in (-\infty, +\infty) \} \text{,} \qquad \mc{I}^- := \{ u = -\infty \text{, } v \in (-\infty, +\infty) \} \text{.} \]

We make the convention, both here and below, that uppercase Roman indices $A$, $B$, etc.~correspond to coordinates on the level spheres $\mc{S}_{u, v}$ of $(u, v)$, while Greek indices $\alpha$, $\beta$, etc.~correspond to all $n + 1$ coordinate functions $u, v, y^1, \dots, y^{n-1}$.

\begin{definition}
For any fixed $\e > 0$:
\begin{itemize}
\item Let $f_\e$ denote the function 
\begin{equation} \label{fe} f_\e := (v + \e)^{-1} (-u + \e)^{-1} \text{.} \end{equation}

\item Moreover, for $\omega > 0$, we define the corresponding domain
\begin{equation} \label{domain} \mc{D}^\e_\omega := \{ 0 < f_\e < \omega \} \text{.} \end{equation}

\item Define also the following subsets of future and past null infinity:
\begin{equation} \label{Ie} \mc{I}^+_\e := \{ v = +\infty, u \leq \e \} \subset \mc{I}^+ \text{,} \qquad \mc{I}^-_\e := \{ u = -\infty, v \geq -\e \} \subset \mc{I}^- \text{.} \end{equation}
\end{itemize}
\end{definition} 

The function $f_\e$ is a key quantity that will be used both to state and to prove our theorems.
Thus, we collect some simple observations concerning $f_\e$ here:
\begin{itemize}
\item The positive level sets of $f_\e$ turn out to be strongly pseudo-convex (see Section \ref{sec:psc}).
Moreover, these level sets all intersect $\mc{I}^+$ and $\mc{I}^-$ at the spheres $\{ u = \e, v = +\infty \}$ and $\{ v = -\e, u = -\infty \}$, respectively (see Figure \ref{fig:minkowski:pseudoconvexity}).

\item Note that $\{ f_\e = 0 \}$ corresponds to the segments $\mc{I}^+_\e \cup \mc{I}^-_\e$ of null infinity.

\item Note also that $f^\e \sim r^{-2}$ near spatial infinity $\iota^0$, while $f^\e \sim r^{-1}$ at the interior points of $\mc{I}_\e^+$ and $\mc{I}_\e^-$. 
\end{itemize}

We will be considering perturbations of $g_{\rm M}$ of the general form:
\begin{equation} \label{gen.form.i} \begin{aligned}
g &= \mu\, \ud u^2 - 4 K \, \ud u \ud v + \nu\, \ud v^2 + \sum_{A, B = 1}^{n-1} r^2 \gamma_{AB} \ud y^A \ud y^B \\
&\qquad + \sum_{A = 1}^{n-1} ( c_{Au} \ud y^A \ud u + c_{Av} \ud y^A \ud v ) \text{,}
\end{aligned} \end{equation}
where again $r = v - u$, and $u, v$ are defined as in \eqref{optical}.
In order to describe precisely the asymptotic conditions required for our coefficients $\mu$, $\nu$, $K$, $c_{Au}$, $c_{Av}$, $\gamma_{AB}$, we make the following definitions:

\begin{definition} \label{O}
Given a function $G = G (u, v)$, we define the following:
\begin{itemize}
\item A $\mc{C}^0$-function $\varphi$ belongs to $\mc{O}^\delta (G)$ iff $|\varphi| \leq \delta G$.

\item A $\mc{C}^0$-function $\varphi$ belongs to $\mc{O} (G)$ iff $\varphi \in \mc{O}^\delta (G)$ for some $\delta > 0$.

\item If both $G$ and $\varphi$ are $\mc{C}^1$, then we say that $f \in \mc{O}_1^\delta (G)$ iff
\[ \varphi \in \mc{O}^\delta (G) \text{,} \qquad | \partial_u \varphi | \leq \delta | \partial_u G | \text{,} \qquad | \partial_v \varphi | \leq \delta | \partial_v G | \text{,} \qquad | \partial_I \varphi | \leq \delta G \text{.} \]

\item Likewise, we say that $\varphi \in \mc{O}_1 (G)$ iff $\varphi \in \mc{O}_1^\delta (G)$ for some $\delta > 0$.
\end{itemize}
\end{definition}

In terms of the above, the decay assumptions for $g$ can now be stated as follows:
\begin{itemize}
\item \emph{There exists $\delta > 0$, sufficiently small with respect to $\e$, such that}
\begin{equation} \label{components.i} \begin{aligned}
K = 1 + \mc{O}_1^\delta ( r^{-2} ) \text{,} &\qquad \gamma_{AB} = \g_{AB} + \mc{O}^\delta_1 ( r^{-1} ) \text{,} \\
c_{Au}, c_{Av} = \mc{O}_1^\delta ( r^{-1} ) \text{,} &\qquad \mu, \nu = \mc{O}^\delta_1 ( r^{-3} ) \text{.}
\end{aligned} \end{equation}
\end{itemize}

We present two results for wave equations over such background metrics.
The results depend on the asymptotic behaviour of the lower-order terms, and the assumptions  vary accordingly.
Roughly speaking, our first theorem applies to linear wave equations over $(\R^{n+1}, g)$ whose lower order terms decay
sufficiently rapidly at infinity.
\footnote{The required fall-off matches the one needed to rule out the existence of the zero-eigenvalue for the 
corresponding elliptic operator. }
We will show that if solutions of such equations vanish {\it to infinite order} at $\mc{I}^+_\e$ and $\mc{I}^-_\e$, 
then the solution must vanish in an open domain in $\R^{n+1}$ which contains $\mc{I}^+_\e \cup \mc{I}^-_\e$ on its boundary.

\begin{theorem} \label{thm1}
Consider a perturbed Minkowski metric $g$ over $\R^{n+1}$ of the form \eqref{gen.form.i}, \eqref{components.i}.
Consider also any wave operator $L_g := \Box_g + a^\alpha \partial_\alpha + V$, where
\begin{equation} \label{thm1.coeff} \begin{aligned}
a^u \in \mc{O} ( (v + \e)^{-1} r^{-\frac{1}{2}} ) \text{,} &\qquad a^v \in \mc{O} ( (-u + \e)^{-1} r^{-\frac{1}{2}} ) \text{,} \\
a^I \in \mc{O} ( f^\frac{1}{2} r^{-\frac{3}{2}} ) \text{,} &\qquad V \in \mc{O} ( f^{1 + \eta} ) \text{,}
\end{aligned} \end{equation}
for some $\eta > 0$.
Let $\omega > 0$, and consider any $\mathcal{C}^2$-solution $\phi$ on $\mc{D}^\e_\omega$ of the equation $L_g \phi = 0$, which in addition satisfies
\begin{equation} \label{key.assn1} \int_{ \mc{D}^\e_\omega } \phi^2 r^N + \int_{ \mc{D}^\e_\omega } | \partial_v \phi |^2 r^N + \int_{ \mc{D}^\e_\omega } | \partial_u \phi |^2 r^N + \sum_{I = 1}^{n-1} \int_{ \mc{D}^\e_\omega } | \partial_I \phi |^2 r^N < \infty \text{,} \end{equation}
for all $N \in \mathbb{N}$.
Then, there exists $0 < \omega^\prime < \omega$ so that $\phi \equiv 0$ in $\mc{D}^\e_{\omega^\prime}$. 
\end{theorem}

We digress now to discuss the optimality of the above, as well as some alternate ways of phrasing this result.

\begin{rmk}
An alternate statement of Theorem \ref{thm1} is in terms of a differential \emph{inequality}.
More specifically, one can rephrase Theorem \ref{thm1} with $\phi$ satisfying
\[ | \Box_g \phi | \leq \sum_\alpha | a^\alpha | | \partial_\alpha \phi | + | V | | \phi | \text{,} \]
rather than the linear equation $L_g \phi \equiv 0$.
Analogous variants also hold true with respect to the remaining main theorems in this paper.
Moreover, these alternative formulations can be proved in exactly the same manner. 
\end{rmk}


\begin{rmk}
The assumption $V \in \mc{O} ( f^{1 + \eta} )$, $\eta > 0$, implies that $V$ vanishes at a  rate $r^{-2 - 2 \eta}$ at spatial infinity and at a rate of $r^{-1 - \eta}$ at (future and past) null infinities.
In fact, straightforward examples for elliptic operators of the form $\Delta + V$ reveal that if $V$ is only assumed to decay at a rate $r^{-2 + 2 \eta}$, the result would fail.
Thus, our theorem is nearly optimal with regards to the decay conditions on the lower-order terms.
\footnote{In fact, this decay condition can be relaxed a bit, using the generalization of the Carleman estimates described in the remark at the end of Section \ref{sec:carleman}. However, we prefer to give the cleaner statement above, rather than burden the reader with the most general version of the result possible.} 
\end{rmk}

\begin{rmk}
Since this result only assumes smoothness of the solution in $\mc{D}^\e_\omega$, as opposed to all of $\R^{n+1}$, the assumption of vanishing of infinite order, in the sense of \eqref{key.assn1}, is \emph{necessary}.
Indeed, in Minkowski space-time, one can consider
\begin{equation} \label{counterex_order} \phi_N := \partial_{i_1} \dots \partial_{i_N} r^{-n + 2} \text{,} \end{equation}
where the $\partial_{i_j}$'s are spatial Cartesian coordinate derivatives on $\R^n$; see also \cite{helg:radon}.
This function solves $\Delta_{\R^n} \phi_N = 0$ away from the origin, and it vanishes to order $N + n - 2$ at infinity.
\end{rmk}

\begin{rmk}
It is worth noting that in Minkowski space-time, our methods imply that a smooth solution 
$\phi$ vanishes in the whole domain $f < \e^{-2}$, which in particular includes the entire double-null cone centered at the origin. 
Thus, standard energy estimates imply that in this case, $\phi$ vanishes in the entire space-time. 
\end{rmk}

Let us recall at this point the standard Penrose conformal compactification of Minkowski space-time.
\footnote{This can be thought of as a Lorentzian analogue of the stereographic projection.} 
Let
\begin{equation} \label{penrose_conf} \Omega := (1 + u^2)^{-\frac{1}{2}} (1 + v^2)^{-\frac{1}{2}}  \text{,} \end{equation} 
and consider the new metric $\bar{g} := \Omega^2 g_{\rm M}$.
Applying the change of coordinates $U := \tan^{-1} u$ and $V := \tan^{-1} v$, we see that
\begin{equation} \label{penrose_met} \bar{g} = -4\, \ud U \ud V + \sin^2 (V - U) ( \ud \theta^2 + \sin^2 \theta \ud \phi^2 ) \text{.} \end{equation}
In particular, $\bar{g}$ extends smoothly to the boundaries $V = +\frac{\pi}{2}$ and $U = -\frac{\pi}{2}$, which correspond to $\mc{I}^+$ and $\mc{I}^-$, respectively.
In fact, the compactified manifold $(\bar{\R}^{n+1}, \bar{g})$ is isometric to a relatively compact domain in the Einstein cylinder $\Sph^n \times \R$ (with the natural product metric); cf.~Figure \ref{fig:penrose:compact}.
\begin{figure}[tb]
  \centering
  \includegraphics{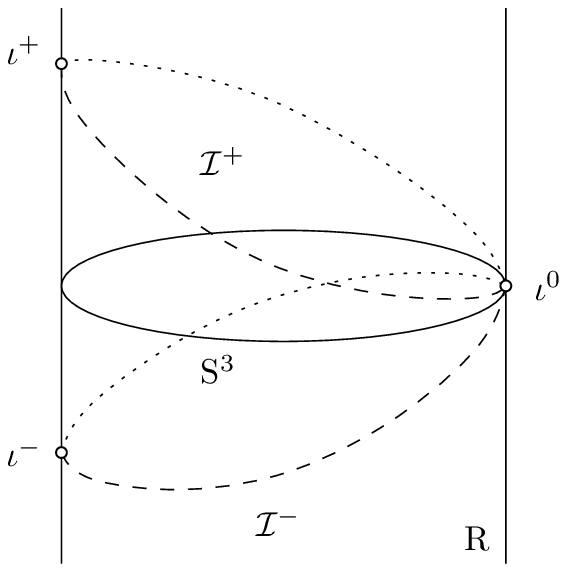}
  \caption{Penrose conformal compactification of Minkowski space-time.}
  \label{fig:penrose:compact}
\end{figure}

\begin{rmk}\label{rmk:penrose}
The methods of \cite{alin:non_unique} strongly suggest that in the Minkowski setting, the assumption of (infinite-order) vanishing on at least
\[ \mc{I}^+_0 \cup \mc{I}^-_0 = \{ v = \infty \text{, } u \leq 0 \} \cup \{ u = -\infty \text{, } v \geq 0 \} \]
is also necessary.
In particular, \cite{alin:non_unique} showed that unique continuation across a (smooth) hypersurface $\mc{H}$ requires pseudo-convexity of $\mc{H}$, in general.
Well-known examples of hypersurfaces that are \emph{not} pseudo-convex in Minkowski space-time are the one-sheeted hyperboloids $\mc{H}_C = \{ -t^2 + r^2 = C \}$.
The failure of pseudo-convexity is captured by the fact that any null geodesic tangent to such a hyperboloid $\mc{H}_C$ will remain on $\mc{H}_C$.

Recalling the Penrose compactification above, and using the conformal Laplacian and its conformal covariance 
(see Section \ref{sec:operators:transform}), we observe that $\bar{\phi} := \Omega^{-(n-1)/2} \phi$ 
solves an analogous wave equation with respect to the compactified metric $\bar{g}$.
Now, the hyperboloids $\mc{H}_C$ are mapped to smooth surfaces $\bar{\mc{H}}_C$ in the Einstein cylinder that converge to $\mc{I}_0^+ \cup \mc{I}^-_0$ (embedded in the Einstein cylinder) as $C \nearrow \infty$.
Furthermore, by the conformal invariance of null geodesics, the $\bar{\mc{H}}_C$'s continue to be ruled by null geodesics.
\footnote{This property can also be seen by examining the null geodesics in the compactified space-time $\bar{\R}^{n+1} \subset \Sph^n \times \R$. 
The geodesic motion can be decomposed into a constant motion in the $\R$-direction and a geodesic motion on $\Sph^n$.
Then, the focusing of the $\bar{\mc{H}}_C$'s (and of the null geodesics on them) is just a feature of the positive curvature of $\Sph^n$.}

In particular, the existence of these non-pseudoconvex surfaces seem to suggest that solutions to the equation in the compactified picture which vanish to infinite order on less than $\mc{I}^+_0 \cup \mc{I}^-_0$ will not necessarily vanish in the interior near those boundaries.
It would follow that $\phi$ also does not vanish.
\end{rmk}

\begin{rmk}
In certain cases, one would like to have assumptions in terms of $\mc{C}^\infty$-norms instead 
of the weighted Sobolev norms in \eqref{key.assn1}. Moreover, based on a formal analysis of the wave equation on the characteristic surfaces 
$\mc{I}^+ \cup \mc{I}^-$, one would expect that, under sufficient regularity assumptions
 on the solution on $\mc{I}^+ \cup \mc{I}^-$, one should be able to only assume vanishing of the radiation field of the solution to derive the above uniqueness.

This is indeed true in Minkowski space-time.
Transforming the wave equation to the Penrose compactified picture, as before, we can argue that \emph{if the solution $\bar{\phi}$ in the Penrose setting were $\mc{C}^\infty$ up to the entire boundary $\mc{I}^+ \cup \mc{I}^-$,} \footnote{In the sense that the function admits a $\mc{C}^\infty$-extension to the entire ambient manifold $\Sph^n \times \R$.} \emph{then only the assumption that $\bar{\phi} = 0$ on $\mc{I}^+ \cup \mc{I}^-$ suffices to derive that the solution vanishes in the interior.}
To see this, one uses the vanishing of $\ud \bar{\phi}$ at $\iota^0$, along with the wave equation for $\bar{\phi}$ and its derivatives (as propagation equations along null geodesics on $\mc{I}^\pm$), to show that all derivatives of $\bar{\phi}$ vanish on $\mc{I}^+ \cup \mc{I}^-$.
This implies infinite-order vanishing in the physical picture, and the result will follow from Theorem \ref{thm1}.

However, the assumption that a function should be $\mc{C}^\infty$ in the compactified setting is quite strong.
For example, all the functions $\phi_N$ defined in \eqref{counterex_order} yield corresponding functions $\bar{\phi}_N$ in the compactified picture which fail to be $\mc{C}^\infty$ precisely at $\iota^0$. 
Furthermore, it should be noted that generic perturbed space-times of the type \eqref{gen.form.i} yield non-smooth metrics in the compactified picture \cite{ch:rome}, and hence the above formal analysis is not possible.
\end{rmk}

A question of separate interest is wave equations on perturbed Minkowski space-times with potentials of order $O(1)$.
This in particular includes massive Klein-Gordon equations, where, as observed by examples
 in \cite{bic_scho_tod:time_per_sc:2, biz_wass:boson}, the previous result fails.
It turns out that one must assume faster than polynomial vanishing of the solution on $\mc{I}^+_\e \cup \mc{I}^-_\e$ to derive uniqueness in this setting.
In particular, we require the solution to vanish faster than the \emph{exponential} rate $\exp ( r^{-4/3} )$.

\begin{theorem}
\label{O1}
Consider a perturbed Minkowski metric $g$ over $\R^{n+1}$ in the form 
\eqref{gen.form.i} and \eqref{components.i}, and consider a wave operator $L_g := \Box_g + a^\alpha \partial_\alpha + V$, where
\begin{equation} \label{O1.coeff} \begin{aligned}
a^u \in \mc{O} ( v^{-1} f^{-\frac{1}{3}} r^{-\frac{1}{2}} ) \text{,} &\qquad a^v \in \mc{O} ( (-u)^{-1} f^{-\frac{1}{3}} r^{-\frac{1}{2}} ) \text{,} \\
a^I \in \mc{O} ( f^\frac{1}{6} r^{-\frac{3}{2}} ) \text{,} &\qquad V \in \mc{O} ( 1 ) \text{.}
\end{aligned} \end{equation}
Let $\omega > 0$, and consider any $\mathcal{C}^2$-solution $\phi$ on $\mc{D}^\e_\omega$ of the equation $L_g \phi = 0$, which in addition satisfies
\begin{equation} \label{key.assn1.O1} \int_{ \mc{D}^\e_\omega } \phi^2 e^{ N r^\frac{4}{3} } + \int_{ \mc{D}^\e_\omega } | \partial_v \phi |^2 e^{ N r^\frac{4}{3} } + \int_{ \mc{D}^\e_\omega } | \partial_u \phi |^2 e^{ N r^\frac{4}{3} } + \sum_{I = 1}^{n-1} \int_{ \mc{D}^\e_\omega }| \partial_I \phi |^2 e^{ N r^\frac{4}{3} } < \infty \text{,} \end{equation}
for all $N \in \mathbb{N}$.
Then, there exists $0 < \omega^\prime < \omega$ so that $\phi \equiv 0$ in $\mc{D}^\e_{\omega^\prime}$. 
\end{theorem}  
  
In view of the results of \cite{mesh:inf_decay, mesh:diff_ineq}, the required rate of vanishing for the solution is nearly optimal, given the assumptions on the bounds of the lower-order terms.
In particular, it was shown that when $V$ is allowed to be complex-valued, there exist solutions of $( \Delta + V ) \phi = 0$ in $\R^3 \setminus B_1$ obeying the bound $|\phi| \leq C \exp ( -c r^{4/3} )$ which do not vanish.

\subsection{Schwarzschild and positive mass space-times.}
We present our next theorem for Schwarzschild space-times and general perturbations, which include all Kerr metrics.
Surprisingly, the theorem here requires a \emph{weaker} assumption at infinity than in Minkowski space-time; in particular only vanishing on \emph{arbitrarily small} portions of $\mc{I}^+, \mc{I}^-$ 
near $\iota^0$ is assumed.
Morally, this is due to the \emph{stronger} pseudo-convexity arising from the positive mass, rather than from where the chosen hyperboloids are anchored, as in the Minkowski case.

Recall the form of the Schwarzschild metric in the exterior region
\footnote{For the sake of clarity, we stress that we are referring to the Schwarzschild metric in $n + 1 = 4$ dimensions only. We note that the higher-dimensional Schwarzschild metrics actually fall under the assumptions of Theorems \ref{thm1}, \ref{O1} above, {\it not} the theorems below.}
\beq g_{\rm S} := - \left( 1 - \frac{2 m_S}{r} \right) \ud t^2 + \left( 1 - \frac{2 m_S}{r} \right)^{-1} \ud r^2 + r^2 \sum_{A, B = 1}^2 \g_{AB} \ud y^A \ud y^B \text{,} \label{eq:schwarzschild:intro}\eeq
where the mass $m_S$ is a positive constant, and where $r > 2 m_S$.
To express the metric in null coordinates, we recall the definition of the Regge-Wheeler coordinate:
\beq \label{tortoise} r^* (r) := \int_{r_0}^r \left( 1 - \frac{2 m_S}{s} \right)^{-1} \ud s \text{,} \qquad r_0 > 2 m_S \text{.} \eeq
For a fixed constant $r_0 > 2 m_S$, we denote the corresponding optical functions by
\[ u = u_{r_0} := \frac{t - r^*}{2} \text{,} \qquad v = v_{r_0} = \frac{t + r^*}{2} \text{,} \]
and the Schwarzschild metric takes the form:   
\beq \label{schwarz} g_{\rm S} = -4 \left( 1 - \frac{2 m_S}{r} \right) \ud u \ud v + r^2 \sum_{A, B = 1}^2 \g_{AB} \ud y^A \ud y^B \text{.} \eeq

While future and past null infinity are identified (for any choice of $r_0$) with
\[ \mc{I}^+ := \{ v = +\infty \text{, } - \infty < u < \infty \} \text{,} \qquad \mc{I}^- = \{ u = -\infty \text{, } -\infty < v < \infty \} \text{,} \]
respectively, we note that $\{ u = 0 \}$ and $\{ v = 0 \}$ intersect on $\{ t = 0 \}$ at the sphere of radius $r = r_0$; c.f.~Fig.~\ref{fig:v:S:r0}.
\begin{figure}[bt]
  \centering
  \includegraphics{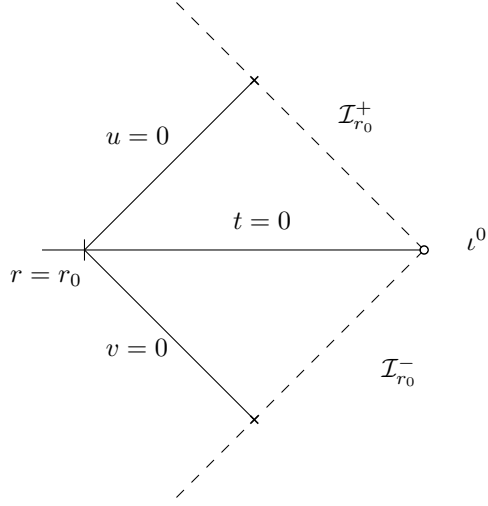}
  \caption{Any small neighborhood of $i^0$ can be chosen to be the region $u<0$, $v>0$.}
  \label{fig:v:S:r0}
\end{figure}
Thus, the region $\mc{D}^{r_0} = \{ v > 0 \text{, } u < 0 \}$ corresponds to a \emph{subdomain} of the entire Schwarzschild exterior.
More precisely, by choosing $r_0$ as large as we wish, $\mc{D}^{r_0}$ corresponds to the exterior region of a bifurcate null surface 
emanating from a sphere of \emph{arbitrarily large radius} $r_0$, or equivalently, an arbitrarily small neighborhood of spacelike infinity, $i^0 = \{ u = -\infty \text{, } v = \infty \}$.

Our next result is that for certain perturbations of the Schwarzschild metric, analogues of Theorems \ref{thm1} and \ref{O1} hold, with the assumption of (infinite-order) vanishing on $\mc{I}^+_\e \cup \mc{I}^-_\e$ replaced by vanishing just on the regions
\[ \mc{I}^+_{r_0} := \{ v = +\infty \text{, } u < 0 \} \text{,} \qquad \mc{I}^-_{r_0} := \{ u = -\infty, v > 0 \} \text{,} \]
\emph{for any chosen $r_0 > 2 m_S$},
and the domain replaced by 
\begin{equation}  \label{fr} \mc{D}^{r_0}_\omega := \{ 0 < f_{r_0} < \omega \} \quad\text{where } f_{r_0} := \frac{1}{(-u)v} \text{.} \end{equation}

Since this feature solely relies on the positivity of the mass $m_S>0$ we are led to consider an even more general class of metrics:
These background metrics correspond to space-times with positive 
but \emph{non-constant} Bondi energy and angular momentum at portions of $\mc{I}^+$ and $\mc{I}^-$ which ``join up" at $\iota^0$.

Specifically, we consider a manifold of the form
\begin{equation} \label{gen.domain} \mc{D} = (-\infty, 0) \times (0, \infty) \times \Sph^{n-1} \text{,} \end{equation}
and we let $u$ and $v$ denote the projections from $\mc{D}$ to its first and second components.
On $\mc{D}$, we consider metrics of the form:
\begin{equation} \label{gen.form.m} \begin{aligned}
g &= \mu \,\ud u^2 - 4 K \,\ud u \ud v + \nu \,\ud v^2 + \sum_{A, B = 1}^{n-1} r^2 \gamma_{AB} \ud y^A \ud y^B \\
&\qquad + \sum_{A = 1}^{n-1} ( c_{Au} \ud y^A \ud u + c_{Av} \ud y^A \ud v ) \,,
\end{aligned} \end{equation}
where $r$ and $m$ are now smooth positive \emph{functions} on $\mc{D}$ which satisfy certain bounds that we impose below, and $y^A:1,\ldots,n-1$ are coordinates on the level spheres $\mc{S}_{u, v}$ of $(u, v)$.

We impose the following assumptions:
\begin{itemize}
\item The components of $g$ satisfy the following bounds:
\begin{equation} \label{components.m} \begin{aligned}
K = 1 - \frac{2 m}{r} \text{,} &\qquad \gamma_{AB} = \g_{AB} + \mc{O}_1 \left( \frac{1}{v - u} \right) \text{,} \\
c_{Au}, c_{Av} = \mc{O}_1 \left( \frac{1}{v - u} \right) \text{,} &\qquad \mu, \nu = \mc{O}_1 \left( \frac{1}{ (v - u)^3 } \right) \,,
\end{aligned} \end{equation}
where $\g$ is the round metric.
\item $m$ satisfies the uniform lower bound
\begin{equation} \label{ass.m.0} m \ge m_{\min} > 0 \text{,} \end{equation}
and $\ud m$ satisfies the following uniform bounds, for some $\eta>0$:
\footnote{Note that \eqref{ass.m.0} and \eqref{ass.m.1} imply $m$ is uniformly bounded from above and has limits at $\mc{I}^\pm$.}
\begin{equation} \label{ass.m.1} | \partial_I m | = \mc{O} ((v-u)(-uv)^{-\eta}) \text{,} \qquad | \partial_u m |, | \partial_v m | = \mc{O} \bigl((v-u)^{-2}\bigr) \text{.} \end{equation}

\item $r$ is bounded on a level set of $v - u$, i.e., there exist $C, M > 0$ such that
\begin{equation} \label{uvr0} 1 \ll | r (p) | \leq C \text{,} \qquad \text{ whenever } v (p) - u (p) = M \text{.} \end{equation}
Moreover, $r$ satisfies the differential inequality
\begin{align}
\label{uvr} &\left[ 1 + \mc{O} \left( \frac{1}{ (v - u)^2 } \right) \right] \ud v - \left[ 1 + \mc{O} \left( \frac{1}{ (v - u)^2 } \right) \right] \ud u = \\
\notag &\qquad = \left( 1 + \frac{2m}{r} \right) \ud r + \sum_{I = 1}^{n-1} \mc{O} \left( \frac{1}{v - u} \right) \ud y^I \text{.}
\end{align}

\item The following estimate holds for some $\eta>0$:
\begin{equation} \label{ass.box} \left| \Box_g \left( \frac{m}{r} \right) \right| = \mc{O} \Bigl( (-uv)^{-1-\eta} \Bigr) \text{.} \end{equation}
\end{itemize}

Note that the conditions imposed on the metrics above allow non-constant mass and angular momentum at $\mc{I}^+, \mc{I}^-$ 
and are consistent with (and weaker than) the ones imposed by Sachs, \cite{sachs:waves}.

We define
\begin{equation} \label{fm} f := \frac{1}{(-u)v} \text{,} \qquad \mc{D}_\omega := \{ 0 < f < \omega \} \text{.} \end{equation}

\begin{theorem} \label{thm1.pm}
Consider a metric $g$ on $\mc{D}$ of the form \eqref{gen.form.m}, satisfying the conditions \eqref{components.m}-\eqref{ass.box}, and consider a wave operator $L_g := \Box_g + a^\alpha \partial_\alpha + V$, with
\begin{equation} \label{thm1.pm.coeff} \begin{aligned}
a^u \in \mc{O} ( v^{-1} r^{-\frac{1}{2}} ) \text{,} &\qquad a^v \in \mc{O} ( (-u)^{-1} r^{-\frac{1}{2}} ) \text{,} \\
a^I \in \mc{O} ( f^\frac{1}{2} r^{-\frac{3}{2}} ) \text{,} &\qquad V \in \mc{O} ( f^{1 + \eta} ) \text{,}
\end{aligned} \end{equation}
for some $\eta > 0$.
Let $\omega > 0$, and consider any $\mathcal{C}^2$-solution $\phi$ on $\mc{D}_\omega$ of the equation $L_g \phi = 0$, which in addition satisfies
\begin{equation} \label{key.assn1.M} \int_{ \mc{D}_\omega } \phi^2 r^N + \int_{ \mc{D}_\omega } | \partial_v \phi |^2 r^N + \int_{ \mc{D}_\omega } | \partial_u \phi |^2 r^N + \sum_{I = 1}^{n-1} \int_{ \mc{D}_\omega } | \partial_I \phi |^2 r^N < \infty \text{,} \end{equation}
for all $N \in \mathbb{N}$.
Then, there exists $0 < \omega^\prime < \omega$ so that $\phi \equiv 0$ in $\mc{D}_{\omega^\prime}$. 
\end{theorem}

Finally, we present an analogue of Theorem \ref{O1} for this class of metrics.

\begin{theorem} \label{O1.pm}
Consider a metric $g$ on $\mc{D}$ of the form \eqref{gen.form.m}, satisfying the conditions \eqref{components.m}-\eqref{ass.box}, and consider a wave operator $L_g := \Box_g + a^\alpha \partial_\alpha + V$, with
\begin{equation} \label{O1.pm.coeff} \begin{aligned}
a^u \in \mc{O} ( v^{-1} f^{-\frac{1}{3}} r^{-\frac{1}{2}} ) \text{,} &\qquad a^v \in \mc{O} ( (-u)^{-1} f^{-\frac{1}{3}} r^{-\frac{1}{2}} ) \text{,} \\
a^I \in \mc{O} ( f^\frac{1}{6} r^{-\frac{3}{2}} ) \text{,} &\qquad V \in \mc{O} ( 1 ) \text{.}
\end{aligned} \end{equation}
Let $\omega > 0$, and consider any $\mc{C}^2$-solution $\phi$ on $\mc{D}_\omega$ of the equation $L_g \phi = 0$, which in addition satisfies
\begin{equation} \label{key.assn1.O1pm} \int_{ \mc{D}_\omega } \phi^2 e^{ N r^\frac{4}{3} } + \int_{ \mc{D}_\omega } | \partial_v \phi |^2 e^{ N r^\frac{4}{3} } + \int_{ \mc{D}_\omega } | \partial_u \phi |^2 e^{ N r^\frac{4}{3} } + \sum_{I = 1}^{n-1} \int_{ \mc{D}_\omega }| \partial_I \phi |^2 e^{ N r^\frac{4}{3} } < \infty \text{,} \end{equation}
for all $N \in \mathbb{N}$.
Then, there exists $0 < \omega^\prime < \omega$ so that $\phi \equiv 0$ in $\mc{D}_{\omega^\prime}$. 
\end{theorem}

\begin{rmk}
  It turns out that the above class of perturbations \eqref{gen.form.m}, \eqref{components.m} also includes all Kerr metrics (both sub- and super-extremal).
While this is not apparent in the Boyer-Lindquist coordinates, \footnote{Indeed, this cannot be achieved by defining $u, v$ purely in terms of the usual coordinates $t, r$. The obstruction is the coefficient of $\ud r^2$ in these coordinates.} there is a special coordinate transformation, discussed in Appendix \ref{sec:kerr}, which brings the Kerr metrics in the above form.
Roughly, the transformation is designed to undo the $2 \pi$-periodic rotation of the components in Boyer-Lindquist coordinates.
The Kerr metric in these \emph{co-moving} coordinates is then \emph{one order closer} to the Schwarzschild metric in terms of powers of $r^{-1}$ and hence satisfies \eqref{components.m}.
\end{rmk}

\begin{rmk}
As we will explain in more detail in the next section, the moral reason behind this weaker vanishing assumption in Theorem \ref{thm1.pm} is precisely the \emph{extra} convexity (towards null infinity) of certain null geodesics in Schwarzschild compared to Minkowski.
The pseudo-convexity of the above function $f_{r_0}$ directly implies that any rotational null geodesic in Schwarzschild which is $\delta$-close to $\iota^0$ (in the inverted picture, with respect to the inverted null coordinates $U$, $V$) \footnote{See the next section.} will in fact intersect both $\mc{I}^+, \mc{I}^-$ $\sqrt{\delta}$-close to $\iota^0$.
This is a manifestation of the blow-up of the $\rho$-component of the Weyl curvature in the inverted metric.
\end{rmk}

\begin{rmk}
We note that an analogue of our theorems for free wave equations on static warped product backgrounds (which fall under the zero-mass class considered in Thm.~\ref{thm1}) has been derived in \cite{sa-barreto:support} using the Radon transform; see earlier work \cite{helg:radon} for free waves on the Minkowski background.\footnote{As discussed in Remark~\ref{rmk:penrose}, in the absence of pseudo-convexity of some sort, the presence of  time-dependent lower-order terms in the equation would not allow for the generalization of these results, even on the Minkowski background.}
Analogous results were recently obtained in \cite{baskin-wang:rad} on the Schwarzschild space-time for spherically symmetric waves which are trivial near $i^0$. 
\end{rmk}

\subsection{Discussion of the Ideas}

The proof of all the above theorems will be based on new Carleman estimates. 
Such estimates are a common tool in unique continuation problems.
In fact, we derive our estimates in a uniform way, essentially proving all theorems above together.

The approach we use in deriving our Carleman estimates follows the method adopted in \cite{IK:illposed}, 
based on energy currents associated to the energy-momentum tensor of the wave equations under consideration.
There are several challenges that we must overcome in deriving Carleman estimates in our setting, 
which generally arise from the geometry of infinity in asymptotically Minkowskian spaces. 
We highlight some of these in the next section which deals only with the Minkowski and Schwarzschild space-times as 
model cases. It is useful to synopsize all of them here, however: 
 
\begin{itemize}
\item \emph{Degenerating pseudo-convexity:} The first step in deriving our estimates 
is to construct a function $f$ for each background whose level sets are pseudo-convex. 
These will be the functions $f_\e, f_{r_0}$ defined above (we collectively denote these by $f$). 
A first difficulty arises here, in that the pseudo-convexity of the level sets of $f$ {\it degenerates} towards infinity. 

\item \emph{A conformal inversion:} Partly forced by the methods for deriving Carleman estimates, we consider a conformal transformation of the domains $\mc{D}$ where we seek to derive the vanishing of the solutions of \eqref{wave}, so as to convert the null infinities $\mc{I}^+_\e, \mc{I}^-_\e$, etc. above into boundaries at finite (affine) distance.
It is natural to expect one has the freedom to perform 
such conformal transformations in view of the conformal invariance of null geodesics, whose geometry is the 
key to the pseudo-convexity requirement. 

The standard conformal transformation in our setting would be the Penrose conformal compactification.
However, it turns out that we \emph{can not} derive our Carleman estimates with that tool.
To overcome this, we consider the equation \eqref{wave} above with respect to the new metric $\bar{g} := K^{-1} f^2 g$.
This is actually a very natural transformation, and should be thought of as a warped version of the  conformal 
{\it isometry} of the Minkowski space-time, $g_{\rm M} \to u^{-2} v^{-2} g_{\rm M}$.
In this setting, the null infinities $\mc{I}^+_\e, \mc{I}^-_\e$ are transformed to a \emph{complete} double null cone with vertex at $\iota^0$.
One has a certain extra convexity near these cones in our picture (see the next section), which makes our result possible.  

\item \emph{Reparametrizations:} After this inversion, certain aspects of the analysis resemble the strong unique continuation discussed above.
In particular, it turns out it is necessary to work not with the function $f$,
but a new function $F (f)$ with the same level sets that ``accelerates'' away from $\mc{I}^\pm_\e$ faster.
It is this choice of function $F$ (given the foliation by the level sets of $f$) that differentiates between Theorems \ref{thm1}, and \ref{thm1.pm} from 
Theorems \ref{O1} and \ref{O1.pm}.\footnote{In choosing $F(f)$ correctly (in the conformally inverted picture) for Theorem \ref{thm1}, we were guided by \cite{ik:private}.}

\item \emph{Absorption of Error Terms:} To derive our weighted $L^2$-estimates, various error terms that arise in the analysis must be absorbed into the main terms.
This can be done readily in the setting of Theorem \ref{uniq.H}, where the initial surface $\Sigma$ is smooth and strongly pseudo-convex.
However in our setting the degeneration of the pseudo-convexity makes this very delicate, due to the presence of weights 
in the Carleman estimates that vanish/blow-up at different rates towards the boundary.  
Our choice of the conformal inversion is essential here in ensuring that error terms {\it can} be absorbed. 
\end{itemize}
 
\subsection{Outline of the Paper}

For the reader's convenience, we start with a separate Section \ref{sec:psc}, where the pseudo-convexity
 of the functions $f_\e, f_{r_0}$ is derived in the model Minkowski and Schwarzschild space-times.
In particular, the \emph{stronger} pseudo-convexity that the Schwarzschild space-times exhibit becomes apparent.
We also briefly present the Carleman estimates we will be deriving in those space-times.

In Section \ref{sec:carleman}, we derive the Carleman estimates needed for our theorems.
The estimates we derive here are adapted to the conformally inverted metric, and allow for the degenerating pseudo-convexity. 
 
In Section \ref{sec:massloss}, we prove our results
 in a uniform way for all Theorems \ref{thm1}, \ref{O1}, \ref{thm1.pm}, and \ref{O1.pm} together.
We first transform the operators under consideration to new operators in the conformally inverted picture. 
We then show that the assumptions of the Carleman estimates, Propositions \ref{thm.carleman_f1} and \ref{thm.carleman_f2}, are satisfied in this inverted setting.
We end the section with the (standard) proof that a Carleman estimate implies the desired vanishing.

Finally, in Appendix \ref{sec:kerr}, we discuss co-moving coordinates for the 
Kerr exteriors, which show that they fall under the assumptions of Theorem \ref{thm1.pm}.



\section{The Model Space-times} \label{sec:psc}

Since the proof of our theorems (presented in Sections \ref{sec:carleman} and \ref{sec:massloss}) somewhat conceals the role of the underlying geometry in our results, we present here some key constructions in the two basic space-times to which our theorems apply: the Minkowski and Schwarzschild space-times. 


\subsection{Pseudo-convexity in Minkowski Space-time}
\label{sec:minkowski:pseudo}


 In double null coordinates, the Minkowski metric takes the form
\begin{equation}
  g=-4\ud u\ud v+r^2\g_{AB}\ud y^A\ud y^B\,,\qquad r(u,v)=v-u\,,
\end{equation}
where the outgoing and ingoing null hypersurfaces are precisely the level sets of the optical functions $u$, $v$, respectively, and the hyperboloids $\mc{H}_C$ are expressed as 
\begin{equation}
  \label{eq:v:M:hyperboloids}
  -uv =  C > 0\,,
\end{equation}
intersecting the null infinities at the endpoints of the asymptotes $u=0$ and $v=0$; cf.~Figure~\ref{fig:minkowski:pseudoconvexity}.
We recall the function
\[ f= f_\e=\frac{1}{(-u+\e)(v+\e)} \text{,} \]
whose level sets can be thought of as  perturbations of the hyperboloid \eqref{eq:v:M:hyperboloids} which asymptote to $u=\epsilon$ and $v=-\epsilon$.

\begin{figure}[bt]
  \centering
  \includegraphics{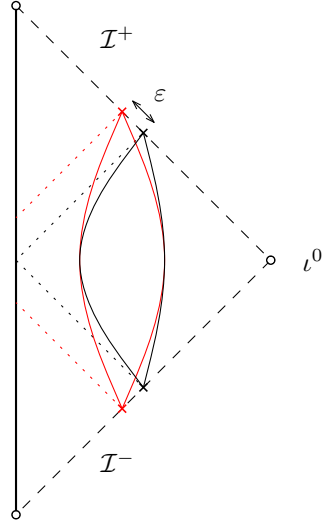}
  \caption{Pseudoconvexity in Minkowski space: The null geodesics (black) tangential to the level sets of $f$ (red) remain in the outer component.}
  \label{fig:minkowski:pseudoconvexity}
\end{figure}

We shall show that there exists a function $h$ such that
\begin{equation}
  \label{eq:v:M:pi}
  \pi=h \:g-\nabla^2 f
\end{equation}
restricted to the tangent space of the level sets of $f$ is strictly positive-definite.
This positivity encodes the pseudo-convexity and plays a key role in our Carleman estimates.
 We check the positivity of $\pi$ in a  suitable frame.

\subsubsection*{Orthonormal frame}

We introduce an orthonormal frame $(N,T,E_1,\dots, E_{n-1})$ ad-apted to the level sets of $f$. 
Let $(E_1,\dots, E_{n-1})$ be an orthonormal frame on the sphere:
\begin{equation}
   r^2\g (E_A,E_B)=\delta_{AB}.
\end{equation}
Clearly, the (timelike future-directed) unit vector tangent to the level sets of $f_\e$ and orthogonal to  $E_1,\dots, E_{n-1}$ is then
\begin{equation}
  \label{eq:v:M:frame}
  T=\frac{1}{2}\frac{1}{\sqrt{f}}\frac{1}{v+\varepsilon}\frac{\partial}{\partial u}+\frac{1}{2}\frac{1}{\sqrt{f}}\frac{1}{-u+\varepsilon}\frac{\partial}{\partial v}\,,
\end{equation}
and the (spacelike) unit normal $N$ to our level sets is given by
\begin{equation}
  N=\frac{1}{2}\frac{1}{\sqrt{f}}\frac{1}{v+\varepsilon}\frac{\partial}{\partial u}-\frac{1}{2}\frac{1}{\sqrt{f}}\frac{1}{-u+\varepsilon}\frac{\partial}{\partial v}.
\end{equation}

\subsubsection*{Pseudo-convexity}

Since
\begin{subequations}
\begin{gather}
  \h{f}{uu}=\frac{2f}{(-u+\varepsilon)^2}\,,\quad
  \h{f}{uv}=-f^2\,,\quad
  \h{f}{vv}=\frac{2f}{(v+\varepsilon)^2}\,,\\ 
  \h{f}{uA}=0\,,\quad
  \h{f}{AB}=-f^2\frac{r}{2}\bigl(r+2\varepsilon\bigr)\g_{AB}\,,
\end{gather}
\end{subequations}
we find that
\begin{gather}
  (\nabla^2 f)(T,T)=\frac{1}{2} f^2\,,\quad (\nabla^2 f)(T,N)=0\,,\quad (\nabla^2 f)(N,N)=\frac{3}{2}f^2\,,\\
  (\nabla^2 f)(E_A,E_B)=-\frac{1}{2}f^2\delta_{AB}-\frac{\varepsilon}{r}f^2\delta_{AB}\,,\quad (\nabla^2 f)(E_A, T)=0  \,.
\end{gather}
Thus, 
choosing
\begin{equation}
  \label{eq:v:M:h}
  h=-\frac{1}{2}f^2-\frac{1}{2}\frac{\varepsilon}{r}f^2\,,
\end{equation}
we ensure that the tensor $\pi$ is diagonal with respect to $(N,T,E_1,\dots, E_{n-1})$, and positive on the level sets of $f$:
\begin{gather}
  \label{eq:v:M:f:pseudoconvexity}
  \pi(T,T)=\frac{1}{2}\frac{\varepsilon}{r}f^2> 0\\
  \pi(E_A,E_B)=\frac{1}{2}\frac{\varepsilon}{r}f^2\delta_{AB} > 0\,.
\end{gather}
These formulas make precise the qualitative picture of Fig.~\ref{fig:minkowski:pseudoconvexity}, namely, that the pseudo-convexity is stronger the larger $\epsilon$ is, but also degenerates as the level sets approach the null hypersurfaces at infinity.
\footnote{In \cite{IK:illposed}, the above foliation by level sets of $f_\e$ with $\epsilon<0$ is used to prove a 
unique continuation result for an ill-posed characteristic problem for linear wave equations with continuous coefficients. 
(The function $f$ is then pseudo-convex in the opposite direction.)
Their data is prescribed on a bifurcate null hypersurface in Minkowski space, emanating from a sphere of radius $r_0>0$ on $u+v=0$. The surfaces $(-u+\e)(v+\e)=c$ for any $-r_0/2<\epsilon<0$ foliate its exterior while intersecting the null 
hypersurfaces for each value of $c>(r_0/2+\epsilon)(r_0/2+\epsilon)$.
Note that as the bifurcation sphere is shrunk to a point, $r_0\to 0$, the pseudoconvexity of the foliation degenerates,
 as $\epsilon\to 0$ by construction. The authors of \cite{IK:illposed} have a separate unique continuation result for the case $\epsilon=0$, 
 which they generously shared with us, \cite{ik:private}.}

  Since the pseudo-convexity as defined in (\ref{pseudo.cvx}) is a conformally invariant property,
  inspired by the conformal inversion of Minkowski space-time, we define:
\begin{equation}
  \label{eq:v:M:inversion}
  \bar{g}=f^2 g.
\end{equation}
Upon introducing the new coordinates 
$U=-(-u+\e)^{-1}$, $V=(v+\e)^{-1}$,
we find
\begin{equation}
  \label{eq:v:M:inverted}
  \bar{g}=-4\,\ud U\ud V+f^2r^2\g\,,
\end{equation}
and see that the conformal transformation \eqref{eq:v:M:inversion} in fact represents an inversion that maps the level sets of $f$ to hyperbolas in the $(U,V)$-plane, cf.~Figure~\ref{fig:minkowski:inversion}:
\begin{equation}
  f=-UV\,.
\end{equation}

\begin{figure}[bt]
  \centering
  \includegraphics{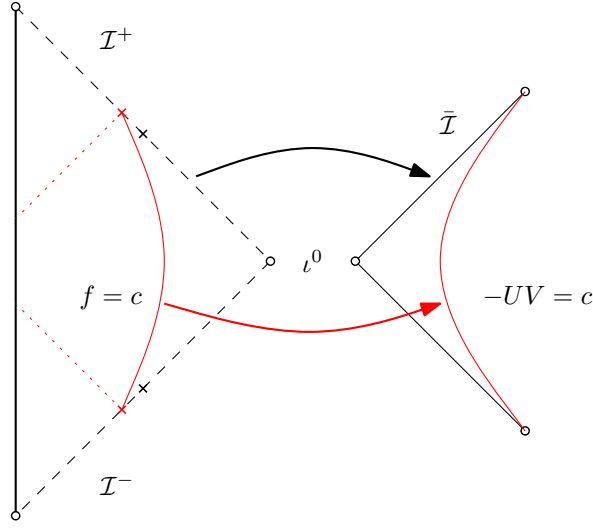}
  \caption{Depiction of the warped inversion of Minkowski space.}
  \label{fig:minkowski:inversion}
\end{figure}

\subsubsection*{Pseudo-convexity in the Inverted Picture.}

We recall the transformation law of the Hessian of a  function under conformal rescalings.
In general, if $\ol{g}=\Omega^2g$ and $\ol{\nabla},\nabla$ are the Levi-Civita connections corresponding to $\ol{g}, g$, respectively, then: 
 \begin{multline}
  \label{eq:hessian:transformation}
  \ol{\nabla}^2_{\mu\nu}f=\h{f}{\mu\nu}\\+\frac{\Omega^2}{2}\,\partial_\mu\bigl(\Omega^{-2}\bigr)\,\partial_\nu f+\frac{\Omega^2}{2}\,\partial_\nu\bigl(\Omega^{-2}\bigr)\,\partial_\mu f-\frac{\Omega^2}{2}\,g^{\alpha\beta}\partial_\alpha\bigl(\Omega^{-2}\bigr)(\partial_\beta f)\,g_{\mu\nu}\;.
\end{multline}
Thus, evaluating this formula on tangential null vectors $X$, $X(f)=0$, we see that the level sets of $f$ are {\it still} pseudoconvex with respect to $\ol{g}$.
Furthermore, given the conformal inversion \eqref{eq:v:M:inversion}, we can simply define
\begin{equation}
  \ol{h}:=h f^{-2}+1
\end{equation}
to ensure that $\ol{\pi}=\ol{h}\,\ol{g}-\ol{\nabla}^2f$ is positive when restricted to the level sets of $f$. 
In fact, with respect to the orthonormal frame for $\ol{g}$,
\begin{equation}
  \bar{N}:=f^{-1}N\,,\quad \bar{T}:=f^{-1}T\,,\quad \bar{E}_A:=f^{-1}E_A
\end{equation}
we then immediately obtain:
\begin{equation}
  \ol{\pi}(\ol{T},\ol{T})=\frac{1}{2}\frac{\varepsilon}{r}\,,\quad
  \ol{\pi}(\ol{E}_A,\ol{E}_B)=\frac{1}{2}\frac{\varepsilon}{r}\delta_{AB}\,.
\end{equation} 
Moreover all off-diagonal terms of $\ol{\pi}$ vanish in this frame.

\subsection{Pseudo-convexity in the Schwarzschild Exterior}

We now discuss the pseudo-convexity properties of the function $f_{r_0}$ defined in \eqref{fr} in the Schwarzschild space-times. 
It turns out that the behavior of null geodesics is here substantially different, yielding \emph{stronger} (pseudo-)convexity for the level sets of $f_{r_0}$, and ultimately a stronger Carleman estimate.

\subsubsection*{Inversion.}

Recall the metric \eqref{eq:schwarzschild:intro} for the Schwarzschild exterior in double null coordinates.
Here, $u=0$ and $v=0$ are fixed by a choice of $r_0>2m_S$, and
\begin{equation}
  \label{eq:v:S:f}
 f= f_{r_0}=\frac{1}{(-u)v}.
\end{equation}
We now consider the inverted metric
\begin{equation}
  \label{eq:v:S:g:c}
  \gb=\frac{1}{1-\frac{2m}{r}}f^2g\,,
\end{equation}
on the domain $\mathcal{D}=\{(u,v)\::u<0\,,v>0\}$.

Introducing coordinates $U=u^{-1}$, $V=v^{-1}$, 
the metric $\bar{g}$ takes the form:
\begin{equation}
  \label{eq:v:S:g:i}
  \gb=-4\,\ud U\ud V+ \frac{f^2 r^2}{1-\frac{2m}{r}} \cdot \g\,, 
\end{equation}
and $f$ takes the same form as in the Minkowski setting:
\begin{equation}  \label{eq:v:S:f:i}  
   f=-UV\,.
\end{equation}
We now proceed to show in the inverted space-time that the level sets of $f$ are indeed pseudo-convex.

\subsubsection*{Hessian of $f$}

We  calculate:
\begin{gather}
  \Hb{UU}f=0\,,\quad
  \Hb{UV}f=-1\,,\quad
  \Hb{VV}f=0\,,\quad
  \Hb{U A}f=0\,,\quad\Hb{V A}f=0\,,
\end{gather}
while the non-trivial contribution is now contained in:
\begin{equation}
  \begin{split}
  \Hb{AB}{f}
  &=\frac{1}{2}(U-V)\,fr\,\g_{AB}+\frac{f^2r^2}{1-\frac{2m}{r}}\g_{AB}+\frac{1}{4}\frac{f^2r^2}{1-\frac{2m}{r}}\frac{2m}{r^2}\Bigl(\frac{1}{V}-\frac{1}{U}\Bigr)\\
  &=-\frac{1}{2}\frac{\rs}{r}\gb_{AB}+\gb_{AB}+\frac{3}{4}\frac{2m}{r}\frac{\rs}{r}\gb_{AB}\,,\qquad r^\ast=v-u\,.
\end{split}
\end{equation}
\subsubsection*{Orthonormal Frame}

 We construct a frame $(\bar{N},\bar{T}, \bar{E}_1,\dots, \bar{E}_{n-1})$ which is analogous to the one 
 in the (inverted) Minkowski space-time: 
\begin{equation}
  \Nb=\frac{1}{2}\frac{1}{\sqrt{f}}\Bigl[U\,\partial_{U}+V\,\partial_{V}\Bigr]\qquad \Tb=\frac{1}{2}\frac{1}{\sqrt{f}}\Bigl[-U\,\partial_{U}+V\,\partial_{V}\Bigr]
\end{equation}
and $\bar{E}_i$ tangent to the spheres $\mc{S}_{U,V}$ 
satisfying  
\begin{equation}
  \gb(\bar{E}_A,\bar{E}_B)=R^2\g(\bar{E}_A,\bar{E}_B)=\delta_{AB}\,.
\end{equation}
Then we obtain: 
\begin{subequations}
\begin{gather}
  \Hb{\Tb\Tb}f=-\frac{1}{2}\,,\qquad\Hb{\Nb\Nb}f=\frac{1}{2}\,,\qquad 
  \Hb{\Tb\Nb}f=0\label{eq:v:S:Hf:TN}\,,\\
  \Hb{\bar{E}_A\bar{E}_B}f=-\frac{1}{2}\frac{\rs}{r}\delta_{AB}+\delta_{AB}+\frac{3}{4}\frac{2m}{r}\frac{\rs}{r}\delta_{AB}\,,\qquad \Hb{\bar{E}_A\Tb}=0\,.
\end{gather}
\end{subequations}

\subsubsection*{Pseudoconvexity}

Let then
\begin{gather}
  \bar{h}=\frac{1}{2}-\frac{1}{4}\Bigl(\frac{\rs}{r}-1\Bigr)+\frac{3}{8}\frac{2m}{r}\frac{\rs}{r}\,,\\  
  \bar{\pi}= \bar{h} \gb-\Hb{}f\,. \label{eq:v:S:pi}
\end{gather}
As previously discussed, the positivity of $\bar{\pi}$ restricted to the orthogonal complement of $\bar{N}$ signals the pseudo-convexity of the level sets of $f$.
We find
\begin{gather}
  \pi_{\Tb\Tb}=\frac{1}{4}\frac{2m}{r}\log\lvert r-2m\rvert-\frac{3}{8}\frac{1}{r}\Bigl(\frac{2m}{r}\rs+\frac{2}{3}r_0^\ast\Bigr)\\
  \pi_{AB}=\frac{1}{4}\frac{2m}{r}\log\lvert r-2m\rvert\:\delta_{AB}-\frac{3}{8}\frac{1}{r}\Bigl(\frac{2m}{r}\rs+\frac{2}{3}r_0^\ast\Bigr)\delta_{AB}\,.
\end{gather}
In particular, irrespective of how large $r_0^\ast$ is chosen, all tangential components are positive for $r$ large enough (depending on the choice of $r_0^\ast$).
This is what allows  unique continuation to hold in an arbitrarily small neighborhood of spatial infinity 
for Schwarzschild space-times of positive mass $m_{\rm S} > 0$.

\subsection{Reparametrizations and Carleman Estimates}

For both of our model spacetimes, after the conformal inversion, we have a metric of the common form
\begin{equation} \label{inverted.common} \bar{g} = - 4 \ud U \ud V + \mc{R}^2 \g_{AB} \ud y^A \ud y^B \text{,} \end{equation}
while the foliating functions $f_\e$ and $f_{r_0}$ are now given by $f = -UV$.

Consider a smooth function $\phi$ on the (inverted) space-time.
The process for deriving the Carleman estimates for $\phi$ roughly follows the geometric method in \cite{IK:illposed}: we can compare this to an energy estimate for $\psi = e^{ -\lambda F (f) } \phi$, where $F (f)$ denotes an appropriate reparametrization of the level sets of $f$.
More precisely, one integrates the divergence of a modified energy current,
\begin{equation}
  \label{eq:v:S:current}
  J^w_\beta [\psi]= Q_{\alpha\beta}[\psi] \bar{\nabla}^\alpha f +\frac{1}{2} ( \partial_\beta w) \cdot \psi^2 -\frac{1}{2} w \cdot \partial_\beta (\psi^2) + P^\flat_\beta \text{,}
\end{equation}
where $Q [\psi]$ denotes the standard energy-momentum tensor for the wave equation (see \eqref{eql.Q}).
In particular, one utilizes the gradient of $f$ as a multiplier vector field in \eqref{eq:v:S:current}.
In contrast to the usual energy estimates, here we wish for \emph{the bulk terms of the integrals to be positive} and for \emph{the boundary terms to vanish}.

By choosing $w$ in \eqref{eq:v:S:current} appropriately, depending on $h$ in the preceding discussions (see \eqref{eq.carleman_f_psf}) and on $f$, the divergence of \eqref{eq:v:S:current} will produce precisely the tensors $\bar{\pi}$ from the previous discussions, capturing the pseudo-convexity of the level sets of $f$.
This pseudo-convexity produces positive bulk terms that are quadratic in $\bar{\nabla}_{ \bar{T} } \psi$ and $\bar{\nabla}_{ \bar{E}_a } \psi$, i.e., the derivatives of $\psi$ in directions tangent to the level sets of $f$.

On the other hand, to obtain positivity for the normal derivative $\bar{\nabla}_{ \bar{N} } \psi$ and for $\psi$ itself, one relies on the choice of reparametrization $F (f)$ of $f$.
From computations, one can see that at least a logarithmic blowup of $F (f)$ as $f \searrow 0$ is necessary.
However, the apparent choice $F (f) := \log f$ (which would correspond to the decaying potential cases of Theorems \ref{thm1} and \ref{thm1.pm}) does \emph{not} suffice to produce sufficiently positive weights; to obtain the desired results, one adds a small extra acceleration to the above $F (f)$; see \eqref{eq.F1}.
We also note that the bounded potential cases of Theorems \ref{O1} and \ref{O1.pm} correspond to the reparametrization $F (f) = - f^{-2/3}$.

The above ultimately results in an inequality of the form
\begin{align}
\label{eq.model_algebraic_integral} \int_{ \mc{D}_{\omega^\prime} } \mc{W}_L | \mc{L} \psi |^2 &\geq C \lambda \int_{ \mc{D}_{\omega^\prime} } \left( \mc{W}_N | \nablai_{ \bar{N} } \psi |^2 + \mc{W}_T | \bar{\nabla}_{ \bar{T} } \psi |^2 + \mc{W}_T \sum_{a = 1}^{n - 1} | \nablai_{ \bar{E}_a } \psi |^2 \right) \\
\notag &\qquad + C \lambda^3 \int_{ \mc{D}_{\omega^\prime} } \mc{W}_0 \cdot \psi^2 + \int_{ \mc{D}_{\omega^\prime} } \mc{E} \text{,}
\end{align}
where $\mc{D}_{\omega^\prime}$ denotes the region $\{ 0 < f < \omega^\prime \}$ for some sufficiently small $\omega^\prime > 0$, and where $\mc{L}$ is the conjugated wave operator
\[ \mc{L} = e^{- \lambda F (f) } \bar{\Box} e^{ \lambda F (f) } \text{.} \]
Moreover, $\mc{W}_L$, $\mc{W}_N$, $\mc{W}_T$, and $\mc{W}_0$ are positive weights that depend on the pseudo-convexity of the level sets of $f$ and the reparametrization $F$.
The only term in \eqref{eq.model_algebraic_integral} that is not positive is the integral over the ``error" $\mc{E}$, which must be absorbed into the remaining positive terms.
That this is possible depends largely on the specific common forms for $\bar{g}$ and $f$ in the inverted settings.


By expressing $\psi$ back in terms of $\phi$, and by appropriately controlling the error terms, we obtain the desired Carleman estimates.
The exact estimate depends on the amount of pseudo-convexity in the level sets of $f$,\footnote{This is encoded in the function $\Psi$ in Propositions \ref{thm.carleman_f1} and \ref{thm.carleman_f2}.}
 as well as on the chosen reparametrization $F$.
For example, in the case of wave equations on Schwarzschild space-times with decaying potential, we have the following:
\begin{align*}
\int_{ \mc{D}_{ \omega^\prime } } f^{-2 \lambda + 1} e^{2 \lambda f^p} | \bar{\Box} \phi |^2 &\gtrsim \lambda \int_{ \mc{D}_{ \omega^\prime } } f^{-2 \lambda + 1} e^{2 \lambda f^p} \cdot f^{-1 + p} | \bar{\nabla}_{ \bar{N} } \phi |^2 \\
\notag &\qquad + \lambda \int_{ \mc{D}_{ \omega^\prime } } f^{-2 \lambda + 1} e^{2 \lambda f^p} \cdot \frac{\log r}{f r} \left( | \bar{\nabla}_{ \bar{T} } \phi |^2 + \sum_{a = 1}^{n - 1} | \bar{\nabla}_{ \bar{E}_a } \phi |^2 \right) \\
\notag &\qquad + \lambda^3 \int_{ \mc{D}_{ \omega^\prime } } f^{-2 \lambda + 1} e^{2 \lambda f^p} \cdot f^{-2 + p} \phi^2 \text{.}
\end{align*}
Here, $p$ is a small positive constant; see \eqref{eq.F1}.

Finally, we remark that the vanishing assumption required on $\phi$ in order for the relevant boundary terms to vanish depends again on the choice of $F$.
In the case $F (f) = \log f + \text{correction}$ (for Theorems \ref{thm1} and \ref{thm1.pm}), the requirement is that $\phi$ vanishes at a \emph{superpolynomial} rate in terms of $f$.
Similarly, when $F (f) = - f^{-2/3}$ (for Theorems \ref{O1} and \ref{O1.pm}), then $\phi$ must vanish at a \emph{superexponential} rate.

\section{Carleman Estimates} \label{sec:carleman}

In this section, we establish the general Carleman estimates that will be used to prove our results for all the spacetimes under consideration in this paper.

Although the setting we introduce is abstract, in order to prove all our results in a uniform way, the reader should keep in mind that the space-times we consider will be \emph{conformal inversions} of the original, physical space-times to which Theorems \ref{thm1}, \ref{O1}, \ref{thm1.pm}, and \ref{O1.pm} refer.
The key point behind the general estimates in this section is that the machinery outlined in Section \ref{sec:psc} is \emph{robust}, in the sense that the analysis goes through for sufficiently mild perturbations of $\bar{g}$ of the form \eqref{inverted.common}.
The closeness properties are expressed in terms of specially adapted frames.

\subsubsection*{Preliminaries}

Our estimates will be for the wave operator $\Boxi = \Boxi_\gi$, for an incomplete $(n+1)$-dimensional Lorentz manifold $(\mc{D}, \gi)$.
As discussed earlier, these are weighted $L^2$-estimates; a key ingredient of the weight will be a pseudo-convex function $f \in \smooth{\mc{D}}$.
We assume that the level sets of $f$ are timelike, i.e.,
\begin{equation} \label{eq.f_timelike} \nablai^\alpha f \nablai_\alpha f = \gi^{\alpha\beta} \nablai_\alpha f \nablai_\beta f > 0 \text{.} \end{equation}

Moreover, we assume at each point of $\mc{D}$, there is a local frame, $(\Ei{0}, \dots, \Ei{n})$, which is ``adapted to $f$", that is:
\begin{itemize}
\item The frame is orthonormal, with
\begin{equation} \label{eq.carleman_f_m} g ( \Ei{\alpha}, \Ei{\beta} ) = m_{\alpha\beta} \text{,} \qquad [ m_{\alpha\beta} ]_{\alpha, \beta = 0}^n := \operatorname{diag} (-1, 1, \dots, 1) \text{.} \end{equation}

\item $\Ei{0}, \dots, \Ei{n-1}$ are tangent to the level sets of $f$, and $\Ei{n}$ (which is normal to the level sets of $f$), satisfies $\Ei{n} f > 0$.
\end{itemize}
In particular, these are analogues of the frames $(\bar{T}, \bar{E}_1, \dots, \bar{E}_{n-1}, \bar{N})$ in Section \ref{sec:psc}.
They also provide a natural way to measure tensor fields on $(\mc{D}, \gi)$:
\begin{itemize}
\item Given a vector field $X$ on $\mc{D}$, we define
\begin{equation} \label{eq.norm_vf} | X |^2 := \sum_{\alpha = 0}^n [ g ( X, \Ei{\alpha} ) ]^2 \text{.} \end{equation}

\item Similarly, for a covariant $k$-tensor $A$ on $\mc{D}$, we define
\begin{equation} \label{eq.norm_cov} | A |^2 := \sum_{\alpha_1, \dots, \alpha_k = 0}^n | A ( \Ei{\alpha_1}, \dots, \Ei{\alpha_k} ) |^2 \text{.} \end{equation}
\end{itemize}

For technical reasons related to the vanishing assumptions required for our Carleman estimates, we make the following definition: 
a sequence $(D_k)$ of compact subsets of $\mc{D}$ is called an \emph{exhaustion of $\mc{D}$} if the following conditions hold:
\begin{itemize}
\item The $D_k$'s are increasing: $D_k \subset D_{k+1}$ for each $k$.

\item For each $n$, the boundary $\partial D_k$ can be written as a finite union of smooth space-like and time-like 
 hypersurfaces of $\mc{D}$. 
\end{itemize}

Finally, given a function $w \in \smooth{\mc{D}}$, we define the corresponding quantity
\begin{equation} \label{eq.carleman_f_psf} h^{ (w) } = h := w + \frac{1}{2} \Box f - \frac{n-1}{4} \in \smooth{\mc{D}} \text{,} \end{equation}
In particular, $h$ will be the factor connected to the pseudo-convexity of the level sets of $f$, in that we wish for the restriction of $h \cdot \gi - \nablai^2 f$ to the level sets of $f$ to be nonnegative-definite (see the discussion in Section \ref{sec:psc}).

\subsubsection*{The Main Estimates}

With the above background and definitions in place, we are now prepared to state our two main Carleman estimates.
In the statements below, and also in the upcoming proofs, we will use the notation $A \simeq B$ to mean that $A \leq c B$ and $B \leq c A$ for some constant $c > 0$.

The first Carleman estimate corresponds to functions vanishing superpolynomially with respect to $f$; this is used for Theorems \ref{thm1}, and \ref{thm1.pm}.

\begin{proposition} \label{thm.carleman_f1}
Let $(\mc{D}, \gi)$ and $f$ be as above, with $f$ sufficiently small on $\mc{D}$, and fix orthonormal frames $(\Ei{0}, \dots, \Ei{n})$ adapted to $f$, in the above sense, which cover all of $\mc{D}$.
Furthermore, fix a constant $p > 0$, and fix $\Psi, w \in \smooth{\mc{D}}$, with $\Psi$ satisfying
\begin{equation} \label{eq.carleman_f1_ass_gap} 0 \leq \Psi \ll f^p \text{.} \end{equation}
Assume the following conditions hold:
\begin{itemize}
\item For any vector field $X$ tangent to the level sets of $f$,
\begin{equation} \label{eq.carleman_f1_ass_pseudoconvex} ( - \nablai^2 f + h \cdot \gi )(X, X) \simeq \Psi \cdot | X |^2 \text{.} \end{equation}

\item The following bounds hold for $f$:
\footnote{In particular, for the model spacetimes of Section \ref{sec:psc}, the left-hand side of \eqref{eq.carleman_f1_ass_f_strong} vanishes entirely, as does the quantity $| \nabla^2_{ \Ei{n} \Ei{n} } f - \frac{1}{2} |$ in \eqref{eq.carleman_f1_ass_f_weak}.}
\begin{align}
\label{eq.carleman_f1_ass_f_strong} | f^{-\frac{1}{2}} \nablai_{\Ei{n}} f - 1 | + \sum_{i = 0}^{n-1} | \nablai^2_{\Ei{n} \Ei{a}} f | \ll \Psi \text{,} \\
\label{eq.carleman_f1_ass_f_weak} \left| \Boxi f - \frac{n+1}{2} \right| + \left| \nablai^2_{\Ei{n} \Ei{n}} f - \frac{1}{2} \right| \ll f^p \text{.}
\end{align}

\item $w$ satisfies the following estimates:
\begin{equation} \label{eq.carleman_f1_ass_w} | w | \ll f^p \text{,} \qquad | \Boxi w | \lesssim f^{p - 1} \text{.} \end{equation}
\end{itemize}
Let $\phi \in \smooth{\mc{D}}$, such that it satisfies the following vanishing condition:
\begin{itemize}
\item For each $N > 0$, there exists an exhaustion $(D_k)$ of $\mc{D}$ such that, if $\nu_k$ is a unit normal for the components of $\partial D_k$, then
\footnote{Here, the integral is with respect to the volume forms of the components of $\partial D_k$.}
\begin{equation} \label{eq.carleman_f1_ass_vanishing} \lim_{k \nearrow \infty} \int_{ \partial D_k } f^{-N} e^{N f^p} ( | \nu_k | + | \nablai_{ \nu_k } w | ) ( \phi^2 + | \nablai \phi |^2 ) = 0 \text{.} \end{equation}
\end{itemize}
Then, for sufficiently large $\lambda > 0$, the following estimate holds for $\phi$:
\begin{equation} \label{eq.carleman_f1} \begin{aligned}
\int_{ \mc{D} } f^{-2 \lambda + 1} e^{2 \lambda f^p} \cdot | \Box \phi |^2 &\gtrsim \lambda \int_{ \mc{D} } f^{-2 \lambda + 1} e^{2 \lambda f^p} \cdot f^{p - 1} | \nablai_{ \Ei{n} } \phi |^2 \\
&\qquad + \lambda \int_{ \mc{D} } f^{-2 \lambda + 1} e^{2 \lambda f^p} \cdot f^{-1} \Psi \sum_{i = 0}^{n - 1} | \nablai_{ \Ei{i} } \phi |^2 \\
&\qquad + \lambda^3 \int_{ \mc{D} } f^{-2 \lambda + 1} e^{2 \lambda f^p} \cdot f^{-2 + p} \phi^2 \text{.}
\end{aligned} \end{equation}
\end{proposition}

The next Carleman estimate corresponds to functions vanishing superexponentially with respect to $f$; this is used for Theorems \ref{O1} and \ref{O1.pm}.

\begin{proposition} \label{thm.carleman_f2}
Let $(\mc{D}, \gi)$ and $f$ be as above, with $f$ sufficiently small on $\mc{D}$, and fix orthonormal frames $(\Ei{0}, \dots, \Ei{n})$ adapted to $f$, in the above sense, which cover all of $\mc{D}$.
Furthermore, fix a constant $q > 0$, and fix $\Psi, w \in \smooth{\mc{D}}$, with $\Psi$ satisfying
\begin{equation} \label{eq.carleman_f2_ass_gap} 0 \leq \Psi \ll 1 \text{,} \end{equation}
Assume the following conditions hold:
\begin{itemize}
\item For any vector field $X$ tangent to the level sets of $f$,
\begin{equation} \label{eq.carleman_f2_ass_pseudoconvex} ( - \nabla^2 f + h \cdot \gi )(X, X) \simeq \Psi \cdot | X |^2 \text{.} \end{equation}

\item The following bounds hold for $f$:
\begin{align}
\label{eq.carleman_f2_ass_f_strong} f^{-q} | f^{-\frac{1}{2}} \nablai_{\Ei{n}} f - 1 | + \sum_{i = 0}^{n-1} | \nablai^2_{\Ei{n} \Ei{a}} f | \ll \Psi \text{,} \\
\label{eq.carleman_f2_ass_f_weak} \left| \Boxi f - \frac{n+1}{2} \right| + \left| \nablai^2_{\Ei{n} \Ei{n}} f - \frac{1}{2} \right| \ll 1 \text{.}
\end{align}

\item $w$ satisfies the following estimates:
\begin{equation} \label{eq.carleman_f2_ass_w} | w | \ll 1 \text{,} \qquad | \Boxi w | \lesssim f^{-2q - 1} \text{.} \end{equation}
\end{itemize}
Let $\phi \in \smooth{\mc{D}}$, such that it satisfies the following vanishing condition:
\begin{itemize}
\item For each $N > 0$, there exists an exhaustion $(D_k)$ of the boundary of $\mc{D}$ such that, if $\nu_k$ is a unit normal for the components of $\partial D_k$, then
\begin{equation} \label{eq.carleman_f2_ass_vanishing} \lim_{k \nearrow \infty} \int_{ \partial D_k } e^{-N f^{-q}} ( | \nu_k | + | \nablai_{ \nu_k } w | ) ( \phi^2 + | \nablai \phi |^2 ) = 0 \text{.} \end{equation}
\end{itemize}
Then, for sufficiently large $\lambda > 0$, the following estimate holds for $\phi$: 
\begin{equation} \label{eq.carleman_f2} \begin{aligned}
\int_{ \mc{D} } f^{q + 1} e^{2 \lambda f^{-q}} \cdot | \Box \phi |^2 &\gtrsim \lambda \int_{ \mc{D} } f^{q + 1} e^{2 \lambda f^{-q}} \cdot f^{-q - 1} | \nablai_{ \Ei{n} } \phi |^2 \\
&\qquad + \lambda \int_{ \mc{D} } f^{q + 1} e^{2 \lambda f^{-q}} \cdot f^{-q - 1} \Psi \sum_{i = 0}^{n - 1} | \nablai_{ \Ei{i} } \phi |^2 \\
&\qquad + \lambda^3 \int_{ \mc{D} } f^{q + 1} e^{2 \lambda f^{-q}} \cdot f^{-3q - 2} \phi^2 \text{.}
\end{aligned} \end{equation}
\end{proposition}

\subsection{Proof of the Estimates I: Preliminary Bounds} \label{sect:carleman_proof_1}

The remainder of this section is focused on the proofs of Propositions \ref{thm.carleman_f1} and \ref{thm.carleman_f2}.
Here, we establish some preliminary estimates that will be essential later.
The actual Carleman estimates themselves will be derived in Section \ref{sect:carleman_proof_2}.

We will  prove both propositions simultaneously.
This can be accomplished by working with general reparametrizations of $f$.
To extract Propositions \ref{thm.carleman_f1} and \ref{thm.carleman_f2}, we need only consider the specific reparametrizations $F = F (f)$ of $f$ corresponding to those settings, which we describe below.

\subsubsection{Reparametrizations}

To prove Proposition \ref{thm.carleman_f1}, we define
\begin{equation} \label{eq.F1} F = F_1 := \log f - f^p \text{.} \end{equation}
Letting $\prime$ denote differentiation with respect to $f$, then
\begin{equation} \label{eq.F1_deriv} F_1 \simeq \log f \text{,} \qquad F_1^\prime = f^{-1} - p f^{p - 1} \text{,} \qquad F_1^\prime \simeq f^{-1} \text{,} \end{equation}
as long as $f$ is sufficiently small.
On the other hand, for Proposition \ref{thm.carleman_f2}, we define
\begin{equation} \label{eq.F2} F = F_2 := - f^{-q} \text{.} \end{equation}
Observe that
\begin{equation} \label{eq.F2_deriv} F_2^\prime = q f^{-q - 1} \text{.} \end{equation}
For the rest of this section, we let $F$ be either $F_1$ or $F_2$, corresponding to the proof of Proposition \ref{thm.carleman_f1} or \ref{thm.carleman_f2}, respectively.
Observe:

\begin{lemma} \label{thml.F_est}
Both choices of $F$ satisfy
\begin{equation} \label{eql.F_est} f F^\prime \gtrsim 1 \text{,} \qquad f | F^{\prime\prime} | \lesssim F^\prime \text{.} \end{equation}
Furthermore, for sufficiently large $\lambda > 0$,
\begin{equation} \label{eql.F_exp_est} e^{-F} \geq 1 \text{,} \qquad f^{-1} \lesssim F^\prime \lesssim e^{- \lambda F} \text{.} \end{equation}
\end{lemma}

Throughout our proof, we will also refer to the auxiliary function
\begin{equation} \label{eql.G} G := - ( f F^\prime )^\prime \text{.} \end{equation}
In particular, when $F$ is either $F_1$ or $F_2$, then $G$ is, respectively,
\begin{equation} \label{eql.G12} G_1 = p^2 f^{p - 1} \text{,} \qquad G_2 = q^2 f^{-q - 1} \text{.} \end{equation}
Note that in both cases $G = G_1$ or $G = G_2$, we have the following properties:

\begin{lemma} \label{thml.G_est}
Both choices of $G$ satisfy
\begin{equation} \label{eql.G_est} 0 < G \lesssim F^\prime \text{.} \end{equation}
Furthermore, $\Psi$ is related to $F$ and $G$ (in both cases) via the following estimates:
\begin{equation} \label{eql.Psi_est} F^\prime \Psi \ll G \text{,} \qquad \Psi \ll 1 \text{,} \qquad \Psi \ll f G \lesssim f F^\prime \text{.} \end{equation}
\end{lemma}

In addition, in terms of the above language, the assumptions \eqref{eq.carleman_f1_ass_pseudoconvex}-\eqref{eq.carleman_f1_ass_w} and \eqref{eq.carleman_f2_ass_pseudoconvex}-\eqref{eq.carleman_f2_ass_w} for Propositions \ref{thm.carleman_f1} and \ref{thm.carleman_f2} imply the following:

\begin{proposition} \label{thml.carleman_f_ass}
Assuming the hypotheses of Proposition \ref{thm.carleman_f1} and \ref{thm.carleman_f2}, and setting $F$ and $G$ accordingly as above, then:
\begin{align}
\label{eql.carleman_f_ass_pseudoconvex} ( - \nablai^2 f + h \cdot \gi )(X, X) &\simeq \Psi | X |^2 \text{,} \\
\label{eql.carleman_f_ass_strong} F^\prime | f^\frac{1}{2} \nablai_{ \Ei{n} } f - f | + \sum_{i = 0}^{n-1} | \nablai^2_{\Ei{n} \Ei{i}} f | &\ll \Psi \text{,} \\
\label{eql.carleman_f_ass_weak} F^\prime \left| \Boxi f - \frac{n+1}{2} \right| + F^\prime \left| \nablai^2_{\Ei{n} \Ei{n}} f - \frac{1}{2} \right| &\ll G \text{,} \\
\label{eql.carleman_f_ass_w} F^\prime |w| \ll G \text{,} \qquad | \Boxi w | &\lesssim f F^\prime G \text{.}
\end{align}
Furthermore, the vanishing conditions \eqref{eq.carleman_f1_ass_vanishing} and \eqref{eq.carleman_f2_ass_vanishing} can be aggregated as
\begin{equation} \label{eql.carleman_f_ass_vanishing} \lim_{k \nearrow \infty} \int_{ \partial D_k } e^{-N F} ( | \nu_k | + | \nablai_{ \nu_k } w | ) ( \phi^2 + | \nablai \phi |^2 ) = 0 \text{,} \qquad N > 0 \text{.} \end{equation}
\end{proposition}

\begin{remark}
Note that metrically equivalent tensor fields---e.g., a vector field $T^\alpha$ and the corresponding one-form $T_\alpha := \gi_{\alpha\beta} T^\beta$---have the same tensor norm (as defined in \eqref{eq.norm_vf} and \eqref{eq.norm_cov}).
In particular, this is true for the differential $\nabla_\alpha \psi = \partial_\alpha \psi$ of a scalar $\psi \in \smooth{\mc{D}}$ and its $\gi$-gradient $\nablai^\alpha \psi = \gi^{\alpha\beta} \nabla_\beta \psi$.
\end{remark}

\subsubsection{Estimates for $f$}

The first task is to obtain preliminary estimates for $f$.
For convenience, we will use throughout the proof the abbreviations
\begin{equation} \label{eql.ell} \ell := \nablai^\alpha f \nablai_\alpha f \text{,} \qquad \bar{\ell} := \frac{1}{2} \nablai^\alpha f \nablai_\alpha \ell = \nablai^\alpha f \nablai^\beta f \nablai^2_{\alpha\beta} f \text{.} \end{equation}
Note in particular that
\[ \Ei{n} = \ell^{-\frac{1}{2}} \cdot \gradi f \text{,} \]
while $\bar{\ell}$ encodes the normal component of $\nablai^2 f$.

\begin{lemma} \label{thml.nor_decomp}
The following estimates hold:
\begin{equation} \label{eql.nor_est} | \nablai f | \simeq f^\frac{1}{2} \text{,} \qquad \ell \simeq f \text{,} \qquad F^\prime | \ell - f | \ll \Psi \text{.} \end{equation}
\end{lemma}

\begin{proof}
The last inequality in \eqref{eql.Psi_est} and \eqref{eql.carleman_f_ass_strong} imply that
\begin{equation} \label{eql.ell_est_pre} | \nablai_{ \Ei{n} } f - f^\frac{1}{2} | \ll f^{-\frac{1}{2}} ( F^\prime )^{-1} \Psi \ll f^\frac{1}{2} \text{.} \end{equation}
Since $| \nabla f | = \nablai_{ \Ei{n} } f = \ell^\frac{1}{2}$ by definition, the first two estimates in \eqref{eql.nor_est} follow immediately.
Finally, we apply \eqref{eql.ell_est_pre} and the comparison $\ell \simeq f$ to obtain
\[ | \ell - f | \leq | \ell^\frac{1}{2} - f^\frac{1}{2} | | \ell^\frac{1}{2} + f^\frac{1}{2} | \ll f^\frac{1}{2} \cdot f^{-\frac{1}{2}} (F^\prime)^{-1} \Psi = (F^\prime)^{-1} \Psi \text{,} \]
completing the proof of the final estimate in \eqref{eql.nor_est}.
\end{proof}

\begin{lemma} \label{thml.hessian_decomp}
With $F$ and $G$ as before, the following estimates hold:
\begin{equation} \label{eql.hessian_est} F^\prime \left| h - \frac{1}{2} \right| \ll G \text{,} \qquad | \nablai^2 f | \lesssim 1 \text{,} \qquad F^\prime \left| \bar{\ell} - \frac{1}{2} f \right| \ll f G \text{.} \end{equation}
\end{lemma}

\begin{proof}
First, by rewriting \eqref{eq.carleman_f_psf} as
\[ h = \frac{1}{2} + w + \frac{1}{2} \left( \Box f - \frac{n+1}{2} \right) \text{.} \]
and by applying \eqref{eql.G_est}, \eqref{eql.carleman_f_ass_weak}, and \eqref{eql.carleman_f_ass_w}, we obtain the first inequality in \eqref{eql.hessian_est}.
For $| \nablai^2 f |$, we begin by applying \eqref{eql.G_est}, \eqref{eql.Psi_est}, \eqref{eql.carleman_f_ass_strong}, and \eqref{eql.carleman_f_ass_weak}, which yields
\begin{equation} \label{eql.hessian_f_est_1} | \nablai^2_{\Ei{n} \Ei{n}} f | \lesssim \frac{1}{2} + (F^\prime)^{-1} G \lesssim 1 \text{,} \qquad \sum_{i = 0}^{n - 1} | \nablai^2_{\Ei{n} \Ei{i}} f | \lesssim \Psi \lesssim 1 \text{.} \end{equation}
For the fully tangential components of $\nablai^2 f$, we note the algebraic identity
\[ \nablai^2_{ \Ei{a} \Ei{a} } f + \nablai^2_{ \Ei{b} \Ei{b} } f + 2 \nablai^2_{ \Ei{a} \Ei{b} } f = \nablai^2_{ \Ei{a} + \Ei{b}, \Ei{a} + \Ei{b} } f \text{,} \qquad 1 \leq a, b \leq n - 1 \text{.} \]
Thus, combining the above with \eqref{eql.carleman_f_ass_pseudoconvex}, we obtain
\begin{equation} \label{eql.hessian_f_est_2} \sum_{a, b = 1}^{n-1} | \nablai^2_{\Ei{a} \Ei{b}} f | \lesssim \Psi + h \lesssim 1 \text{,} \end{equation}
where in the last step, we applied \eqref{eql.Psi_est} and the first part of \eqref{eql.hessian_est}.
From \eqref{eql.hessian_f_est_1} and \eqref{eql.hessian_f_est_2}, we obtain the second inequality in \eqref{eql.hessian_est}.

Finally, using \eqref{eql.carleman_f_ass_weak}, \eqref{eql.nor_est}, and the second inequality in \eqref{eql.hessian_est}, we estimate
\[ \left| \bar{\ell} - \frac{1}{2} f \right| \leq \ell \left| \nablai^2_{ \Ei{n} \Ei{n} } f - \frac{1}{2} \right| + \frac{1}{2} | \ell - f | \ll f (F^\prime)^{-1} G + (F^\prime)^{-1} \Psi \text{.} \]
Since $\Psi \lesssim f G$ by \eqref{eql.Psi_est}, this proves the last inequality in \eqref{eql.hessian_est}.
\end{proof}

\subsection{Proof of the Estimates II: The Main Derivation} \label{sect:carleman_proof_2}

We are now prepared to derive the Carleman estimates, \eqref{eq.carleman_f1} and \eqref{eq.carleman_f2}, in earnest. 
Here we will follow some of the nomenclature of \cite{IK:illposed}.
Throughout, we fix $\phi \in \smooth{\mc{D}}$ and $\lambda > 0$.
Moreover, for convenience, we define the following:
\begin{itemize}
\item Define $\psi \in \smooth{\mc{D}}$ and the operator $\mc{L}$ by
\begin{equation} \label{eql.psi_phi} \psi := e^{-\lambda F} \phi \text{,} \qquad \mc{L} \psi := e^{-\lambda F} \Boxi ( e^{\lambda F} \psi ) \text{.} \end{equation}

\item Define the auxiliary function $w^\prime \in \smooth{\mc{D}}$ by
\begin{equation} \label{eql.w_prime} w^\prime := w - \frac{n - 1}{4} = h - \frac{1}{2} \Boxi f \text{.} \end{equation}

\item In addition, we define the shorthands
\begin{equation} \label{eql.S} S \psi := \nablai^\alpha f \nablai_\alpha \psi \text{,} \qquad S_w \psi := S \psi - w^\prime \psi \text{.} \end{equation}

\item Let $Q$ denote the stress-energy tensor for the wave equation, applied to $\psi$:
\begin{equation} \label{eql.Q} Q_{\alpha\beta} := \nablai_\alpha \psi \nablai_\beta \psi - \frac{1}{2} \gi_{\alpha\beta} \nablai^\mu \psi \nablai_\mu \psi \text{.} \end{equation}
\end{itemize}

\subsubsection{Algebraic Expansions}

The first step is an algebraic expansion of the expression $\mc{L} \psi S_w \psi$.
From this, we can obtain a pointwise lower bound for $| \mc{L} \psi |$.

\begin{lemma} \label{thml.algebraic}
Recall the notations \eqref{eql.psi_phi}-\eqref{eql.Q}, and define also the following:
\begin{align}
\label{eql.algebraic_coeff} \Lambda = - ( F^\prime )^2 \bar{\ell} - F^\prime F^{\prime\prime} \ell^2 - ( F^\prime )^2 \ell h \text{,} &\qquad \mc{E} = 2 F^\prime h + F^{\prime\prime} \ell \text{,} \\
\notag \pi_{\alpha\beta} = -\nablai_{\alpha\beta} f + h \gi_{\alpha\beta} \text{,} &\qquad P_\beta = Q_{\alpha\beta} \nablai^\alpha f + P^\flat_\beta + P^\sharp_\beta \text{,} \\
\notag P^\sharp_\beta = - w^\prime \cdot \psi \nablai_\beta \psi + \frac{1}{2} \nablai_\beta w^\prime \cdot \psi^2 \text{,} &\qquad P^\flat_\beta = \frac{1}{2} \lambda^2 \ell ( F^\prime )^2 \nablai_\beta f \cdot \psi^2 \text{.}
\end{align}
Then, the following inequality holds:
\begin{align}
\label{eql.algebraic} ( F^\prime )^{-1} | \mc{L} \psi |^2 &\geq 3 \lambda^2 F^\prime \cdot | S_w \psi |^2 + 2 \lambda \pi_{\alpha\beta} \cdot \nablai^\alpha \psi \nablai^\beta \psi + 2 \lambda^3 \Lambda \cdot \psi^2 \\
\notag &\qquad + 2 \lambda^2 \mc{E} \cdot \psi S_w \psi - \lambda \Boxi w \cdot \psi^2 + 2 \lambda \nablai^\beta P_\beta \text{.}
\end{align}
\end{lemma}

\begin{proof}
We begin by expanding $\mc{L} \psi$ as follows:
\begin{align*}
\mc{L} \psi &= \Boxi \psi + 2 \lambda \cdot \nablai^\alpha F \nablai_\alpha \psi + e^{-\lambda F} \Boxi e^{\lambda F} \cdot \psi \\
&= \Boxi \psi + 2 \lambda F^\prime \cdot S \psi + \lambda^2 \nablai^\alpha F \nablai_\alpha F \cdot \psi + \lambda \Boxi F \cdot \psi \\
&= \Boxi \psi + 2 \lambda F^\prime \cdot S \psi + \lambda^2 ( F^\prime )^2 \ell \cdot \psi + \lambda F^{\prime\prime} \ell \cdot \psi + \lambda F^\prime \Box f \cdot \psi \\
&= \Boxi \psi + 2 \lambda F^\prime \cdot S_w \psi + \lambda^2 ( F^\prime )^2 \ell \cdot \psi + \lambda \mc{E} \cdot \psi \text{,}
\end{align*}
Multiplying the above by $S_w \psi$ yields
\begin{equation} \label{eql.algebraic_1} \mc{L} \psi S_w \psi = \Boxi \psi S_w \psi + 2 \lambda F^\prime \cdot | S_w \psi |^2 + \lambda^2 ( F^\prime )^2 \ell \cdot \psi S_w \psi + \lambda \mc{E} \cdot \psi S_w \psi \text{.} \end{equation}

Letting $\mc{A} = \lambda^2 ( F^\prime )^2 \ell$, then the product rule implies
\begin{align}
\label{eql.product_rule} \mc{A} \cdot \psi S_w \psi &= \frac{1}{2} \mc{A} \cdot \nablai^\beta f \nablai_\beta ( \psi^2 ) - \mc{A} w^\prime \cdot \psi^2 \\
\notag &= \frac{1}{2} \nablai^\beta ( \mc{A} \nablai_\beta f \cdot \psi^2 ) - \frac{1}{2} \nablai^\beta f \nablai_\beta \mc{A} \cdot \psi^2 - \mc{A} h \cdot \psi^2 \\ 
\notag &= \nablai^\beta P^\flat_\beta + \lambda^2 \Lambda \cdot \psi^2 \text{,}
\end{align}
where in the last step, we recalled \eqref{eql.algebraic_coeff}, and we observed that
\begin{align*}
- \frac{1}{2} \nablai^\beta f \nablai_\beta \mc{A} - \mc{A} h &= - \frac{1}{2} \lambda^2 \nablai^\beta f \nablai_\beta \ell \cdot ( F^\prime )^2 - \lambda^2 F^\prime \ell \cdot \nablai^\beta f \nablai_\beta ( F^\prime ) - \lambda^2 ( F^\prime )^2 \ell h \\
&= - \lambda^2 ( F^\prime )^2 \bar{\ell} - \lambda^2 F^\prime F^{\prime\prime} \ell^2 - \lambda^2 ( F^\prime )^2 \ell h \text{.}
\end{align*}
Combining \eqref{eql.algebraic_1} and \eqref{eql.product_rule}, we see that
\begin{equation} \label{eql.algebraic_2} \mc{L} \psi S_w \psi = \Boxi \psi S_w \psi + 2 \lambda F^\prime \cdot | S_w \psi |^2 + \lambda^2 \Lambda \cdot \psi^2 + \lambda \mc{E} \cdot \psi S_w \psi + \nablai^\beta P^\flat_\beta \text{.} \end{equation}

Next, recalling the stress-energy tensor $Q$ in \eqref{eql.Q}, we compute
\begin{align*}
\nablai^\beta ( Q_{\alpha\beta} \nablai^\alpha f ) &= \Boxi \psi S \psi + \nablai^2_{\alpha\beta} f \cdot \nablai^\alpha \psi \nablai^\beta \psi - \frac{1}{2} \Boxi f \cdot \nablai^\beta \psi \nablai_\beta \psi \text{,} \\
\nablai^\beta P^\sharp_\beta &= - w^\prime \cdot \psi \Boxi \psi - w^\prime \cdot \nablai^\beta \psi \nablai_\beta \psi + \frac{1}{2} \Boxi w \cdot \psi^2 \text{.}
\end{align*}
Summing the above identities, we obtain
\begin{equation} \label{eql.algebraic_3} \nablai^\beta ( Q_{\alpha\beta} \nablai^\alpha f + P^\sharp_\beta ) = \Boxi \psi S_w \psi - \pi_{\alpha\beta} \cdot \nablai^\alpha \psi \nablai^\beta \psi + \frac{1}{2} \Boxi w \cdot \psi^2 \text{.} \end{equation}
Combining \eqref{eql.algebraic_2} and \eqref{eql.algebraic_3} yields
\begin{align}
\label{eql.algebraic_4} \mc{L} \psi S_w \psi &= 2 \lambda F^\prime \cdot | S_w \psi |^2 + \pi_{\alpha\beta} \cdot \nablai^\alpha \psi \nablai^\beta \psi + \lambda^2 \Lambda \cdot \psi^2 \\
\notag &\qquad + \lambda \mc{E} \cdot \psi S_w \psi - \frac{1}{2} \Boxi w \cdot \psi^2 + \nablai^\beta P_\beta \text{.}
\end{align}
Finally, \eqref{eql.algebraic} follows immediately from \eqref{eql.algebraic_4} and the following basic inequality:
\[ \mc{L} \psi S_w \psi \leq \frac{1}{2} \lambda^{-1} ( F^\prime )^{-1} \cdot | \mc{L} \psi |^2 + \frac{1}{2} \lambda F^\prime \cdot | S_w \psi |^2 \text{.} \qedhere \]
\end{proof}

\subsubsection{Positivity and Error Estimates}

We next show that, except for the divergence term, the right-hand side of \eqref{eql.algebraic} is positive.
The first step of this process is to show that $\Lambda$ is positive and that $\mc{E}$ is appropriately bounded.
More specifically, $\Lambda$ represents the weight of the zero-order terms in the Carleman estimates, and it must absorb all the other zero-order weights in our derivation (including $\mc{E}$).

\begin{lemma} \label{thml.Lambda_E}
The following estimates hold:
\begin{equation} \label{eql.Lambda_E} \Lambda \simeq f F^\prime G \text{,} \qquad | \mc{E} | \lesssim G \text{.} \end{equation}
Moreover, for sufficiently large $\lambda$, the following inequality holds:
\begin{equation} \label{eql.algebraic_Lambda} ( F^\prime )^{-1} | \mc{L} \psi |^2 \geq 2 \lambda^2 F^\prime \cdot | S_w \psi |^2 + 2 \lambda \pi_{\alpha\beta} \cdot \nablai^\alpha \psi \nablai^\beta \psi + \lambda^3 \Lambda \cdot \psi^2 + 2 \lambda \nablai^\beta P_\beta \text{,} \end{equation}
\end{lemma}

\begin{proof}
We begin with the bound for $\mc{E}$.
From its definition in \eqref{eql.algebraic_coeff}, we can write
\begin{align}
\label{eql.E_1} \mc{E} &= F^\prime + f F^{\prime\prime} + F^\prime ( 2 h - 1 ) + F^{\prime\prime} ( \ell - f ) \\
\notag &= - G + F^\prime ( 2 h - 1 ) + F^{\prime\prime} ( \ell - f ) \text{.}
\end{align}
By \eqref{eql.F_est}, \eqref{eql.Psi_est}, \eqref{eql.nor_est}, and \eqref{eql.hessian_est}, we estimate
\[ F^\prime | 2 h - 1 | \ll G \text{,} \qquad F^{\prime\prime} | \ell - f | \lesssim f^{-1} F^\prime | \ell - f | \ll f^{-1} \Psi \ll G \text{.} \]
Combining the above inequalities with \eqref{eql.E_1} results in the inequality for $\mc{E}$ in \eqref{eql.Lambda_E}.

Next, for $\Lambda$, we expand its definition in \eqref{eql.algebraic_coeff}:
\begin{equation} \label{eql.Lambda_1} \Lambda = - \frac{1}{2} ( F^\prime )^2 f - F^\prime F^{\prime\prime} f^2 - \frac{1}{2} ( F^\prime )^2 f + E_\Lambda = f F^\prime G + E_\Lambda \text{,} \end{equation}
where the error terms $E_\Lambda$ are given by
\[ E_\Lambda = - F^\prime F^{\prime\prime} ( \ell^2 - f^2 ) - ( F^\prime )^2 \left( \bar{\ell} - \frac{1}{2} f \right) - ( F^\prime )^2 \left( \ell h - \frac{1}{2} f \right) \text{.} \]
We can then estimate $E_\Lambda$ using \eqref{eql.F_est}, \eqref{eql.nor_est}, and \eqref{eql.hessian_est}:
\begin{align*}
| E_\Lambda | &\lesssim F^\prime \cdot f F^{\prime\prime} \cdot | \ell - f | + ( F^\prime )^2 \left( \left| \bar{\ell} - \frac{1}{2} f \right| + | \ell - f | + f \left| h - \frac{1}{2} \right| \right) \\
&\ll F^\prime \Psi + f F^\prime G \text{.}
\end{align*}
Thus, \eqref{eql.Psi_est} yields $| E_\Lambda | \ll f F^\prime G$, which with \eqref{eql.Lambda_1} implies the first part of \eqref{eql.Lambda_E}.

Finally, for \eqref{eql.algebraic_Lambda}, we first observe that
\begin{equation} \label{eql.algebraic_Lambda_1} 2 \lambda^2 \mc{E} \cdot \psi S_w \psi \leq \lambda^2 F^\prime \cdot | S_w \psi |^2 + \lambda^2 ( F^\prime )^{-1} \mc{E}^2 \cdot \psi^2 \text{.} \end{equation}
Moreover, \eqref{eql.F_est}, \eqref{eql.G_est}, \eqref{eql.carleman_f_ass_w}, and \eqref{eql.Lambda_E}, we have
\begin{align*}
\lambda | \Boxi w | \cdot \psi^2 &\lesssim \lambda f F^\prime G \cdot \psi^2 \text{,} \\
\lambda^2 ( F^\prime )^{-1} \mc{E}^2 \cdot \psi^2 &\lesssim \lambda^2 ( F^\prime )^{-1} G^2 \cdot \psi^2 \lesssim \lambda^2 f F^\prime G \cdot \psi^2 \text{.}
\end{align*}
Thus, applying the above, along with the first bound in \eqref{eql.Lambda_E}, we see that
\begin{equation} \label{eql.algebraic_Lambda_2} \lambda | \Boxi w | \cdot \psi^2 + \lambda^2 ( F^\prime )^{-1} \mc{E}^2 \cdot \psi^2 \leq \lambda^3 \Lambda \cdot \psi^2 \text{,} \end{equation}
for sufficiently large $\lambda$.
Applying \eqref{eql.algebraic_Lambda_1} and \eqref{eql.algebraic_Lambda_2} to \eqref{eql.algebraic} results in \eqref{eql.algebraic_Lambda}.
\end{proof}

The remaining estimate deals with terms in \eqref{eql.algebraic_Lambda} that are quadratic in $\nablai \psi$.

\begin{lemma} \label{thml.pi_deriv}
There exist constants $\mc{C}_1, \mc{C}_2 > 0$ such that
\begin{align}
\label{eql.pi_deriv} F^\prime \cdot | S_w \psi |^2 + \Lambda \cdot \psi^2 &\geq \mc{C}_1 G \cdot | S \psi |^2 \text{,} \\
\notag G \cdot | S \psi |^2 + \pi_{\alpha\beta} \cdot \nablai^\alpha \psi \nablai^\beta \psi &\geq \mc{C}_2 f G \cdot | \nablai_{ \Ei{n} } \psi |^2 + \mc{C}_2 \Psi \cdot \sum_{i = 0}^{n - 1} | \nablai_{ \Ei{i} } \psi |^2 \text{.}
\end{align}
In particular, if $\lambda$ is sufficiently large, then there exists $\mc{C} > 0$ such that 
\begin{align}
\label{eql.algebraic_pi} ( F^\prime )^{-1} | \mc{L} \psi |^2 &\geq \mc{C} \lambda f G \cdot | \nablai_{ \Ei{n} } \psi |^2 + \mc{C} \lambda \Psi \cdot \sum_{i = 0}^{n - 1} | \nablai_{ \Ei{i} } \psi |^2 \\
\notag &\qquad + \mc{C} \lambda^3 \Lambda \cdot \psi^2 + 2 \lambda \nablai^\beta P_\beta \text{.}
\end{align}
\end{lemma}

\begin{proof}
First, note that \eqref{eql.algebraic_pi} follows by combining \eqref{eql.algebraic_Lambda} and \eqref{eql.pi_deriv} and by taking large enough $\lambda$.
Moreover, observe that, by \eqref{eql.F_est} and \eqref{eql.G_est}, we obtain
\[ G \cdot | S \psi |^2 \lesssim G \cdot | S_w \psi |^2 + G | w^\prime |^2 \cdot \psi^2 \lesssim F^\prime \cdot | S_w \psi |^2 + f F^\prime G | w^\prime |^2 \cdot \psi^2 \text{.} \]
Since $| w^\prime | \lesssim 1$ by \eqref{eql.G_est} and \eqref{eql.carleman_f_ass_w}, then the above and \eqref{eql.Lambda_E} imply the first inequality in \eqref{eql.pi_deriv}.
Thus, it remains only to prove the second inequality in \eqref{eql.pi_deriv}.

For this, we begin by decomposing
\begin{align}
\label{eql.deformation_decomp} \pi_{\alpha\beta} \cdot \nablai^\alpha \psi \nablai^\beta \psi &= \pi_{ \Ei{n} \Ei{n} } | \nablai_{ \Ei{n} } \psi |^2 + \sum_{i, j = 0}^{n-1} \pi_{ \Ei{i} \Ei{j} } \nablai_{ \Ei{i} } \psi \nablai_{ \Ei{j} } \psi \\
\notag &\qquad + 2 \sum_{i = 0}^{n-1} \pi_{ \Ei{n} \Ei{i} } \nablai_{ \Ei{n} } \psi \nablai_{ \Ei{i} } \psi \\
\notag &= \mc{K}_{\perp\perp} + \mc{K}_{\|\|} + \mc{K}_{\perp\|} \text{,}
\end{align}
By the assumption \eqref{eql.carleman_f_ass_pseudoconvex}, we have the lower bound
\[ \mc{K}_{\|\|} \gtrsim \Psi \cdot \sum_{i = 0}^{n-1} | \nablai_{ \Ei{i} } \psi |^2 \text{.} \]
Thus, by \eqref{eql.deformation_decomp} and the above, we obtain, for some $\mc{C}^\prime > 0$, that
\begin{equation} \label{eql.pi_deriv_1} \pi_{\alpha\beta} \cdot \nablai^\alpha \psi \nablai^\beta \psi \geq \mc{C}^\prime \Psi \cdot \sum_{i = 0}^{n-1} | \nablai_{ \Ei{i} } \psi |^2 - | \mc{K}_{\perp\perp} | - | \mc{K}_{\perp\|} | \text{,} \end{equation}
Moreover, recalling \eqref{eql.nor_est}, we see that
\begin{equation} \label{eql.pi_deriv_2} G \cdot | S \psi |^2 = \ell G \cdot | \nabla_{ \Ei{n} } \psi |^2 \text{.} \end{equation}

To control $\mc{K}_{\perp\|}$, we expand and bound using \eqref{eql.carleman_f_ass_strong}:
\begin{equation} \label{eql.pi_deriv_11} | \mc{K}_{\perp\|} | \lesssim \sum_{i = 0}^{n-1} | \nablai^2_{ \Ei{n} \Ei{i} } f | | \nablai_{ \Ei{n} } \psi | | \nablai_{ \Ei{i} } \psi | \ll \Psi \cdot \sum_{\alpha = 0}^n | \nablai_{ \Ei{\alpha} } \psi |^2 \text{.} \end{equation}
Furthermore, for $\mc{K}_{\perp\perp}$, we expand
\[ \mc{K}_{\perp\perp} = \left( \frac{1}{2} - \nablai^2_{ \Ei{n} \Ei{n} } f \right) \cdot | \nablai_{ \Ei{n} } \psi |^2 + \left( h - \frac{1}{2} \right) \cdot | \nablai_{ \Ei{n} } \psi |^2 \text{.} \]
Applying \eqref{eql.carleman_f_ass_weak} and \eqref{eql.hessian_est}, we then estimate
\begin{equation} \label{eql.pi_deriv_12} | \mc{K}_{\perp\perp} | \ll ( F^\prime )^{-1} G \cdot | \nablai_{ \Ei{n} } \psi |^2 \lesssim f G \cdot | \nablai_{ \Ei{n} } \psi |^2 \text{.} \end{equation}
Combining \eqref{eql.nor_est} and \eqref{eql.pi_deriv_1}-\eqref{eql.pi_deriv_12} results in the second inequality of \eqref{eql.pi_deriv}.
\end{proof}

\subsubsection{Integral and Boundary Estimates}

Having obtained a pointwise lower bound for $| \mc{L} \psi |$ from \eqref{eql.algebraic_pi}, it remains to integrate this over $\mc{D}$.
The only term on the right-hand side of \eqref{eql.algebraic_pi} which needs not be nonnegative is the divergence of $P$, which contributes  boundary terms after integration.
Here, we show, by integrating over $D_k$ and passing to the limit, that this boundary contribution vanishes.

\begin{lemma} \label{thml.boundary}
The following limit holds:
\begin{equation} \label{eql.boundary} \lim_{k \nearrow \infty} \int_{ D_k } \nablai^\beta P_\beta = 0 \text{.} \end{equation}
As a result, for sufficiently large $\lambda$, we have that
\begin{align}
\label{eql.algebraic_integral} \int_{ \mc{D} } ( F^\prime )^{-1} | \mc{L} \psi |^2 &\gtrsim \lambda \int_{ \mc{D} } \left( f G \cdot | \nablai_{ \Ei{n} } \psi |^2 + \Psi \cdot \sum_{i = 0}^{n - 1} | \nablai_{ \Ei{i} } \psi |^2 \right) \\
\notag &\qquad + \lambda^3 \int_{ \mc{D} } f F^\prime G \cdot \psi^2 \text{.}
\end{align}
\end{lemma}

\begin{proof}
Recalling the definition of $P$ in \eqref{eql.algebraic_coeff}, the divergence theorem yields
\[ \int_{ D_k } \nablai^\beta P_\beta = \int_{ \partial D_k } Q_{\alpha\beta} \nu_k^\beta \nablai^\alpha f + \int_{ \partial D_k } P^\flat_\beta \nu_k^\beta + \int_{ \partial D_k } P^\sharp_\beta \nu_k^\beta = I^1_k + I^2_k + I^3_k \text{.} \]
Thus, to prove \eqref{eql.boundary}, it suffices to show that each $I^l_k$ vanishes as $k \nearrow \infty$.

Recalling \eqref{eql.psi_phi} and applying \eqref{eql.nor_est}, we obtain
\[ | \nablai \psi | \lesssim e^{-\lambda F} \sum_{\alpha = 0}^n | \nablai_{ \Ei{\alpha} } \phi | + \lambda e^{-\lambda F} | \nablai_{ \Ei{n} } F | \cdot | \phi | \lesssim e^{-\lambda F} | \nablai \phi | + \lambda e^{-\lambda F} f^\frac{1}{2} F^\prime \cdot | \phi | \text{,} \]
By \eqref{eql.F_exp_est}, we see that for sufficiently large $\lambda$,
\begin{equation} \label{eql.boundary_psi} | \nablai \psi | \lesssim e^{-2 \lambda F} ( | \nablai \phi | + | \phi | ) \text{.} \end{equation}

For $I^1_k$, we expand $Q$ using \eqref{eql.Q} to obtain
\begin{align*}
| I^1_k | \leq \int_{ \partial D_k } | \nablai^\alpha f \nablai_\alpha \psi | | \nablai_{ \nu_k } \psi | + \frac{1}{2} \int_{ \partial D_k } | \nablai_{ \nu_k } f | | \nablai^\alpha \psi \nablai_\alpha \psi | \text{.}
\end{align*}
Applying \eqref{eql.boundary_psi}, the above becomes
\begin{equation} \label{eql.boundary_1} | I^1_k | \leq \int_{ \partial D_k } | \nablai f | | \nu_k | | \nablai \psi |^2 \lesssim \int_{ \partial D_k } e^{-4 \lambda F} | \nu_k | ( | \nablai \phi |^2 + \phi^2 ) \text{.} \end{equation}
Similarly, for $I^2_k$, we expand $P^\flat_\beta$ using \eqref{eql.algebraic_coeff} and estimate
\begin{equation} \label{eql.boundary_2} | I^2_k | \lesssim \lambda^2 \int_{ \partial D_k } f ( F^\prime )^2 | \nu_k | | \nabla f | \cdot \psi^2 \lesssim \lambda^2 \int_{ \partial D_k } e^{-4 \lambda F} | \nu_k | \cdot \phi^2 \text{,} \end{equation}
where we used \eqref{eql.F_exp_est} to control $F^\prime$.
Finally, for $I^3_k$, we expand $P^\sharp_\beta$, and use the trivial bound $| w^\prime | \lesssim 1$, which can be obtained from \eqref{eql.carleman_f_ass_w}:
\begin{align*}
| I^3_k | &\lesssim \int_{ \partial D_k } | \nu_k | \cdot | \psi | | \nablai \psi | + \int_{ \partial D_k } | \nablai_{ \nu_k } w | \cdot \psi^2 \\
&\lesssim \int_{ \partial D_k } | \nu_k | \cdot | \nablai \psi |^2 + \int_{ \partial D_k } ( | \nu_k | + | \nablai_{ \nu_k } w | ) \cdot \psi^2 \text{.}
\end{align*}
Recalling \eqref{eql.boundary_psi}, we obtain the estimate
\begin{equation} \label{eql.boundary_3} | I^3_k | \lesssim \int_{ \partial D_k } e^{-4 \lambda F} ( | \nu_k | + | \nablai_{ \nu_k } w | ) \cdot ( | \nabla \phi |^2 + \phi^2 ) \text{.} \end{equation}

Recalling the vanishing condition \eqref{eql.carleman_f_ass_vanishing}, then \eqref{eql.boundary_1}-\eqref{eql.boundary_3} imply that
\[ \lim_{ k \nearrow \infty } ( | I^1_k | + | I^2_k | + | I^3_k | ) = 0 \text{,} \]
which completes the proof of \eqref{eql.boundary}.
For \eqref{eql.algebraic_integral}, we integrate \eqref{eql.algebraic_pi} over $D_k$:
\begin{align*}
\int_{ D_k } ( F^\prime )^{-1} | L \psi |^2 &\geq \mc{C} \lambda \int_{ D_k } f G \cdot | \nablai_{ \Ei{n} } \psi |^2 + \mc{C} \lambda \int_{ D_k } \Psi \cdot \sum_{i = 0}^{n - 1} | \nablai_{ \Ei{i} } \psi |^2 \\
&\qquad + \mc{C} \lambda^3 \int_{ D_k } \Lambda \cdot \psi^2 + 2 \lambda \int_{ D_k } \nablai^\beta P_\beta \text{.}
\end{align*}
Taking a limit of this as $k \nearrow \infty$, the last term on the right-hand side vanishes by \eqref{eql.boundary}.
Moreover, by applying the monotone convergence theorem on the remaining (nonnegative) terms, we obtain the desired inequality \eqref{eql.algebraic_integral}.
\end{proof}

\subsubsection{Completion of the Proof}

Recall now from \eqref{eql.psi_phi} that
\begin{equation} \label{eql.unconjugate_1} \psi = e^{-\lambda F} \phi \text{,} \qquad \mc{L} \psi = e^{-\lambda F} \Box \phi \text{,} \qquad \nablai_{ \Ei{i} } \psi = e^{-\lambda F} \nablai_{ \Ei{i} } \phi \text{,} \end{equation}
for any $0 \leq i < n$.
Moreover,
\[ \nablai_{ \Ei{n} } \psi = e^{-\lambda F} \cdot \nablai_{ \Ei{n} } \phi - \lambda e^{-\lambda F} F^\prime \nablai_{ \Ei{n} } f \cdot \phi \text{,} \]
so that, recalling \eqref{eql.carleman_f_ass_strong} and \eqref{eql.nor_est}, we can estimate
\[ e^{-2 \lambda F} \cdot | \nablai_{ \Ei{n} } \phi |^2 \lesssim | \nablai_{ \Ei{n} } \psi |^2 + \lambda^2 ( F^\prime | \nablai_{ \Ei{n} } f | )^2 \cdot \psi^2 \lesssim | \nablai_{ \Ei{n} } \psi |^2 + \lambda^2 f ( F^\prime )^2 \cdot \psi^2 \text{.} \]
Multiplying both sides by $( F^\prime )^{-1} G$ and applying \eqref{eql.F_est} yields
\begin{equation} \label{eql.unconjugate_2} e^{-2 \lambda F} ( F^\prime )^{-1} G \cdot | \nablai_{ \Ei{n} } \phi |^2 \lesssim f G \cdot | \nablai_{ \Ei{n} } \psi |^2 + \lambda^2 f F^\prime G \cdot \psi^2 \text{.} \end{equation}

Finally, combining \eqref{eql.algebraic_integral} with \eqref{eql.unconjugate_1} and \eqref{eql.unconjugate_2}, and letting
\[ \mc{W}_\lambda = e^{-2 \lambda F} ( F^\prime )^{-1} \text{,} \]
we obtain the generalized Carleman inequality
\begin{align}
\label{eql.algebraic_final} \int_{ \mc{D} } \mc{W}_\lambda \cdot | \Box \phi |^2 &\gtrsim \lambda \int_{ \mc{D} } \mc{W}_\lambda \left( G \cdot | \nablai_{ \Ei{n} } \phi |^2 + F^\prime \Psi \cdot \sum_{i = 0}^{n - 1} | \nablai_{ \Ei{i} } \phi |^2 \right) \\
\notag &\qquad + \lambda^3 \int_{ \mc{D} } \mc{W}_\lambda \cdot f (F^\prime)^2 G \cdot \psi^2 \text{.}
\end{align}
To recover the final inequalities \eqref{eq.carleman_f1} and \eqref{eq.carleman_f2} from \eqref{eql.algebraic_final}, we do the following:
\begin{itemize}
\item Replace $F = F_1$ in the case of \eqref{eq.carleman_f1}, or $F = F_2$ for \eqref{eq.carleman_f2}.

\item Similarly, substitute $G$ by either $G_1$ or $G_2$.

\item The weights $\mc{W}_\lambda$ corresponding to $F = F_1$ and $F = F_2$, respectively, are
\[ \mc{W}_{\lambda, 1} = f^{-2 \lambda + 1} e^{2 \lambda f^p} \text{,} \qquad \mc{W}_{\lambda, 2} = f^{q + 1} e^{2 \lambda f^{-q} } \text{.} \]
\end{itemize}
This completes the proofs of Propositions \ref{thm.carleman_f1} and \ref{thm.carleman_f2}.

\begin{remark}
Note that up to and including \eqref{eql.algebraic_final}, the preceding proof made no references to the explicit definitions of $F_1$ and $F_2$.
In particular, the entirety of our proof relied only on the following assumptions:
\begin{itemize}
\item Conditions characterizing $F$ and $G$: \eqref{eql.F_est}, \eqref{eql.F_exp_est}, \eqref{eql.G_est}, and \eqref{eql.Psi_est}.
\footnote{These essentially require that $F$ grows at least as fast as $\log f$, and also that $F$ should not be $\log f$ itself. The latter point accounts for the correction term added to $F_1$.}

\item Conditions \eqref{eql.carleman_f_ass_pseudoconvex}-\eqref{eql.carleman_f_ass_vanishing}, which guarantee that the underlying space-time and $f$ 
are sufficiently close to the special cases considered in Section \ref{sec:psc}.
\end{itemize}
If the above assumptions hold for our setting and for our choice of $F$, then the preceding proof implies that the more 
general Carleman estimate \eqref{eql.algebraic_final} holds.
Such a generalized estimate allows for different unique continuation results for operators $L_g$ as in \eqref{wave}, but with the lower-order terms allowed to have different asymptotic behavior from those considered in Theorems \ref{thm1}-\ref{O1.pm}.
\end{remark}

\section{Proof of Theorems \ref{thm1}-\ref{O1.pm}.} \label{sec:massloss}

\subsection{Preliminaries and Notation.}

We now consider the class of space-times $(\mc{D}, g)$ and operators $L_g=\Box_g+a^\alpha\partial_\alpha+V$ 
addressed in Theorems \ref{thm1}-\ref{O1.pm}.
The results will be proven \emph{together}, by reducing them to Propositions \ref{thm.carleman_f1}, \ref{thm.carleman_f2}.
There are three key steps:

We first consider certain {\it special} conformal rescalings
of the underlying metrics~$g$,
\begin{equation*}
  g\mapsto\bar{g} = \Omega^2 g \text{,}
\end{equation*}
which transform the boundaries at infinity $\mc{I}^+_\e\cup\mc{I}^-_\e$ and $\mc{I}^+_{r_0}\cup\mc{I}^-_{r_0}$ into
{\it complete} double null cones emanating from a  point; we denote these cones by $\ol{\mc{I}}$ for the purpose of this discussion. 
The domains $\mc{D}^\e_\omega$, $\mc{D}_\omega$
are then mapped to  the {\it exterior} domains of these cones; see Figure \ref{fig:minkowski:inversion}.
This is a generalization of the {\it warped inversion} discussed in Section \ref{sec:minkowski:pseudo}.
Now, $L_g$ transforms to a {\it new} wave operator $\bar{L}_{\ol{g}}$ (whose principal symbol is $\Box_{\ol{g}}$)
defined over the manifolds $(\mc{D}, \bar{g})$ such that the solutions $\phi$ 
to $L_g \phi=0$ yield new solutions $\bar{\phi}$ to $\bar{L}_{\bar{g}} \bar{\phi}=0$;
cf.~Section \ref{sec:operators:transform}.

Then, we prove that the inverted metrics $\bar{g}$ and the operators $\bar{L}_{\bar{g}}$  fulfill the requirements
of Propositions \ref{thm.carleman_f1}, \ref{thm.carleman_f2}; cf.~Section \ref{sec:carleman:verify:assumptions}. In particular, we can derive Carleman 
estimates for $\Box_{\bar{g}}$, for functions defined over 
the domains $\mc{D}^\e_{\omega'}$, $\mc{D}_{\omega'}$ (depending on the setting). 

In the last step, we show that these Carleman estimates directly imply that $\bar{\phi}=0$ on 
the domains $\mc{D}^\e_{\omega'}$,  $\mc{D}_{\omega'}$; cf.~Section \ref{sec:uc:arg}. This implies that $\phi$, which solves the original equation $L_g\phi=0$,
vanishes on the same domain. 

In order to give a unified proof of Theorems \ref{thm1}-\ref{O1.pm}, we introduce some uniform language and notation. 
We will refer to the  settings of Theorems \ref{thm1}, \ref{O1} as the ``zero mass'' $\mathbf{(M_0)}$ settings and to Theorems \ref{thm1.pm}, \ref{O1.pm} as the ``positive mass'' $\mathbf{(M_+)}$ settings.
  We collectively refer to the domains $\mc{D}^\e_\omega$, $\mc{D}_\omega$ 
as $\mc{D}_\omega$; our metrics $g$ below will be defined on those domains. 
Recall that $\mc{D}$ is covered by coordinates $ \{u,v,y^1,\dots, y^{n-1}\}$
with $-u,v\in (0,+\infty)$ and $y^1,\dots y^{n-1}$ which cover the sphere $\mathbb{S}^{n-1}$.
We will be denoting $\partial_A:=\partial_{y^A}$. 

\begin{definition}
Let $\uu:=-u+\e$, $\vv:=v+\e$ in the zero mass settings, and $\uu:=-u$, $\vv:=v$ in the positive mass settings.
We also set $r^\ast:=v-u$ for brevity.
\end{definition}

Note that in all settings of Section~\ref{sec:results}, the functions $f_\e, f$ concide with:
\beq \label{eq:f:uv}
f=\frac{1}{\uu\vv}\,.
\eeq

Moreover the metric in all theorems considered is of the general form: 
\begin{align}
\label{gen.form} g &:= \mu \,\ud u^2-4 K\,\ud u\ud v+\nu\, \ud v^2+\sum_{A,B=1}^{n-1}r^2\,\gamma_{AB}\,\ud y^A\ud y^B \\
\notag &\qquad + \sum_{A=1}^{n-1}\,c_{Au}\,\ud y^A\ud u+\sum_{A=1}^{n-1}\,c_{Av}\,\ud y^A\ud v \text{,}
\end{align}
where $r$ is a smooth function that:
\begin{itemize}
\item[$\mathbf{(M_0)}$] is given by $r=r^\ast=v-u$ in the zero mass settings,
\item[$\mathbf{(M_+)}$] obeys the differential inequality \eqref{uvr} in the positive mass setting:
  \begin{equation*}
   \Bigl(1+\mathcal{O}\bigl({\rs}^{-2}\bigr)\Bigr)\ud v-\Bigl(1+\mathcal{O}\bigl({\rs}^{-2}\bigr)\Bigr)\ud u=\Bigl(1+\frac{2m}{r}\Bigr)\ud r+\sum_{A=1}^{n-1}\mathcal{O}\bigl({\rs}^{-1}\bigr)\ud y^A
  \end{equation*}
Moreover, $r$ is required to be both large and bounded above on a fixed level set $r^\ast=M$; see \eqref{uvr0}.
\end{itemize}

The metric coefficients belong to the classes 
\begin{gather}\label{components}
  K = 1 - \frac{2m}{r} \text{,} \qquad \gamma_{AB} = \g_{AB} + \mc{O}^\delta_1 \left( \frac{1}{r^*} \right) \text{,} \\
  c_{Au}, c_{Av} = O_1^\delta \left( \frac{1}{r^*} \right) \text{,} \quad \mu = \mc{O}^\delta_1 \left( \frac{1}{{r^*}^3} \right) \text{,} \quad \nu = \mc{O}^\delta_1 \left( \frac{1}{{r^*}^3} \right) \text{,} \notag
\end{gather}
where $\g_{AB}$ is the {\it round} metric on $\mathcal{S}_{u,v}$ with respect to the coordinates $y^A$,
and $\delta>0$ is a positive constant, and $m$ a smooth function such that
\begin{itemize}
\item[$\mathbf{(M_0)}$] $0<\delta\ll \e$,\footnote{See the statement of Theorems \ref{thm1}, \ref{O1}.}
and $m$ satisfies $m=\mc{O}_1^\delta(\frac{1}{r^*})$. 
\item[$\mathbf{(M_+)}$] $m$ is a smooth {\it positive} function on the spacetime satisfying a uniform lower bound $m \ge m_{\rm min} > 0$ in the domains $\mc{D}$ (see \eqref{ass.m.0}), and $\delta$ is \emph{any} positive constant.
Furthermore, from \eqref{ass.m.1}, $m$ satisfies, for some $\eta>0$,
  \begin{equation*}
    \lvert\partial_u m\rvert, \lvert\partial_v m\rvert\leq C (v-u)^{-2}\,,\qquad \lvert \partial_A m\rvert\leq C (v-u)f^\eta\,.
  \end{equation*}
\end{itemize}
On the level of second derivatives, we always require \eqref{ass.box} for some $\eta>0$,
\begin{equation}\label{eq:ddmr}
  \Box_g \left( \frac{m}{r} \right) = \mathcal{O}^\delta (f^{1+\eta}) \text{.}
\end{equation}

With this notation in place, we can define the conformal inversion of the metrics~$g$. 
The {\it inverted metric} $\ol{g}$ is now defined by:
\beq
\label{conf.inv}
\ol{g}=K^{-1}f^2g.
\eeq

\subsection{Transformation of the Equations to the Inverted Picture} \label{sec:operators:transform}

Let us next see how solutions to a wave equation over $(\mc{D},g)$
correspond to solutions of a {\it new} wave equation over $(\mc{D}, \bar{g})$. 
Consider on $\mathcal{D}$ the equation
\begin{equation}
  \label{eq:a:equation:physical}
  L_g \phi := \Box_g \phi + a^\alpha \partial_\alpha \phi + V \phi = 0 \text{.}
\end{equation}

Recall the conformal Laplacian, 
\[ P_g := \Box_g - \frac{n - 1}{4 n} R_g \text{,} \]
where $R_g$ is the scalar curvature for $g$.
This operator enjoys the following conformal transformation law for all functions $\Omega>0$ and $\phi$:
\begin{equation}
  \label{eq:a:conformal:law}
  P_{\Omega^2 g} ( \Omega^{-\frac{n - 1}{2}} \phi ) = \Omega^{-\frac{n + 3}{2} } P_g \phi \text{.}
\end{equation}
Therefore, letting $\bar{g}=\Omega^2g$, $\bar{\phi} = \Omega^{-\frac{n - 1}{2} } \phi$, we derive:
\begin{equation}
  \Box_g \phi = \Omega^{\frac{n + 3}{2} } \Box_{\bar{g}} \bar{\phi} + \frac{n - 1}{4 n} ( \Omega^{\frac{n + 3}{2} } R_{\bar{g}} - \Omega^{\frac{n - 1}{2} } R_g ) \bar{\phi} \text{.}
\end{equation}

Thus, a direct computation shows that
$ L_g \phi = \Omega^{\frac{n + 3}{2} } \bar{L}_{ \bar{g} } \bar{\phi}$,
where $\bar{L}_{ \bar{g} }$ is the corresponding operator
\begin{align}
\label{eq:Lbar} \bar{L}_{ \bar{g} } &= \Box_{ \bar{g} } + \bar{a}^\alpha \partial_\alpha + \bar{V} \text{,} \qquad \bar{a}^\alpha = \Omega^{-2} a^\alpha \text{,} \\
\notag \bar{V} &= \Omega^{-2} V + \frac{n - 1}{4} \Omega^{-4} a^\alpha \partial_\alpha ( \Omega^{2 } ) + \frac{n - 1}{4 n} ( \Omega^{-2 } R_g - R_{ \bar{g} } ) \text{.}
\end{align}
In particular, it holds that \emph{$L_g \phi \equiv 0$ if and only of $\bar{L}_{ \bar{g} } \bar{\phi} \equiv 0$}.

Finally, recall from the classical Yamabe equation that the difference of scalar curvatures between conformally related metrics is in general given by
\begin{align}
\label{yamabe} \Omega^{-2 } R_g - R_{ \bar{g} } &= \frac{4n}{n - 1} \Omega^{-\frac{n + 3}{2}} \Box_g ( \Omega^{\frac{n - 1}{2}} ) \\
\notag &= 2 n \Omega^{-3 } \Box_g \Omega + n (n - 3) \Omega^{-4} \cdot g ( \nabla \Omega, \nabla \Omega ) \text{.}
\end{align}

\subsection{Verification of the Key Assumptions for the Carleman Estimates in the Inverted Picture} \label{sec:carleman:verify:assumptions}

We shall now prepare the application of the Carleman estimates of Section \ref{sec:carleman} to the wave operators $\Box_{\ol{g}}$ corresponding to 
$\ol{g}$, and functions supported in $\mc{D}_{\omega'}$ that {\it vanish} in a suitable weighted sense 
on the (semi-compactified) boundary $\ol{\mc{I}}$. We recall that this boundary is (in an intrinsic sense) a {\it complete} double null cone,
emanating from a point that corresponds to spatial infinity. 
\footnote{We note that {\it if} we had considered the Penrose compactification instead of \eqref{conf.inv} then the segments $\mc{I}^+_\e,\mc{I}^-_\e$ would have been converted to an {\it incomplete} double null cone; in that setting the boundary spheres $\{\vv=+\infty, \uu=0\}, \{\uu=\infty, \vv=0\}$ would have caused significant difficulties.}
 

In this section, we show that with our choice \eqref{eq:f:uv} for the function $f$, we can find a function $w$ such that the requirements of Propositions \ref{thm.carleman_f1}, \ref{thm.carleman_f2} are fulfilled for the metrics $\ol{g}$.
In particular, we show that the level sets of $f$ are pseudo-convex for $\ol{g}$; in fact,
\[ \bar{\pi} = - \bar{\nabla}^2 f + h^{(w)} \bar{g} \text{,} \]
restricted to the level sets of $f$, is a positive tensor.
We demonstrate this property using an explicit asymptotically orthonormal frame (with respect to $\bar{g}$) adapted to $f$.
In Section \ref{sec:uc:arg}, we shall then apply Propositions~\ref{thm.carleman_f1}, \ref{thm.carleman_f2} to functions that satisfy the appropriate vanishing conditions.

The model spacetimes discussed in Section~\ref{sec:psc} provide a guide for the choice of these functions and the construction of the frames.
We recall at this point that the term $\frac{m}{r}$ in the metric component $-4K\ud u\ud v$ of $g_S$ and the {\it positivity} of $m$ yield {\it extra} positivity for the tensor $\bar{\pi}$ in tangential directions.
The underlying cause of the improved pseudo-convexity is the discrepancy between the functions $r^*:=v-u$ and $r$.
The same type of positivity is present for all our ``positive mass space-times''.
More specifically, it is captured by the components \eqref{other.orth}, \eqref{eq:snablaf:TT} of the Hessian of $f$ below.
This leads us to derive certain lower bounds for the difference $r^*-r$ for all space-times under consideration, 
which as we shall see highlights the importance of the mass term when it is positive rather than zero. 

\subsubsection{Bounds on $\rs-r$.}

Recall that in the zero mass setting we have simply $\rs=r$.
In the positive mass setting, we remark that (\ref{uvr}) together with the 
boundedness below of $r$ on the set $\{r^*=M\}$ imply that for $r^*$ large enough,
\begin{equation}
  \label{easy.bd}
  r \simeq \rs\,.
\end{equation}
Here we have only used the trivial bound $1\leq 1+2m/r \leq 2$, whereas taking the lower order terms into account,
the relation \eqref{uvr} along with the bounds on the derivatives of $m$ imply: 

\begin{lemma}
In the positive mass setting, there exists $C>0$ so that for $r^*$ large enough:
\beq
\label{key.bound}
r^*-r\geq 2\, m_{\rm min}\,\log r -C. 
\eeq
\end{lemma}

\begin{proof}
Recall that by assumption $|r^*-r|\le C_0$ on $\{r^*=M\}$. 
The inequality (\ref{key.bound}) can be proven separately for $v$ sufficiently large and $-u$ sufficiently large; we here show the first case.
We derive from \eqref{uvr}:
\begin{align}
\label{diff.form} \ud r^* &= \ud(r+2m\log r) - \mc{O} \left( \frac{1}{{\rs}^2} \right) \ud v - 2 \partial_v m \log r \,\ud v - 2 \partial_u m \log r\,\ud u \\
\notag &\qquad + \mc{O} \left( \frac{1}{{r^*}^2} \right) \ud u+\sum_{A=1}^{n-1} \left[ \mc{O} \left( \frac{1}{r^*} \right) - 2 (\partial_A m) \log r \right] \ud y^A \text{.}
\end{align}
Now, fixing any values $y^A_*, u_*$, 
 consider the curve $\{y^A=y^A_*, u=u_*\}$.
We then integrate \eqref{diff.form} on this curve (the integrals of the 1-forms containing $\ud u,\ud y^A$ then vanish).
In view of \eqref{easy.bd}, the assumed boundedness \eqref{uvr0} of $r$ on $\{r^*=M\}$, and the bound \eqref{ass.m.1} on $|\partial_v m|$, the bound (\ref{key.bound}) follows.
\end{proof}

Note that with the same proof we also have the corresponding upper bound,
\begin{equation}
  \label{eq:easy.upper}
  \rs-r\leq C_m \log r\,,\qquad \text{for $r^\ast$ large enough,}
\end{equation}
where $C$ is constant that depends on $m$.
Let us also note here for future reference that (in all settings)
\begin{equation}
  \label{r.inv}
  v \leq r^\ast \simeq r \text{,} \qquad -u \lesssim r \text{,} \qquad r^{-1} \lesssim \sqrt{f} \text{,}
\end{equation}
as long as $r^\ast$ is large enough.

\subsubsection{Main properties of $f$ as Carleman weight function on $(\mc{D},\bar{g})$}

The main task in this section is to show that $f$, defined by \eqref{eq:f:uv} as  a function on the \emph{conformally inverted spacetime} $(\mc{D},\ol{g})$, where $\ol{g} := K^{-1}f^2g$, fulfills the requirements of Propositions \ref{thm.carleman_f1}, \ref{thm.carleman_f2}, with suitable choices of $w$, $h^{(w)}$, $\Psi$, and for a suitable frame.

The main proposition of this section is then: 
 
\begin{proposition}
\label{calclns}
Consider a metric $g$ of the form (\ref{gen.form}). 
Let the inverted metric $\ol{g}$ be defined by \eqref{conf.inv} and the function $f$ by \eqref{eq:f:uv}.
Let $w$ be the function
\begin{enumerate}
\item[$\mathbf{(M_0)}$] $w:=-\frac{1}{2}\frac{\e}{r^*}$ in the zero-mass setting,
\item[$\mathbf{(M_+)}$] $w:=-\frac{1}{2}\frac{m_{\min}}{r^*}\log r^*$ in the positive mass setting,
\end{enumerate}
and moreover $h:=w+\frac{1}{2}\Box_{\ol{g}}f-\frac{n-1}{4}$, as in \eqref{eq.carleman_f_psf}.

Then, there exist orthonormal frames $(\tilde{E}_0, \dots, \tilde{E}_n)$ adapted to $f$ (as defined in Section \ref{sec:carleman}) such that the conditions \eqref{eq.carleman_f1_ass_gap}-\eqref{eq.carleman_f1_ass_w} of Proposition \ref{thm.carleman_f1} and \eqref{eq.carleman_f2_ass_gap}-\eqref{eq.carleman_f2_ass_w} of Propositions \ref{thm.carleman_f2} are fulfilled (for $p=\eta/2$ and $q=2/3$, respectively), with
\begin{enumerate}
\item[$\mathbf{(M_0)}$] $\Psi :=\frac{\e}{r^*}$ in the zero mass setting,
\item[$\mathbf{(M_+)}$] $\Psi :=\frac{m_{\min}}{r^*}\log r^*$ in the positive-mass setting,
\end{enumerate}
and with $w$ as above.
Furthermore, the frame elements $\tilde{E}_\alpha$ satisfy
\begin{equation}
\label{calclns.est} f^{-\frac{1}{2}} \tilde{u} \cdot \partial_u,\, f^{-\frac{1}{2}} \tilde{v} \cdot \partial_v,\, r^{-1} f^{-1} X = \sum_{\alpha = 0}^n \Bigl[ 1 + \mc{O}^\delta \Bigl( \frac{1}{r^\ast} \Bigr) \Bigr] \tilde{E}_\alpha \text{.}
\end{equation}
for any vector field $X$ tangent to the $\mc{S}_{u, v}$'s which is unit with respect to $\g$.
 \end{proposition}

\begin{proof}
It is convenient to start our calculations  in the physical metric $g$; we then obtain the bounds for the relevant quantities in the inverted metric $\ol{g}$ via standard formulae for conformal transformations.
 In fact, for convenience we consider 
first a ``mild" conformal rescaling of the physical space-time metric $g$: 
\beq
\label{mild}
\sg := K^{-1}g.
\eeq
 
We begin with deriving bounds for the inverse of the metric $\sg$ and its Christoffel symbols.
{\it Convention:} For the components of $\sg,\sGamma$ below, we use the entries $u, v, A=1,\ldots,n-1$, to refer to the coordinates $u,v, y^1,\dots, y^{n-1}$. 

Given the assumptions on the metric \eqref{gen.form}, the rescaling \eqref{mild}, and the bounds \eqref{eq:dmr}, we derive for the inverse of the metric $\sg$:
\begin{gather}
\sg^{uu} = \mc{O}^\delta \Bigl( \frac{1}{{r^*}^3} \Bigr) \text{,} \quad \sg^{uv} = - \frac{1}{2} + \mc{O}^\delta \Bigl( \frac{1}{r^*} \Bigr) \text{,} \quad \sg^{vv} = \mc{O}^\delta \Bigl( \frac{1}{{r^*}^3} \Bigr) \text{,} \\
\sg^{A u}, \sg^{A v} = \mc{O}^\delta \Bigl( \frac{1}{{r^*}^3} \Bigr) \text{,} \quad \sg^{AB} = r^{-2} \Bigl[ \g^{AB} + \mc{O}^\delta \Bigl( \frac{1}{r^*} \Bigr) \Bigr] \text{.}
\end{gather}

 We also derive the following bounds for certain Christoffel symbols of $\sg$ relevant for the conclusions below:
\begin{subequations}\label{eq:sGamma}
 \begin{gather}
   \sGamma_{uu}^u, \sGamma_{vv}^v, \sGamma_{uv}^u, \sGamma_{uv}^v,\sGamma^v_{uu}, \sGamma^u_{vv} = \mc{O}^\delta \Bigl( \frac{1}{{r^*}^4} \Bigr) \text{,} \\
   \sGamma^u_{u I}, \sGamma^v_{vI}, \sGamma^v_{uI}, \sGamma^u_{vI}=\mc{O}^\delta \Bigl( \frac{1}{{r^*}^2} \Bigr) \text{,} \\
   \sGamma^u_{AB}=\frac{1}{2}r\g_{AB}+\mathcal{O}^\delta(1)\,,\quad
    \sGamma^v_{AB}=-\frac{1}{2}r\g_{AB}+\mathcal{O}^\delta(1)\,.
 \end{gather}
\end{subequations}
In deriving the above formulas we used that
\[ \frac{\partial r}{\partial u} = -1 + \frac{2m}{r} + \mathcal{O}^\delta \Bigl( \frac{1}{ {r^\ast}^2 } \Bigr) \text{,} \qquad \frac{\partial r}{\partial v} = 1 - \frac{2m}{r} + \mathcal{O}^\delta \Bigl( \frac{1}{ {r^\ast}^2 } \Bigr) \text{,} \]
which follows from \eqref{uvr}, as well as the estimates
\begin{gather}\label{eq:dmr}
\partial_u \left(\frac{m}{r}\right) \text{,} \, \partial_v \left( \frac{m}{r} \right) = \mathcal{O}^\delta \left( \frac{1}{{\rs}^2} \right) \text{,} \qquad \partial_A \left( \frac{m}{r} \right) = \mathcal{O}^\delta (f^\eta) \text{,} \\
\notag g^{\alpha\beta} \cdot \partial_\alpha \left( \frac{m}{r} \right) \partial_\beta \left( \frac{m}{r} \right) = \mathcal{O}^\delta (f^{1+\eta}) \text{,}
\end{gather}
which follow from our assumptions on $m$ and $r$.

We now consider the frame $(T, E_1, \dots, E_{n-1},N)$, where $E_1,\ldots,E_{n-1}$ is an orthonormal frame for $r^2\g$ and $T$, $N$ are defined by:
\begin{subequations}\label{eq:TNuv}
\begin{gather}
  T := \frac{1}{2} \sqrt{f} \left( \uu \partial_u + \vv \partial_v \right) = \frac{1}{2} \left( \sqrt{\frac{\uu}{\vv}}\partial_u + \sqrt{\frac{\vv}{\uu}}\partial_v \right) \text{,} \\
  N := \frac{1}{2} \sqrt{f} \left( \uu \partial_u - \vv \partial_v \right) = \frac{1}{2}\left( \sqrt{\frac{\uu}{\vv}}\partial_u - \sqrt{\frac{\vv}{\uu}}\partial_v \right) \text{.}
\end{gather}
\end{subequations}
This frame is asymptotically orthonormal; in fact, from \eqref{r.inv}, we have
\begin{subequations}
\begin{gather}
  \sg(T,T)=-1+\mc{O}^\delta \! \Bigl( \frac{f}{r^*} \Bigr) \text{, }
  \sg(T,N)=\mc{O}^\delta \! \Bigl( \frac{f}{r^*} \Bigr) \text{, }
  \sg(N,N)=1+\mc{O}^\delta \! \Bigl( \frac{f}{r^*} \Bigr) \text{, } \\  
  \sg({E}_i,{E}_j) = \delta_{ij}+\mc{O}^\delta \Bigl( \frac{1}{r^*} \Bigr) \text{,} \qquad i,j=1,\ldots,n-1 \text{,} \\
  \sg({T},{E}_i) = \mc{O}^\delta \left( \frac{\sqrt{f}}{r^*} \right) \text{,} \quad \sg({N},{E}_i)=\mc{O}^\delta \left( \frac{\sqrt{f}}{r^*} \right) \text{,} \qquad i=1,\ldots,n-1 \text{.}
\end{gather}
  \label{frame.orthonormal}
\end{subequations}
Furthermore, note that $T$ and $E_1, \dots E_{n-1}$ are tangent to the level sets of $f$.

Next, we calculate the components of the $\sg$-Hessian of $f$ relative to the above frame.
To keep notation simple, we write by convention $\snabla_{ij}f=(\snabla^2 f)(E_i,E_j)$.
Note we may assume without loss of generality that $E_i=\frac{1}{r}\partial_{y^i}$.
\footnote{At any given point we may choose the coordinates $y^A:A=1,\ldots,n-1$ such that the coordinate vectorfields are orthogonal at that point.}

First,
\begin{equation}\label{eq:snablaf:TT}
  \snabla_{TT}f=\frac{1}{4}\Bigl[\frac{\uu}{\vv}\snabla_{uu}f+2\snabla_{uv}f+\frac{\vv}{\uu}\snabla_{vv}f\Bigr]=\frac{1}{2}f^2+\mc{O}^\delta\Bigl(\frac{f^3}{r^\ast}\Bigr)
\end{equation}
Next,
 \begin{equation}\label{eq:snablaf:ij}
   \snabla_{ij}f=-\frac{1}{2}\frac{f^2}{r}(\vv+\uu)\:\delta_{ij}+\mc{O}^\delta \Bigl(\frac{f^2}{r^\ast}\Bigr)\,.
 \end{equation}
In particular, in view of $\vv+\uu=v-u+2\e$ $\mathbf{(M_0)}$ and \eqref{key.bound} we have
\begin{subequations}
\begin{align}
  \mathbf{(M_0)} & \quad\frac{\vv+\uu}{r}=\frac{v-u+2\e}{r}=1+\frac{2\e}{r}\\
  \mathbf{(M_+)} & \quad\frac{\vv+\uu}{r}\geq1+\frac{2 m_{\rm min}}{r}\log r-\frac{C}{r}
\end{align}
\end{subequations}
and thus for all $i\in \{1,\dots, n-1\}$:
\begin{subequations}
\label{other.orth}
\begin{align}
  \mathbf{(M_0)} & \quad-\snabla_{ii} f\geq\frac{1}{2}f^2+\frac{\e}{r}f^2+\mc{O}^\delta\Bigl(\frac{f^2}{r^\ast}\Bigr)\\
  \mathbf{(M_+)} & \quad-\snabla_{ii} f\geq\frac{1}{2}f^2+\frac{m_{\text{min}}}{r}f^2\log r+\mc{O}^\delta\Bigl(\frac{f^2}{r^\ast}\Bigr)
\end{align}
\end{subequations}
in the zero mass and positive mass settings respectively, for $r$ large enough.
Also, 
\begin{equation}
  \label{normal}
  \snabla_{NN}f=\frac{1}{4}\Bigl( \frac{\uu}{\vv}\nabla_{uu}f+ \frac{\vv}{\uu}\snabla_{vv}f-2\snabla_{vu}f  \Bigr)=\frac{3f^2}{2} + \mc{O}^\delta \left( \frac{f^3}{r^*} \right) \text{.}
\end{equation}
Finally, we calculate the off-diagonal terms in the $\sg$-Hessian of $f$ (recall \eqref{r.inv}):
\begin{equation}
\label{off.diag1}
\snabla_{Ti}f=\frac{1}{2r} \left[ \sqrt{\frac{\vv}{\uu}}\snabla_{uI}f+\sqrt{\frac{\uu}{\vv}}\snabla_{vI}f \right] = \mc{O}^\delta \left( \frac{f^2}{r^\ast} \right) \text{,}
\end{equation}
\begin{equation}
\label{off.diag2}
\snabla_{TN}f=\frac{1}{4}\Bigl[ \frac{\uu}{\vv}\snabla_{uu}f - \frac{\vv}{\uu}\snabla_{vv}f \Bigr] = \mc{O}^\delta \Bigl(\frac{f^3}{r^\ast}\Bigr) \text{.} 
\end{equation}
Following the same calculations as in (\ref{off.diag1}) we derive: 
\beq
\label{off.diag3}
\nabla_{Ni}f=O^\delta\Bigl(\frac{f^2}{r^\ast}\Bigr) \text{.}
\eeq

We have thus derived bounds for all components of $(\snabla)^2f$ with respect to the frame $T, N, E_1,\dots, E_{n-1}$. 
For future reference, we note that the above implies:
\begin{align} \label{eq:Boxgsharpf}
\Box_{g^\sharp} f &= (g^\sharp)^{NN}\nabla^\sharp_{NN}f+2(g^\sharp)^{NT}\nabla_{TN}^\sharp f+2(g^\sharp)^{Ni}\nabla_{Ni}^\sharp f \\ 
\notag &\qquad + (g^\sharp)^{TT}\nabla_{TT}^\sharp f+2(g^\sharp)^{Ti}\nabla_{Ti}^\sharp f+(g^\sharp)^{ij}\nabla_{ij}^\sharp f \\
\notag &= \frac{2-(n-1)}{2}f^2-(n-1)\frac{\vv+\uu-r}{r}f^2+\mathcal{O}^\delta\Bigr(\frac{f^2}{r^\ast}\Bigr) \\
\notag &= -\frac{n-3}{2}f^2+\mathcal{O}^\delta (\Psi f^2) \text{,}
\end{align}
where in the last step we have used that by \eqref{eq:easy.upper}, in the case {$\mathbf{(M_+)}$},
\[\lvert \vv+\uu-r\vert=\lvert\rs-r\rvert\leq C \log r\,.\]

\vspace{0.6pc}

We next consider the inverted metric 
\begin{equation}\label{eq:def:sg}
  \bar{g}=f^2\sg\,,
\end{equation}
and seek to calculate the components of $\bar{\nabla}^2f$ in a frame adapted to $f$.
These components tranform under the conformal change \eqref{eq:def:sg} according to the rule:
\begin{equation}
  \label{eq:ctr}
  \ol{\nabla}^2_{\mu\nu} f=\snabla^2_{\mu\nu} f-2(\partial_\mu\log f)(\partial_\nu f)+ (\sg)^{\alpha\beta}(\partial_\alpha\log f)(\partial_\beta f)\sg_{\mu\nu}
\end{equation}
 
Consider the rescaled frame
\begin{equation} 
  \label{eq:frame:bar} 
  \bar{E}_0 := \bar{T} := f^{-1} T \text{,} \qquad \bar{E}_n := \bar{N} := f^{-1} N \text{,} \qquad \bar{E}_i := f^{-1} E_i \text{.} 
\end{equation}
We apply the Gram-Schmidt process to $\bar{E}_1,\ldots, \bar{E}_{n-1},\bar{T}, \bar{N}$ (in that order) to obtain a frame $\tilde{E}_0:=\tilde{T},\tilde{E}_1, \ldots, \tilde{E}_{n-1},\tilde{E}_n:=\tilde{N}$ which is \emph{exactly} $\ol{g}$-orthonormal, \footnote{We let the indices $\tilde{i}$ correspond to the frame~$(\tilde{E}_i)$; the barred indices $\ol{i}$ refer to the frame~$(\bar{E}_i)$.}
\begin{equation}
  \bar{g}_{\tilde{i}\tilde{j}} := \bar{g}(\tilde{E}_i,\tilde{E}_j) = m_{ij}\,.
\end{equation}
Then, by construction, this frame is adapted to $f$, in the sense of Section \ref{sec:carleman}.
In particular, $\tilde{E}_i f = 0$ for all $0\leq i\leq n-1$, while \eqref{frame.orthonormal} implies:
\begin{subequations}\label{eq:new:frame}
\begin{align}
  \tilde{E}_i-\bar{E}_i &= \sum_{1\leq j \leq i}\mathcal{O}^\delta\Bigl(\frac{1}{\rs}\Bigr)\bar{E}_j\,,\qquad 1\leq i\leq n-1\,,\\
  \tilde{T}-\bar{T} &= \mathcal{O}^\delta\left(\frac{f}{\rs}\right)\bar{T}+\sum_{1\leq j\leq n-1}\mathcal{O}^\delta\left(\frac{\sqrt{f}}{\rs}\right)\bar{E}_j\,,\\
  \tilde{N}-\bar{N} &= \mathcal{O}^\delta\Bigl(\frac{f}{\rs}\Bigr)\bar{N}+\mathcal{O}^\delta\Bigl(\frac{f}{\rs}\Bigr)\bar{T}+\sum_{1\leq j\leq n-1}\mathcal{O}^\delta\left(\frac{\sqrt{f}}{\rs}\right)\bar{E}_j\,.
\end{align}
\end{subequations}

First, we compute and estimate $\bar{\nabla} f$, where $\bar{\nabla}$ refers to the Levi-Civita connection of $\ol{g}$.
Using \eqref{eq:TNuv}, \eqref{eq:frame:bar}, and \eqref{eq:new:frame}, we obtain
\[ \bar{\nabla}_{ \tilde{N} } f = \Bigl( 1 + \mc{O}^\delta \Bigl( \frac{f}{r^\ast} \Bigr) \Bigr) \bar{\nabla}_{ \bar{N} } f = f^\frac{1}{2} \Bigl( 1 + \mc{O}^\delta \Bigl( \frac{f}{r^\ast} \Bigr) \Bigr) \text{.} \]
In particular, this implies part of the estimates in \eqref{eq.carleman_f1_ass_f_strong} and \eqref{eq.carleman_f2_ass_f_strong}:
\begin{equation} \label{eq:grad:N} | f^{-\frac{1}{2}} \bar{\nabla}_{ \tilde{N} } f - 1 | = \mc{O}^\delta \Bigl( \frac{f}{r^\ast} \Bigr) = \mc{O}^\delta ( \Psi ) \text{.} \end{equation}
We also note here that
\begin{equation}
  \label{eq:snabla:gradf}
  (\sg)^{\alpha\beta}(\partial_\alpha f)(\partial_\beta f) = (\sg)^{NN}(\nabla_N f)^2=f^3+\mathcal{O}^\delta\Bigl(\frac{f^4}{\rs}\Bigr)\,.
\end{equation}

We next wish to calculate the components of $\bar{\nabla}^2 f$.
For this, we use the transformation law \eqref{eq:ctr} for the Hessian of $f$ under conformal rescalings \eqref{eq:def:sg}.
Setting
\begin{equation}
h' := w+\frac{1}{2} \text{,}
\end{equation}
and using \eqref{eq:snabla:gradf}, we find, for $1 \le i,j \le n-1$,
\begin{align*}
-\ol{\nabla}_{\ol{i}\ol{j}}f &+ h\rq{} \ol{g}_{\ol{i}\ol{j}} = \\
&=-f^{-2} \Bigl[ \nabla^\sharp_{ij}f + (\sg)^{NN} \nabla_{N}(\log f) (\nabla_{N} f) \sg_{ij} - h' f^2 \Bigl( \delta_{ij} + \mc{O}^\delta \Bigl(\frac{1}{r^*}\Bigr) \Bigr) \Bigr] \\
&=-f^{-2}\Bigl[ \nabla^\sharp_{ij} f + f^2 \delta_{ij} - \frac{1}{2} f^2 \delta_{ij} - w f^2 \delta_{ij} + \mc{O}^\delta \Bigl( \frac{f^2}{r^\ast} \Bigr) \Bigr] \text{,}
\end{align*}
where in applying \eqref{eq:ctr}, we have used that $\ol{E}_i f = 0$.
Now, recalling the definition of $w$ in the statement of the proposition, we see that
\begin{align*}
\mathbf{(M_0)}\ w &= - \frac{\e}{2 r^\ast} = - \frac{\e}{2 r} \text{,} \\
\mathbf{(M_+)}\ w &= - \frac{m_{\min}}{2 r^\ast} \log r^\ast \geq - \frac{3 m_{\min}}{4 r} \log r \text{.}
\end{align*}
where we recalled that $r^\ast = r$ in the $\mathbf{(M_0)}$-case and used \eqref{key.bound} in the $\mathbf{(M_+)}$-case.

Now, combining the above and applying \eqref{other.orth}, we obtain
\begin{align}
\notag -\ol{\nabla}_{\ol{i}\ol{j}}f &+ h\rq{} \ol{g}_{\ol{i}\ol{j}} \\
\mathbf{(M_0)}\ &\geq \frac{1}{2}\frac{\e}{r}\delta_{ij}+\mc{O}^\delta \Bigl(\frac{1}{r^\ast}\Bigr) \text{,} \\
\mathbf{(M_+)}\ &\geq \frac{1}{4}\frac{m_{\rm min}}{r}\log r\,\delta_{ij}+\mc{O}\Bigl(\frac{1}{r^\ast}\Bigr) \text{,} \notag
\end{align}
Similarly, using \eqref{eq:snablaf:ij} and \eqref{eq:easy.upper},
\begin{equation}
  -\ol{\nabla}_{\ol{i}\ol{j}}f+h^\prime\ol{g}_{\ol{i}\ol{j}}=\Bigl[ \frac{\vv+\uu-r}{2 r} + w \Bigr] \delta_{ij}+\mathcal{O}^\delta\Bigl(\frac{1}{r^\ast}\Bigr) = \mc{O}^\delta (\Psi) \text{.}
\end{equation}
Therefore, using \eqref{eq:new:frame}, we obtain, for $\rs$ sufficiently large,
\begin{equation}
\label{eq:hessian:ij} -\ol{\nabla}_{\tilde{i}\tilde{j}}f+h^\prime\ol{g}_{\tilde{i}\tilde{j}} = -\ol{\nabla}_{\ol{ij}}f+h^\prime\ol{g}_{\ol{ij}}+\mathcal{O}^\delta\Bigl(\frac{\Psi}{\rs}\Bigr) \simeq \Psi\ \delta_{ij}+\mathcal{O}^\delta\Bigl(\frac{\Psi}{\rs}\Bigr) \text{.}
\end{equation}

Next, using \eqref{off.diag1}, \eqref{eq:ctr}, and \eqref{eq:new:frame}, we find that for $1\le i\le n-1$,
\begin{align}
\label{eq:hessian:Ti} -\ol{\nabla}_{\bar{T} \bar{i}} f+h' \bar{g}_{\bar{T} \bar{i}} &= -f^{-2} [ \nabla^\sharp_{Ti} f + (\sg)^{NN} \nabla_N (\log f) (\nabla_N f) \sg_{Ti} ] + h' \ol{g}_{\bar{T} \bar{i}} \\
\notag &=-f^{-2} \nabla^\sharp_{Ti} f + \mc{O}^\delta \Bigl(\frac{1}{r^*}\Bigr)=\mc{O}^\delta\Bigl(\frac{1}{r^*}\Bigr) \text{,} \\
\notag -\ol{\nabla}_{\tilde{T} \tilde{i}} f + h^\prime \bar{g}_{\tilde{T} \tilde{i}} &= \Bigl[ 1 + \mc{O}^\delta \Bigl( \frac{1}{r^\ast} \Bigr) \Bigr] \sum_{1 \leq k \leq n-1} [ - \ol{\nabla}_{\bar{T} \bar{k}} f + h^\prime \bar{g}_{\bar{T} \bar{k}} ] \\
\notag &\qquad + \sum_{1 \leq k,l \leq n-1} \mathcal{O}^\delta \left( \frac{\sqrt{f}}{\rs} \right) \ol{\nabla}_{\bar{k} \bar{l}} f = \mathcal{O}^\delta (\Psi) \text{.}
\end{align}
Analogously, from \eqref{off.diag3}, \eqref{eq:ctr}, and \eqref{eq:new:frame}, we also see that
\begin{equation} \label{eq:hessian:Ni} \ol{\nabla}_{\bar{N} \bar{i}} f = f^{-2} \nabla^\sharp_{N i} f + \mc{O}^\delta \Bigl( \frac{1}{r^\ast} \Bigr) = \mc{O}^\delta \Bigl( \frac{1}{r^*} \Bigr) \text{,} \qquad \ol{\nabla}_{\tilde{N}\tilde{i}}f = \mc{O}^\delta ( \Psi ) \text{.} \end{equation}

Applying \eqref{eq:snablaf:TT}, we can compute
\begin{align*}
-\ol{\nabla}_{\ol{T}\ol{T}} f &+ h^\prime \bar{g}_{\ol{T}\ol{T}} \\
&= -f^{-2} \Bigl[\nabla^\sharp_{TT}f+(\sg)^{NN}\nabla_{N}(\log f)(\nabla_N f)\sg_{TT}-h'f^2 \sg_{TT} \Bigr] \\
&= -f^{-2} \Bigl[ \frac{1}{2} f^2 - f^2 + \frac{1}{2} f^2 + w f^2 + \mc{O}^\delta \Bigl( \frac{f^2}{r} \Bigr) \Bigr] = -w + \mc{O}^\delta \Bigl( \frac{1}{r^\ast} \Bigr) \text{,} 
\end{align*}
and henceforth
\begin{equation}
\label{eq:hessian:TT} -\ol{\nabla}_{\tilde{T}\tilde{T}}f+ h'\ol{g}_{\tilde{T}\tilde{T}}=-\ol{\nabla}_{\ol{TT}}f+h^\prime\ol{g}_{\ol{TT}}+\mathcal{O}^\delta\left(\frac{\sqrt{f}w}{\rs}\right)
\simeq \Psi \:\lvert \ol{g}_{\tilde{T}\tilde{T}}\rvert\,.
\end{equation}
Furthermore, from \eqref{off.diag2}, we have
\begin{equation} \label{eq:hessian:NT} \begin{split}
\ol{\nabla}_{\bar{T} \bar{N}} f = f^{-2} [ \nabla^\sharp_{TN} f + (\sg)^{NN} \nabla_N (\log f) (\nabla_N f) \sg_{TN} ] = \mc{O}^\delta\Bigl(\frac{1}{r^*}\Bigr) \text{,} \\
\ol{\nabla}_{\tilde{T}\tilde{N}}f = \mc{O}^\delta ( \Psi ) \text{.}
\end{split} \end{equation}
Combining \eqref{eq:grad:N}, \eqref{eq:hessian:Ni}, and \eqref{eq:hessian:NT}, we obtain the requirements \eqref{eq.carleman_f1_ass_f_strong} and \eqref{eq.carleman_f2_ass_f_strong}.

From \eqref{normal} and \eqref{eq:ctr}, we find
\begin{align*}
\ol{\nabla}_{\ol{N}\ol{N}}f &= f^{-2} [ \nabla^\sharp_{NN} f - 2 f^{-1} (\nabla_N f)^2 + f^{-1} (\sg)^{NN} (\nabla_N f)^2 \sg_{NN}] \\
&= \frac{1}{2} + \mc{O}^\delta \Bigl(\frac{1}{\rs}\Bigr) \text{.}
\end{align*}
Thus, we derive, using \eqref{eq:new:frame},
\begin{equation}
\label{eq:hessian:NN} \Bigl\lvert \ol{\nabla}_{\tilde{N}\tilde{N}} f - \frac{1}{2} \Bigr\rvert \leq \Bigl\lvert \ol{\nabla}_{\bar{N} \bar{N}} f - \frac{1}{2} \Bigr\rvert +\mathcal{O}^\delta\Bigl(\frac{f}{\rs}\Bigr) = \mc{O}^\delta \Bigl( \frac{1}{r^\ast} \Bigr) \text{.}
\end{equation}

Now, combining \eqref{eq:hessian:ij}, \eqref{eq:hessian:TT}, and \eqref{eq:hessian:NN}, we also obtain
\begin{equation}
\label{eq:box} \Box_{\ol{g}} f = \frac{n+1}{2} + \mc{O}^\delta (\Psi) \text{,}
\end{equation}
which along with \eqref{eq:hessian:NN} proves \eqref{eq.carleman_f1_ass_f_weak} and \eqref{eq.carleman_f2_ass_f_weak}.
Moreover, setting 
\[ h := w + \frac{1}{2}\Box_{\ol{g}} f - \frac{n-1}{4} \text{,} \]
and applying \eqref{eq:box}, we find that $-\ol{\nabla}^2f+h\,\ol{g}$ satisfies, as desired, \eqref{eq.carleman_f1_ass_pseudoconvex} and \eqref{eq.carleman_f2_ass_pseudoconvex}.

\vspace{0.6pc}

It remains to check the requirements \eqref{eq.carleman_f1_ass_w} and \eqref{eq.carleman_f2_ass_w} on the function $w$.
Clearly, by definition, $w = \mc{O}^\delta (\Psi)$ (this has already been used earlier).
Moreover, using the Christoffel symbols \eqref{eq:sGamma}, we calculate $\Box_{\ol{g}} w$:
\begin{equation}
  \begin{split}
    \Box_{\bar{g}} w &= f^{-2} [ \Box_\sg w + (n-1) f^{-1} (\sg)^{\alpha\beta} \snabla_\alpha f \snabla_\beta w ] \\
    \mathbf{(M_0)} &= f^{-2}\Bigl[\mc{O}^\delta\Bigl(\frac{1}{{r^*}^3}\Bigr)+\mc{O}^\delta\Bigl(\frac{f}{r^*}\Bigr)\Bigr]=\mc{O}^\delta(f^{-\frac{1}{2}})\,,\\
    \mathbf{(M_+)} &= f^{-2}\Bigl[\mc{O}^\delta\Bigl(\frac{\log r^*}{{r^*}^3}\Bigr)+\mc{O}^\delta\Bigl(\frac{f\log r^*}{r^*}\Bigr)\Bigr]=\mc{O}^\delta(f^{-\frac{1}{2}-\eta})\,,\quad \eta>0\,.
  \end{split}
\end{equation}

This completes the proof that $w$, $\Psi$, $f$, and the frame $\tilde{E}_0, \dots, \tilde{E}_n$ have the required properties \eqref{eq.carleman_f1_ass_gap}-\eqref{eq.carleman_f1_ass_w} of Proposition \ref{thm.carleman_f1} and \eqref{eq.carleman_f2_ass_gap}-\eqref{eq.carleman_f2_ass_w} of Proposition \ref{thm.carleman_f2}, on $(\mathcal{D},\ol{g})$.
Finally, the estimate \eqref{calclns.est} follows by unwinding the definitions of the $\tilde{E}_\alpha$'s---in particular, this is a consequence of \eqref{eq:TNuv}, \eqref{eq:frame:bar}, and \eqref{eq:new:frame}.
\end{proof}

\subsection{Carleman Estimates and Uniqueness} \label{sec:uc:arg}

Recall we have shown in Proposition \ref{calclns} that the inverted metrics $\bar{g}$, along with $f$, $w$, $\Psi$, satisfy on $\mc{D}_{\omega^\prime}$ the requirements \eqref{eq.carleman_f1_ass_gap}-\eqref{eq.carleman_f1_ass_w} of Propositions \ref{thm.carleman_f1} and \eqref{eq.carleman_f2_ass_gap}-\eqref{eq.carleman_f2_ass_w} of Propositon \ref{thm.carleman_f2}, if $\omega^\prime$ is sufficiently small.
To complete the proofs of our main theorems, it remains to apply the Carleman estimates \eqref{eq.carleman_f1} and \eqref{eq.carleman_f2} to show that $\phi$ vanishes on $\mc{D}_{\omega^\prime}$.

To be more specific, we will apply the Carleman estimates as follows:
\begin{itemize}
\item[($\mathbf{P}_d$)] To prove Theorems \ref{thm1} and \ref{thm1.pm} (i.e., wave equations with decaying potentials), we will apply Proposition \ref{thm.carleman_f1}, with $p = 2 \eta$.

\item[($\mathbf{P}_b$)] To prove Theorems \ref{O1} and \ref{O1.pm} (i.e., wave equations with bounded potentials), we will apply Proposition \ref{thm.carleman_f2}, with $q = 2/3$.
\end{itemize}
For conciseness, we discuss only the former case ($\mathbf{P}_d$) in the main discussion.
The latter path ($\mathbf{P}_b$) is proved entirely analogously; differences in the computations will be addressed in remarks and footnotes.

\subsubsection{Bounds for $\bar{L}_{ \bar{g} }$}

The first step is to obtain the asymptotic bounds for the lower-order terms $\bar{a}^\alpha$ and $\bar{V}$ of the operator $\bar{L}_{ \bar{g} }$, defined in \eqref{eq:Lbar} and corresponding to $L_g$ in the inverted setting.
Recall once again that $\bar{g}$ is obtained from $g$ by
\[ \bar{g} = K^{-1} f^2 g = \Omega^{2} g \text{.} \]
Here, it will also be useful to decompose the above into two steps,
\[ g^\sharp = \Omega_1^{2} g \text{,} \qquad \bar{g} = \Omega_2^{2} g^\sharp \text{,} \qquad \Omega = \Omega_1\cdot\Omega_2 \text{,} \]
where
\[ \Omega_1^{-2} = 1 - \frac{2 m}{r} \text{,} \qquad \Omega_2 = f \text{.} \]

Note that from \eqref{eq:Lbar}, we immediately obtain
\begin{equation} \label{Vabar_pre} | \bar{a}^\alpha | \lesssim f^{-2} | a^\alpha | \text{,} \qquad | \Omega^{-2} V | \lesssim f^{-2} | V | \lesssim f^{-1 + \eta} \text{.} \end{equation}
Furthermore, another term in the expansion of $\bar{V}$ can be controlled by
\[ | \Omega^{-4} a^\alpha \partial_\alpha \Omega^{2} | \lesssim f^{-4} | a^\alpha \partial_\alpha ( K^{-1} f^2 ) | \lesssim f^{-3} | a^\alpha \partial_\alpha f | + f^{-2} \left| a^\alpha \partial_\alpha \left( \frac{m}{r} \right) \right| \text{.} \]
Recalling the assumptions (\eqref{thm1.coeff} or \eqref{thm1.pm.coeff}) for the $a^\alpha$'s, bounding $\partial_\alpha (m/r)$ using \eqref{eq:dmr}, and recalling \eqref{r.inv}, we see that
\begin{equation} \label{Vbar_2} | \Omega^{-4} a^\alpha \partial_\alpha \Omega^{2} | \lesssim f^{-3} \cdot f^2 r^{-\frac{1}{2}} + f^{-2} \cdot f^{\frac{1}{2} + \eta} r^{-\frac{3}{2}} \lesssim f^{-1 + \eta} \text{.} \end{equation}

The only remaining term in \eqref{eq:Lbar} is the difference of scalar curvatures.
Recalling the general identity \eqref{yamabe} and the bounds \eqref{eq:dmr} and \eqref{eq:ddmr}, we see that
\begin{equation} \label{Vbar_31} | \Omega_1^{-2} R_{ g^\sharp } - R_g | \lesssim \left| \Box_g \left( \frac{m}{r} \right) \right| + \left| g^{\alpha\beta} \partial_\alpha \left( \frac{m}{r} \right) \partial_\beta \left( \frac{m}{r} \right) \right| \lesssim f^{-1 + \eta} \text{.} \end{equation}
Moreover, by \eqref{yamabe}, along with \eqref{eq:Boxgsharpf} and \eqref{eq:snabla:gradf}, we see that
\begin{align}
\label{Vbar_32} \Omega_2^{-2} R_{ \bar{g} } - R_{ g^\sharp } &= 2 n f^{-3} \Box_g f + n (n - 3) f^{-4} \cdot g ( \nabla f, \nabla f ) \\
\notag &= - n (n - 3) f^{-1} + n (n - 3) f^{-1} + \mc{O} ( f^{-1} \Psi ) = \mc{O} ( f^{-1 + \eta} ) \text{.}
\end{align}
Combining \eqref{Vbar_31} and \eqref{Vbar_32}, we obtain the estimate
\begin{equation} \label{Vbar_3} | \Omega^{-2} R_g - R_{ \bar{g} } | \lesssim \Omega_1^{-2} | \Omega_2^{-2} R_{ \bar{g} } - R_{ g^\sharp } | + | \Omega_1^{-2} R_{ g^\sharp } - R_g | \lesssim f^{-1 + \eta} \text{.} \end{equation}

From \eqref{Vabar_pre} and our assumptions for the $a^\alpha$'s (e.g., \eqref{thm1.coeff}), we have
\begin{equation} \label{abar_est} \bar{a}^u = \mc{O} ( \uu f^{-1} r^{-\frac{1}{2}} ) \text{,} \qquad \bar{a}^v = \mc{O} ( \vv f^{-1} r^{-\frac{1}{2}} ) \text{,} \qquad \bar{a}^A = \mc{O} ( f^{-\frac{3}{2}} r^{-\frac{3}{2}} ) \text{.} \end{equation}
Similarly, for $\bar{V}$, we have from \eqref{Vabar_pre}, \eqref{Vbar_2}, and \eqref{Vbar_3} that
\begin{equation} \label{Vbar_est} \bar{V} = \mc{O} ( f^{-1 + \eta} ) \text{.} \end{equation}
The derivations of bounds for $\bar{a}^\alpha$ and $\bar{V}$ in setting of ($\mathbf{P}_b$) is analogous.

\subsubsection{The Vanishing Condition}

The next step is to apply the Carleman estimate \eqref{eq.carleman_f1}.
Let $\phi$ satisfy $L_g \phi \equiv 0$, and let $\bar{\phi} = \Omega^{-\frac{n - 1}{2}} \phi$, so that $\bar{L}_{ \bar{g} } \bar{\phi} \equiv 0$ (see Section \ref{sec:operators:transform}).
Fix a cutoff function $\chi \in \smooth{0, 1}$ satisfying $\chi (x) = 1$ for $x \le 1 - \kappa$ ($0 < \kappa < 1$) and $\chi (x) = 0$ near $x = 1$, and consider $\tilde{\chi}, \tilde{\phi} \in \smooth{ \mc{D}_{ \omega^\prime } }$, given by
\[ \tilde{\chi} := \chi \circ ( \omega^{\prime -1} f ) \text{,} \qquad \tilde{\phi} := \tilde{\chi} \bar{\phi} \text{,} \] 
We claim that $\tilde{\phi}$ fulfills the vanishing requirement \eqref{eq.carleman_f1_ass_vanishing} in Proposition \ref{thm.carleman_f1}.
\footnote{The proof of superexponential vanishing, \eqref{eq.carleman_f2_ass_vanishing}, in the setting of ($\mathbf{P}_b$) is analogous but slightly easier, since the vanishing assumptions \eqref{key.assn1.O1} and \eqref{key.assn1.O1pm} are stronger.}

Since $\tilde{\chi}$ is bounded, the assumptions \eqref{key.assn1} and \eqref{key.assn1.M} imply that
\[ \int_{ \mc{D}_{\omega'} } ( \tilde{\phi}^2 + | \partial_u \tilde{\phi} |^2 + | \partial_v \tilde{\phi} |^2 + | \partial_I \tilde{\phi} |^2 ) r^N \ud V_g < \infty \text{,} \qquad N \in \mathbb{N} \text{,} \]
where $\ud V_g$ is the volume form with respect to the \emph{physical} metric $g$.
From \eqref{r.inv} and \eqref{calclns.est}, we see that each $\tilde{E}_\alpha$, $0 \leq \alpha \leq n$, can be expanded as
\[ \tilde{E}_\alpha = c^u \partial_u + c^v \partial_v + \sum_{I = 1}^{n-1} c^I \partial_I \text{,} \qquad | c^u | + | c^v | + | c^I | \lesssim r^N \text{,} \]
for some $N \in \mathbb{N}$.
As a result, we have the following integral vanishing condition on $\bar{\phi}$, with respect to inverted metric $\bar{g}$ and its associated volume form $d V_{ \bar{g} }$,
\begin{equation} \label{vanish_inv} \int_{ \mc{D}_{\omega'} } ( \tilde{\phi}^2 + | \bar{\nabla} \tilde{\phi} |_{ \bar{g} }^2 ) r^N \ud V_{ \bar{g} } < \infty \text{,} \qquad N \in \mathbb{N} \text{,} \end{equation}
where $| \bar{\nabla} \bar{\phi} |_{ \bar{g} }$ denotes the frame-based ($\bar{g}$-)tensor norm defined in \eqref{eq.norm_vf}, \eqref{eq.norm_cov}.

Given $\lambda > 0$, we construct an exhaustion of $\mc{D}_{ \omega^\prime }$ (as defined in Section \ref{sec:carleman}) by
\[ \mc{D}_k = \{ \tau_k \le f \le \omega^\prime \text{, } \sigma^-_k \le v + u \le \sigma^+_k \} \text{,} \]
where $\tau_k \searrow 0$, $\sigma^-_k \searrow -\infty$, $\sigma^+_k \nearrow +\infty$ are to be chosen.
Note that
\begin{align*}
\partial \mc{D}_k = \mc{H}_k \cup \mc{H} \cup \Sigma_k^+ \cup \Sigma_k^- \text{,} &\qquad \mc{H}_k = \{ f = \tau_k \text{, } \sigma^-_k < v + u < \sigma^+_k \} \text{,} \\
\mc{G}_k = \{ f = \omega^\prime \text{, } \sigma^-_k < v + u < \sigma^+_k \} \text{,} &\qquad \Sigma_k^\pm = \{ \tau^-_k < f < \tau^+_k \text{, } v + u = \sigma^\pm_k \} \text{,}
\end{align*}
$\mc{H}_k$ and $\mc{G}_k$ are timelike, and the $\Sigma^\pm_k$'s are spacelike.
Also, let $\nu_k$ denote the boundary of $\partial \mc{D}_k$.
Due to the support of $\tilde{\phi}$, we have for any $N \in \mathbb{N}$ that 
\beq
\label{seq1} \lim_{k \nearrow \infty} \int_{ \mc{G}_k } ( | \nu_k |_{ \bar{g} } + | \bar{\nabla}_{ \nu_k } w | ) ( \tilde{\phi}^2 + | \bar{\nabla} \tilde{\phi} |^2_{ \bar{g} } ) f^{-N} e^{N f^p} \ud \bar{V}_{ \mc{H} } = 0 \text{,}
\eeq
where $\ud \bar{V}_{ \mc{G}_k }$ is the volume form of the induced metric $\bar{g} |_{ \mc{G}_k }$.
 
For $\mc{H}_k$, we apply the co-area formula to \eqref{vanish_inv} and derive that for all $N \in \mathbb{N}$,
\begin{equation} \label{N-fin}
\int_0^{\omega'} \int_{ \{ f = \tau \} } ( \tilde{\phi}^2 + | \bar{\nabla} \tilde{\phi} |_{ \bar{g} }^2 ) r^N | \bar{\nabla} f |_{ \bar{g} }^{-1} \ud \bar{V}_{ \{ f = \tau \} } \ud \tau < \infty \text{.}
\end{equation}
Note that because of \eqref{r.inv} and \eqref{eq:snabla:gradf}, we can drop the factor $| \bar{\nabla} f |_{ \bar{g} }^{-1}$ in \eqref{N-fin}.
Moreover, by \eqref{r.inv} and the observation that
\[ f^{-N} e^{N f^p} \lesssim f^{-2 N} \lesssim r^{4 N} \text{,} \]
it follows that there exists $\tau_k \searrow 0$ so that
\[ \lim_{k \nearrow \infty} \int_{ \mc{H}_k } ( \tilde{\phi}^2 + | \bar{\nabla} \tilde{\phi} |^2_{ \bar{g} } ) f^{-N} e^{N f^p} \ud \bar{V}_{ \mc{H}_k } = 0 \text{.} \]
The unit normal $\nu_k$ on $\mc{H}_k$ is simply $| {\rm grad} f |_{ \bar{g} }^{-1} \cdot {\rm grad} f$, so we can also bound $| \nu_k |$ by some power of $r$.
Recalling the definition of $w$ (see Proposition \ref{calclns}), we can also similarly bound $| \bar{\nabla}_{\nu_k} w |$.
Combining the above points, we hence obtain the limit
\beq
\label{seq2} \lim_{k \nearrow \infty} \int_{ \mc{H}_k } ( | \nu_k |_{ \bar{g} } + | \bar{\nabla}_{ \nu_k } w | ) ( \tilde{\phi}^2 + | \bar{\nabla} \tilde{\phi} |^2_{ \bar{g} } ) f^{-N} e^{N f^p} \ud \bar{V}_{ \mc{H}_k } = 0 \text{.}
\eeq

The corresponding limits on $\Sigma_k^\pm$ are handled similarly.
In this case, we have
\footnote{In particular, $\nu_k$ will be close to $f^{-1} ( \partial_u + \partial_v )$, which can be expanded in terms of $\bar{T}$ and $\bar{N}$.}
\[ | \nu_k |_{ \bar{g} } \lesssim r^d \text{,} \qquad | \bar{\nabla}_{ \nu_k } w | \lesssim r^d \text{,} \]
for some $d > 0$.
By another co-area formula argument and \eqref{r.inv}, we obtain
\beq
\label{seq3} \lim_{k \nearrow \infty} \int_{ \Sigma_k^+ \cup \Sigma_k^- } ( | \nu_k |_{ \bar{g} } + | \bar{\nabla}_{ \nu_k } w | ) ( \tilde{\phi}^2 + | \bar{\nabla} \tilde{\phi} |^2_{ \bar{g} } ) f^{-N} e^{N f^p} \ud \bar{V}_{ \Sigma_k^+ \cup \Sigma_k^- } = 0 \text{.}
\eeq
Thus, we conclude that the vanishing condition \eqref{eq.carleman_f1_ass_vanishing} holds.

\subsubsection{The Carleman Estimate}

From now on, we assume all spacetime integrals are with respect to $\ud V_{ \bar{g} }$.
Recalling Proposition \ref{calclns} along with the above, we can now apply the Carleman estimate \eqref{eq.carleman_f1} to $\tilde{\phi}$.
In particular, using \eqref{eq.carleman_f1} and bounding the $\tilde{E}_\alpha \tilde{\phi}$ from below using \eqref{calclns.est}, we see that
\begin{align}
\label{carleman_1} \int_{ \mc{D}_{ \omega^\prime } } f^{-2 \lambda + 1} e^{2 \lambda f^p} | \Box_{ \bar{g} } \tilde{\phi} |^2 &\gtrsim \lambda \int_{ \mc{D}_{ \omega^\prime } } f^{-2 \lambda + 1} e^{2 \lambda f^p} f^{-2} \Psi [ \uu^2 ( \partial_u \tilde{\phi} )^2 + \vv^2 ( \partial_v \tilde{\phi} )^2 ] \\
\notag &\qquad + \lambda \int_{ \mc{D}_{ \omega^\prime } } f^{-2 \lambda + 1} e^{2 \lambda f^p} f^{-3} \cdot r^{-2} \Psi \g^{AB} \partial_A \tilde{\phi} \partial_B \tilde{\phi} \\
\notag &\qquad + \lambda^3 \int_{ \mc{D}_{ \omega^\prime } } f^{-2 \lambda + 1} e^{2 \lambda f^p} \cdot f^{-2 + p} \tilde{\phi}^2 \text{.}
\end{align}

\begin{remark}
On the other hand, in the setting of ($\mathbf{P}_b$), we apply Proposition \ref{thm.carleman_f2} instead, with $q = 2/3$.
The corresponding estimate is
\begin{align}
\label{carleman_2} \int_{ \mc{D}_{ \omega^\prime } } f^\frac{5}{3} e^{2 \lambda f^{-\frac{2}{3}} } | \Box_{ \bar{g} } \tilde{\phi} |^2 &\gtrsim \lambda \int_{ \mc{D}_{ \omega^\prime } } f^\frac{5}{3} e^{2 \lambda f^{-\frac{2}{3}}} \cdot f^{-\frac{8}{3}} \Psi [ \uu^2 ( \partial_u \tilde{\phi} )^2 + \vv^2 ( \partial_v \tilde{\phi} )^2 ] \\
\notag &\qquad + \lambda \int_{ \mc{D}_{ \omega^\prime } } f^\frac{5}{3} e^{2 \lambda f^{-\frac{2}{3}} } \cdot f^{-\frac{11}{3}} r^{-2} \Psi \g^{AB} \partial_A \tilde{\phi} \partial_B \tilde{\phi} \\
\notag &\qquad + \lambda^3 \int_{ \mc{D}_{ \omega^\prime } } f^\frac{5}{3} e^{2 \lambda f^{-\frac{2}{3}} } \cdot f^{-4} \tilde{\phi}^2 \text{.}
\end{align}
\end{remark}

\subsubsection{Unique Continuation}
 
The proof of the vanishing of $\tilde{\phi}$ in $\mc{D}_{\omega'}$ now follows from the Carleman estimates via a standard argument.
First, observe that
\beq
\label{tilde.expand} \Box_{ \bar{g} } \tilde{\phi} = ( \Box_{ \bar{g} } \bar{\phi} ) \tilde{\chi} + \bar{\nabla}_\alpha \bar{\phi} \bar{\nabla}^\alpha \tilde{\chi} + \bar{\phi} \Box_{ \bar{g} }\tilde{\chi} = - \bar{a}^\alpha \partial_\alpha \tilde{\phi} - \bar{V} \tilde{\phi} + \mc{J} \text{.}
\eeq
Here, $\mc{J}$ is a function depending on $\tilde{\phi}$ and $\bar{\nabla} \tilde{\phi}$ and supported in $\{ (1 - \kappa) \omega^\prime \le f < \omega^\prime \}$.
Moreover, from the vanishing properties of $\bar{\phi}$ and $\bar{\nabla} \bar{\phi}$, we can also see that
\[ \int_{ \mc{D}_{ \omega^\prime } } \mc{J}^2 r^N < \infty \text{,} \qquad N \in \mathbb{N} \text{.} \]

Thus, applying \eqref{carleman_1} and \eqref{tilde.expand}, we derive:
\begin{align}
\label{carl.1} &\lambda \int_{ \mc{D}_{ \omega' } } f^{-2 \lambda + 1} e^{2 \lambda f^p} \cdot f^{-2} \Psi \cdot [ \uu^2 ( \partial_u \tilde{\phi} )^2 + \vv^2 ( \partial_v \tilde{\phi} )^2 ] \\
\notag &\qquad + \lambda \int_{ \mc{D}_{ \omega' } } f^{-2 \lambda + 1} e^{2 \lambda f^p} \cdot f^{-3} r^{-2} \Psi \cdot \g^{AB} \partial_A \tilde{\phi} \partial_B \tilde{\phi} \\
\notag &\qquad + \lambda^3 \int_{ \mc{D}_{ \omega' } } f^{-2 \lambda + 1} e^{2 \lambda f^p} \cdot f^{-2 + p} \cdot \tilde{\phi}^2 \\
\notag &\lesssim \int_{ \mc{D}_{ \omega' } } f^{-2 \lambda + 1} e^{2 \lambda f^p} \cdot ( | \bar{a}^\alpha \partial_\alpha \tilde{\phi} |^2 + | \bar{V} \tilde{\phi} |^2 + \mc{J}^2 ) \text{.}
\end{align}
As long as $\lambda$ is sufficiently large, in view of the bounds on $\bar{a}^u$, $\bar{a}^v$, $\bar{a}^I$, and $V$ (see \eqref{abar_est} and \eqref{Vbar_est}), we can absorb the first two terms in the right-hand side of \eqref{carl.1} into the left-hand side.
Moreover, recalling the support of $\mc{J}$, and dropping the first-order terms in the left-hand side, we see that
\begin{align}
\label{carl.2}  & \lambda^3 \int_{ \mc{D}_{ (1 - \kappa) \omega' } } f^{-2 \lambda + 1} e^{2 \lambda f^p} \cdot f^{-2 + p} \cdot \tilde{\phi}^2 
\lesssim \int_{ \mc{D}_{ \omega^\prime} \setminus \mc{D}_{ (1 - \kappa) \omega' } } f^{-2 \lambda + 1} e^{2 \lambda f^p} \cdot \mc{J}^2 \text{.}
\end{align}
Since the weight $f^{-2 \lambda + 1} e^{2 \lambda f^p}$ (for sufficiently small $f$) is bounded from below on $\mc{D}_{ (1 - \kappa) \omega^\prime }$ by its value at $f = (1 - \kappa) \omega^\prime$ and is bounded from above on $\mc{D}_{ \omega^\prime } \setminus \mc{D}_{ (1 - \kappa) \omega^\prime }$ by the same value, then we can drop the 
weight from \eqref{carl.2}:
\begin{align}
\label{carl.3} 
&\qquad  \int_{ \mc{D}_{ (1 - \kappa) \omega' } } f^{-2 + p} \cdot \tilde{\phi}^2 \lesssim \lambda^{-3} \int_{ \mc{D}_{ \omega^\prime} \setminus \mc{D}_{ (1 - \kappa) \omega' } } \mc{J}^2 \text{.}
\end{align}

Letting $\lambda \nearrow \infty$, we see that $\bar{\phi} = \tilde{\phi} \equiv 0$ on $\mc{D}_{ (1 - \kappa) \omega' }$.
Since this holds for all $\kappa>0$, this completes the proofs of the main theorems. 
\footnote{Again, the corresponding argument from \eqref{carleman_2} in the ($\mathbf{P}_b$)-case is entirely analogous.}

\appendix

\section{Kerr Metric in Comoving Coordinates} \label{sec:kerr}

 We show here that the Kerr space-times of mass $m>0$ and any angular momentum $a$ 
 fulfill the requirements of Theorem \ref{thm1.pm}. This necessitates a non-trivial choice of ``co-moving'' coordinates.
Their effect is to adapt to the underlying angular momentum, and yield a form of the Kerr metiric which asymptotes to the Schwarzschild solution 
 one order {\it faster} than in the Boyer-Lindquist coordinates; these coordinates were first introduced by Carter in the context of Kerr de Sitter metrics \cite{carter:lecture}.
\footnote{We should note here that 
 the Kerr metric has been written in double null cooridnates by Israel and Pretorius in \cite{frans:kerr}. However
 this depends on some implicitly defined functions and it does not seem straightforward to check that 
 the requirements of Theorem \ref{thm1.pm} are fulfilled.}
\vspace{0.6pc}

We first  perform this change of coordinates in Minkowski spacetime.
Express the $(3+1)$-dimensional Minkowski space-time
in the usual  $(t_0,r_0,\theta_0,\phi_0)$-coordinates:
\begin{equation}
  g_0=-\ud t_0^2+\ud r_0^2+r_0^2\bigl(\ud\theta_0^2+\sin^2\theta_0\ud\phi_0^2\bigr)
\end{equation}
We now change to a ``comoving'' coordinate system given by the transformation:
\begin{equation}
  (t_0,r_0,\theta_0,\phi_0)\to (t,r,\theta,\phi)\,,
\end{equation}
where for some fixed $a>0$,
\begin{subequations}
  \begin{gather}
  t_0(t)=t\,,\qquad
  \phi_0(\phi)=\phi\,,\\
  r_0^2(\theta,r)=r^2+a^2\sin^2\theta\label{eq:transform:r}\,,\qquad
  r_0(\theta,r) \cos\theta_0(\theta,r)= r\cos\theta\,.
\end{gather}\label{eq:transformation}
\end{subequations}
Let us also denote by
\begin{equation}
  \rho^2=r^2+a^2\cos^2\theta
\end{equation}
The metric expressed in these comoving coordinates  takes the form:
\begin{equation}
  g_0=f_0+h_0\,,
\end{equation}
where
\begin{equation}
 f_0=-\ud t^2+(r^2+a^2)\,\sin^2\theta\,\ud\phi^2, \qquad    h_0=\frac{\rho^2}{(r^2+a^2)}\,\ud r^2+\rho^2\,\ud\theta^2\,.
\end{equation}

Let us now recall the Kerr spacetime. It has two Killing vectorfields $T$, $\Phi$, and the metric also takes the form
\begin{subequations}\label{eq:kerr}
\begin{equation}  
  g=f+h
\end{equation}
where $f$ in Boyer-Lindquist coordinates $(t,r,\theta,\phi)$ is the following metric on the group orbits of $T=\partial_t$, $\Phi=\partial_\phi$,
\begin{equation}
  \label{eq:kerr:f}
  f=\sin^2\theta\,\frac{1}{\rho^2}\,\Bigl(a\,\ud t-(r^2+a^2)\,\ud\phi\Bigr)^2-\frac{\Delta_r}{\rho^2}\Bigl(\ud t-a\sin^2\theta\,\ud\phi\Bigr)^2\,,
\end{equation}
and $h$ is the following metric on the orthogonal surfaces:
\begin{equation}
  \label{eq:kerr:h}
  h=\rho^2\Bigl[\frac{1}{\Delta_r}\,\ud r^2+\ud\theta^2\Bigr]\,.
\end{equation}
\end{subequations}
Here, for positive fixed mass $m>0$,
\begin{equation}
  \Delta_r=\bigl(r^2+a^2\bigr)-2mr\,.
\end{equation}
Thus we observe
\begin{equation}
  g\bigr\rvert_{m=0}=g_0 \,, 
\end{equation}
i.e.~by setting the mass $m=0$ in the Kerr solution we obtain Minkowski space  in corotating coordinates.
In short,  we can express the Kerr metrics  as:
\begin{equation}
  g=g_0+\frac{2mr}{\rho^2}\Bigl(\ud t-a\sin^2\theta\,\ud\phi\Bigr)^2+\frac{2mr\rho^2}{\Delta_r\rvert_{m=0}\Delta_r}\,\ud r^2\,.
\end{equation}
We may view the last two terms as perturbation and express them again in the $(t_0,r_0,\theta_0,\phi_0)$-coordinates. Since
\begin{gather}
  \ud t-a\sin^2\theta\ud \phi=\ud t_0-a\sin^2\theta\,\ud\phi_0 \\
  \rho^2\ud r=\Bigl(r_0 (r-r_0)+r^2+a^2\Bigr)\ud r_0-a^2r_0\sin\theta_0\cos\theta\,\ud\theta_0
\end{gather}
we obtain in reference to the Schwarzschild metric,
\begin{equation}
  \label{eq:schwarzschild}
  g_m=-\Bigl(1-\frac{2m}{r_0}\Bigr)\ud t_0^2+\frac{1}{1-\frac{2m}{r_0}}\ud r_0^2+r_0^2\bigl(\ud\theta_0^2+\sin^2\theta_0\ud\phi_0^2\bigr)
\end{equation}
 that for any Kerr metric $g$, as $1/r\to 0$,
\begin{gather}
  \bigl(g-g_m\bigr)_{t_0 t_0}=\frac{2mr}{\rho^2}-\frac{2m}{r_0}=\mathcal{O}\Bigl(\frac{2m}{r_0}\frac{a^2}{r_0^2}\Bigr)\\
  \bigl(g-g_m\bigr)_{t_0 \phi_0}=-\frac{2mr}{\rho^2}a\sin^2\theta=\mathcal{O}\Bigl(\frac{2m}{r_0}a\Bigr)\\
  \bigl(g-g_m\bigr)_{\phi_0\phi_0}=\frac{2mr}{\rho^2}a^2\sin^4\theta=\mathcal{O}\Bigl(\frac{2m}{r_0}a^2\Bigr)\\
  \bigl(g-g_m\bigr)_{r_0r_0}=\frac{2mr}{\rho^2}\frac{1}{\Delta_r}\frac{\bigl(r_0 (r-r_0)+r^2+a^2\bigr)^2}{(r^2+a^2)}-\frac{2m}{r_0}\frac{1}{1-\frac{2m}{r_0}}=\mathcal{O}\Bigl(\frac{2m}{r_0}\frac{a^2}{r_0^2}\Bigr)\label{eq:g:rr}\\
  \bigl(g-g_m\bigr)_{r_0\theta_0}=-2\frac{2mr}{\rho^2}\frac{1}{\Delta_r}a^2 r_0\sin\theta_0\cos\theta\Bigl(\frac{r_0 (r-r_0)}{r^2+a^2}+1\Bigr)=\mathcal{O}\Bigl(\frac{2m}{r_0}\frac{a^2}{r_0}\Bigr)\\
  \bigr(g-g_m\bigr)_{\theta_0\theta_0}=\frac{2mr}{\rho^2}\frac{1}{\Delta_r}a^4\frac{1}{r^2+a^2}r_0^2\sin^2\theta_0\cos^2\theta=\mathcal{O}\Bigl(\frac{2m}{r_0}\frac{a^4}{r_0^2}\Bigr)\,.
\end{gather}
Note that to show \eqref{eq:g:rr} we have used \eqref{eq:transform:r} to infer
\begin{equation}
  r_0-r=\frac{a^2\sin^2\theta}{r_0+r}\,,\qquad r_0(r_0-r)=\mathcal{O}\bigl(a^2\bigr)\,.
\end{equation}

In summary, we have calculated
that with respect to the null coordinates of the underlying Schwarzschild metric $g_m$, $u_0=(t_0-r_0^\ast)/2$, $v_0=(t_0+r_0^\ast)/2$, where $\ud \rs_0/\ud r_0=(1-\frac{2m}{r_0})^{-1}$, the Kerr metrics are of the form
\begin{multline}
  g=\mathcal{O}\bigl(\frac{1}{r_0^3}\bigr)\ud u_0^2-4\Bigl(1-\frac{2m}{r_0}+\mathcal{O}\bigl(\frac{1}{r_0^3}\bigr)\Bigr)\ud u_0\ud v_0+\mathcal{O}\bigl(\frac{1}{r_0^3}\bigr)\ud v_0^2\\
  +\sum_{A=1}^2\mathcal{O}\bigl(\frac{1}{r_0}\bigr)\ud u_0\ud y^A+\sum_{A=1}^2\mathcal{O}\bigl(\frac{1}{r_0}\bigr)\ud v_0\ud y^A+r_0^2\sum_{A,B=1}^2\Bigl(\g_{AB}+\mathcal{O}\bigl(\frac{1}{r_0^3}\bigr)\Bigr)\ud y^A\ud y^B
\end{multline}
Therefore the Kerr metrics $g$ in the coordinates $(u_0,v_0,y^1,y^2)$ take the form \eqref{gen.form}, and fall under the 
assumptions of theorems \ref{thm1.pm}, \ref{O1.pm}.


\providecommand{\bysame}{\leavevmode\hbox to3em{\hrulefill}\thinspace}
\providecommand{\MR}{\relax\ifhmode\unskip\space\fi MR }
\providecommand{\MRhref}[2]{%
  \href{http://www.ams.org/mathscinet-getitem?mr=#1}{#2}
}
\providecommand{\href}[2]{#2}

\end{document}